\documentclass{article}
\usepackage[compress,sort]{natbib}
\usepackage[utf8]{inputenc}
\usepackage{float}
\usepackage{mathtools}
\usepackage{bm,booktabs,rotating}
\usepackage{amsmath,amsfonts}
\usepackage{amssymb}
\usepackage{graphicx}
\usepackage[unicode=true,pdfusetitle,
 bookmarks=true,bookmarksnumbered=false,bookmarksopen=false,
 breaklinks=false,pdfborder={0 0 1},backref=false,colorlinks=false]
 {hyperref}
\usepackage{multicol,soul}
\usepackage{multirow}
\usepackage{color}
\usepackage{setspace}
\usepackage{algorithm}
\usepackage{algpseudocode}
\usepackage{etoolbox}
\AtBeginEnvironment{algorithmic}{\small}
\algblock{ParFor}{EndParFor}
\algnewcommand\algorithmicparfor{\textbf{parfor}}
\algnewcommand\algorithmicpardo{\textbf{do}}
\algnewcommand\algorithmicendparfor{\textbf{end\ parfor}}
\algrenewtext{ParFor}[1]{\algorithmicparfor\ #1\ \algorithmicpardo}
\algrenewtext{EndParFor}{\algorithmicendparfor}

\usepackage{booktabs}
\usepackage{makecell}

\usepackage{mathtools, nccmath}

\usepackage{subcaption}

\algblock{Input}{EndInput}
\algnotext{EndInput}
\algblock{Output}{EndOutput}
\algnotext{EndOutput}

\floatstyle{ruled}
\newfloat{algorithm}{tbp}{loa}
\providecommand{\algorithmname}{Algorithm}
\floatname{algorithm}{\protect\algorithmname}

\usepackage{algorithm}
\usepackage{algorithmicx}
\usepackage{algpseudocode}
\usepackage{enumitem}

\newcommand{\algrule}[1][.2pt]{\par\vskip.5\baselineskip\hrule height #1\par\vskip.5\baselineskip}

\usepackage{tikz}
\usetikzlibrary{decorations.pathreplacing,calc}
\newcommand{\tikzmark}[1]{\tikz[overlay,remember picture] \node (#1) {};}

\makeatletter
    \algrenewcommand\alglinenumber[1]{\tikzmark{\arabic{ALG@line}}\tiny#1:}
\makeatother

\newtheorem{corollary}{Corollary}

\newtheorem{theorem}{Theorem}\newenvironment{proof}{\sl \noindent Proof: \rm}{$\Box$}
\newcommand{\ignore}[1]{}

\newtheorem{remark}{Remark}

\usepackage{bm}
\newcommand{\M}[2][]{\bm{#1{\mathbf{\MakeUppercase{#2}}}}}              
\newcommand{\Me}[3][]{\bm{#1{\mathbf{\MakeUppercase{#2}}}}({#3})}               
\newcommand{\Mb}[3][]{\bm{#1{\mathbf{\MakeUppercase{#2}}}}_{#3}}        
\newcommand{\R}{\mathbb{R}}
\newcommand{\argmin}{\operatorname*{arg\,min}}
\newcommand{\proj}{\operatorname{proj}}
\newcommand{\Wfun}{\operatorname{W}_0}
\usepackage{xcolor}

\title{Minimizing the Arithmetic and Communication Complexity of Jacobi's Method for Eigenvalues and Singular Values: \\ Part Two -- Parallel Algorithms}
\author{James Demmel\footnote{Department of EECS (Computer Science Division) and Department of Mathematics, University of California Berkeley} \and Hengrui Luo\footnote{Department of Statistics, Rice University}
\and Ryan Schneider\footnote{Department of Mathematics, University of California Berkeley}
\and Yifu Wang\footnote{Committee on Computational and Applied Mathematics, University of Chicago}}

\date{}

\begin{document}
\maketitle

\begin{abstract}

This paper presents several parallel versions of Jacobi's method for the symmetric eigenvalue problem and the SVD. A continuation of [Demmel, Luo, Schneider, \& Wang 2025], we develop parallel Jacobi algorithms whose arithmetic cost is optimal and whose bandwidth or latency can match the corresponding lower bounds of parallel matrix multiplication. Our focus is a standard distributed-memory setting with variable processor layouts, including both 2D and 2.5D processor grids. In the 2D case, we demonstrate that a standard implementation of parallel Jacobi achieves a perfect speedup in arithmetic cost -- i.e., complexity $O(n^3/P)$ when done with $P$ processors -- while hitting the 2D matrix-multiplication lower bound for bandwidth and (nearly) the lower bound for latency. By employing a 2.5D processor grid and leveraging 2.5D matrix multiplication, equivalently by increasing the memory per processor, we demonstrate that  parallel Jacobi can achieve even lower bandwidth/latency, though we also prove that these costs cannot simultaneously match the best-known bounds for parallel matrix multiplication in any Jacobi algorithm. Finally, we extend our results to one-sided Jacobi SVD. 

\end{abstract}

\section{Introduction}
Jacobi’s method is a classical algorithm for solving the symmetric eigenvalue problem \cite{Jacobi1846}. Given a dense symmetric matrix $\M{A} \in \mathbb{R}^{n \times n}$, Jacobi’s method applies a sequence of orthogonal similarity transformations that gradually annihilate its off-diagonal entries. This yields an approximately diagonal matrix $\M{D}$ and an orthogonal matrix $\M{Q}$ such that $\M{A} = \M{Q}\M{D}\M{Q}^T$, where the diagonal entries of $\M{D}$ approximate the eigenvalues of $\M{A}$ and the columns of $\M{Q}$ approximate its eigenvectors. Naturally, Jacobi’s method can also be used to compute the singular value decomposition (SVD) of an arbitrary matrix $\M{G} \in {\mathbb R}^{m \times n}$, as the right singular vectors and squared singular values of $\M{G}$ can be obtained by diagonalizing the symmetric Gram matrix $\M{G}^T \M{G}$.
The algorithm similarly extends to Hermitian matrices, though we focus on the real setting here for convenience.


Jacobi's method has remained relevant despite its age, largely because it enjoys two key numerical advantages over other algorithms for the symmetric eigenvalue problem. First, in comparison to 
methods that tridiagonalize $\M{A}$ (or bidiagonalize $\M{G}$ in the SVD setting) \cite{DHS89,Cuppen_dnc,sobczyk2024, dhillon2003orthogonal}, Jacobi's method attains tighter (relative) accuracy bounds for eigenvalues/singular values and eigenvectors/singular vectors \cite{athi2016real,shiri2019fpga,DemmelVeselic92}.  In addition, it is easily parallelizable and therefore well-suited for modern computer architectures -- see for example the many papers that have previously considered parallel Jacobi in detail \cite{Sameh_1971,Berry_Sameh_1989,Shroff_Schreiber_1989}.

\indent Classical Jacobi’s method works at the scalar level, annihilating one off-diagonal pair of entries $\M{A}(i,j)$ and $\M{A}(j,i)$ at a time via a $2 \times 2$ Givens rotation. Each such rotation preserves eigenvalues while decreasing the squared Frobenius norm of the off-diagonal part of $\M{A}$ by $2|\M{A}(i,j)|^2$. Once this off-diagonal ``norm" is sufficiently small, the algorithm terminates with an approximate diagonalization. To improve the efficiency of Jacobi's method, the algorithm can be \textit{blocked}, in which case two $b\times b$ off-diagonal blocks of $\M{A}$ are annihilated at each step instead of a pair of entries, which involves 
diagonalizing a $2b \times 2b$ submatrix (see Section~\ref{section: preliminaries} and \cite{arxiv_manuscript}). It is this blocked version of Jacobi's method that we consider parallelizing. 
Central to this effort is the following observation:\ as long as two off-diagonal blocks are in disjoint block rows and columns of $\M{A}$, then they can be annihilated in parallel. In a standard, distributed memory setting, each block of $\M{A}$ is owned by an individual processor and therefore has size, at most, that of the processor's fast memory (though we will later relax this assumption in Section~\ref{section: 2.5D_parallel}).

\indent As a continuation of \cite{arxiv_manuscript}, we consider parallelizing block Jacobi in a way that minimizes the associated arithmetic/communication costs. 
Since in the parallel setting processors execute concurrently, the runtime of a parallel algorithm is not determined by the total amount of work/communication required but rather the maximum amount performed by any processor along its \textit{critical path}, which refers to the longest set of serial operations within the algorithm. Accordingly, we define our cost measures (arithmetic, bandwidth, and latency) as follows, where $P$ is the number of processors and $n$ is the size of the input matrix. The notation used here follows the standard $\alpha$-$\beta$-$\gamma$ model for communication-avoiding numerical linear algebra \cite{hockney1994communication,solomonik2017trade}.
\begin{enumerate}
\item \emph{Latency cost} $\alpha \cdot S$, proportional to the number of messages $S=S(n,P)$ sent between processors along the critical path, where $\alpha$ is the per-message startup time.
\item \emph{Bandwidth cost} $\beta \cdot W$, proportional to the maximum number of words $W=W(n,P)$ communicated by any processor along the critical path, where $\beta$ is the amount of time required to move a single word. 
\item \emph{Arithmetic cost} $\gamma \cdot F$, proportional to the maximum number of floating-point operations $F=F(n,P)$ performed by any processor along the critical path, where $\gamma$ is the time per flop for data in fast memory.
\end{enumerate}

\indent In this model, latency and bandwidth comprise the total communication cost of an algorithm. For simplicity, we assume that (1) communication and computation do not overlap on each processor and (2) that there is no communication contention across the processor grid. Under these assumptions, the predicted runtime of a parallel algorithm is represented by 
\begin{equation}
T = \alpha\cdot S \;+\; \beta\cdot W \;+\; \gamma\cdot F.\label{eq:total cost model}
\end{equation}

\indent For Jacobi's method, optimal complexity is measured relative to the arithmetic, bandwidth, and latency costs of parallel matrix multiplication, specifically done in a distributed-memory fashion over the same processor grid. This is motivated by the fact that the symmetric eigenvalue problem is at least as difficult as matrix multiplication (see \cite[Appendix A]{arxiv_manuscript}). Table~\ref{tab: complexity_bounds} lists the complexity bounds for parallel Jacobi proved in this paper and, for comparison, the best-known bounds for other symmetric eigensolvers.  Here, we will focus on 2D and 2.5D processor layouts (see Section~\ref{section: preliminaries}); for our purposes, optimal complexity is therefore measured relative to the first two rows of table (a).

\begin{remark} {\normalfont 
    Table~\ref{tab: complexity_bounds} does not include the algorithm of Yau and Lu \cite{Yau_Lu}, despite the fact that it has been parallelized in the past \cite{tisseur1997parallel} and potentially offers even better asymptotic complexity than its alternatives. In particular, Yau and Lu show that their algorithm reduces to logarithmically many matrix multiplications, which suggests that -- in parallel -- it may be capable of achieving the optimal bandwidth and latency costs of 2.5D matrix multiplication simultaneously (which we show Jacobi cannot do in Section~\ref{section: lower_bound}). To the best of our knowledge, this has not yet been established rigorously in the literature, so we leave it as an open problem here. Similarly, we mark the bandwidth/latency costs of QDWH-eig  \cite{Nakatsukasa_Higham_2013} as open since, while we expect it to attain the bounds of 2D parallel matrix multiplication, this has not been shown in full detail.}
\end{remark}

\begin{table}[t]
    \centering
    \renewcommand{\arraystretch}{1.6}
    \renewcommand\multirowsetup{\centering}

    \setlength{\tabcolsep}{3pt}

    \begin{subtable}{\textwidth}
        \centering
        \makebox[\textwidth]{
        \begin{tabular}{
            >{\centering\arraybackslash}p{2cm}   
            >{\centering\arraybackslash}p{1.5cm} 
            >{\centering\arraybackslash}p{2.3cm} 
            >{\centering\arraybackslash}p{2.3cm} 
            >{\centering\arraybackslash}p{2.4cm} 
            >{\centering\arraybackslash}p{2.5cm}   
        }
            \toprule
             \textbf{Algorithm} & \textbf{Source} & \textbf{Arithmetic} & \textbf{Bandwidth} & \textbf{Latency} & \textbf{Notes}  \\
            \midrule
            2D & \cite{Cannon69,vdGW97} & $O(n^3/P)$ & $O(n^2/\sqrt{P})$ & $O(\sqrt{P})$ & \\
            2.5D & \cite{SD11} & $O(n^3/P)$ & $O(n^2/\sqrt{cP})$ & {\small $O(\sqrt{P/c^3}+\log c)$} & $1\leq c\leq P^{1/3}$\\
            3D & \cite{ABGJP95} & $O(n^3/P)$ & $O(n^2/P^{2/3})$ & $O(\log P)$ &  \\
            \bottomrule
        \end{tabular}
        }
        \caption{Complexity upper bounds for parallel matrix multiplication}
        \label{tab: complexity_bounds_matmul}
    \end{subtable}

    \vspace{1.5em}

    \begin{subtable}{\textwidth}
        \centering
        \makebox[\textwidth]{
        \begin{tabular}{
            >{\centering\arraybackslash}p{1.8cm} 
            >{\centering\arraybackslash}p{2cm}   
            >{\centering\arraybackslash}p{1.8cm} 
            >{\centering\arraybackslash}p{2.3cm} 
            >{\centering\arraybackslash}p{2.3cm} 
            >{\centering\arraybackslash}p{2.4cm} 
            >{\centering\arraybackslash}p{2.6cm}   
        }
            \toprule
            \textbf{Type} & \textbf{Algorithm} & \textbf{Source} & \textbf{Arithmetic} & \textbf{Bandwidth} & \textbf{Latency} & \textbf{Notes}  \\
            \midrule

            \multirow{2}{*}{\makecell{Tridiagonal \\ Reduction}} & 2D & \cite{Lang_parallel,MRRR_parallel,dnc_parallel,BDK13-TR} & $O(n^3/P)$ & $O(n^2/\sqrt{P})$& $O(\sqrt{P})$ \\
            & 2.5D & \cite{2.5D_QR_Band_Reduction} & $O(n^3/P)$ & $O(n^2/\sqrt{cP})$ & $O(\sqrt{cP}\log^2P)$ & \makecell{$1 \leq c \leq P^{1/3}$ \\ Values only}  \\
            \midrule

            \multirow{2}{*}{\makecell{Divide-and-\\ Conquer}} & QDWH-eig & \cite{Nakatsukasa_Higham_2013,Sukkari_etal_2019} & $O(n^3/P)$  & Open & Open\\
            & Randomized & \cite{BDD11,Shah_Hermitian,definite_dnc} & $O(n^3/P)$ & $O(\frac{n^2}{\sqrt{P}}\log P)$ & $O(\sqrt{P}\log^2 P)$ &  \\
            \midrule

            \multirow{2}{*}{Jacobi} & 2D & Section~\ref{section: 2D_parallel} & $O(n^3/P)$ & $O(n^2/\sqrt{P})$ & $O(\sqrt{P}\log P)$ & \\
            & {2.5D} & Section~\ref{section: 2.5D_parallel} & $O(n^3/P)$ & $O(n^2/\sqrt{cP})$ & $O(\sqrt{cP}\log(P/c))$ & $1\leq c\leq P^{1/5}$ 
            \\ 


            \bottomrule
        \end{tabular}
        }
        \caption{Complexity upper bounds for various parallel symmetric eigensolvers.}
        \label{tab: complexity_bounds_eig_svd}
    \end{subtable}

    \caption{Complexity bounds -- e.g., critical-path-specific flop, bandwidth, and latency costs -- for (parallel) matrix computations and state-of-the-art algorithms.
    }
    \label{tab: complexity_bounds}
\end{table}

\noindent The remainder of the paper is organized as follows:
\begin{enumerate}
    \item Section~\ref{section: preliminaries} covers the basics of blocked Jacobi's method and introduces the processor layouts considered in this paper (i.e., 2D and 2.5D grids).
    \item Section~\ref{section: 2D_parallel} presents a detailed analysis (with pseudocode) of a parallel version of Jacobi's method on a 2D processor grid. We show that this version of parallel Jacobi achieves arithmetic, bandwidth, and latency costs of $O(n^3/P)$, $O(n^2/\sqrt{P})$, and $O(\sqrt{P}\log P)$, respectively. Hence, its arithmetic and bandwidth costs are optimal, while its latency differs from the 2D matrix-multiplication bound by a factor of $\log P$. 
    We note that the 2D parallel Jacobi algorithm considered in this section has multiple variants and has been analyzed many times in the literature \cite{Shroff_Schreiber_1989,Drmac2009,Hari_2015,VL_Jacobi}.
    \item Section~\ref{section: 2.5D_parallel} extends the 2D approach to a 2.5D processor grid using 2.5D matrix multiplication \cite{SD11}. In this setting, each processor has more than the minimum $\Omega(n^2/P)$ memory, specifically larger by a multiplicative \textit{replication factor} $c$ -- i.e., we replicate data across layers of the processor grid to reduce communication, which requires a minimum memory of size $\Omega(cn^2/P)$ on each processor. We bound the complexity of the resulting 2.5D version of parallel Jacobi and exhibit two optimal  configurations:\ one that is latency optimal and another that is bandwidth optimal.
    \item Section~\ref{section: lower_bound} establishes a lower bound of $\Omega(n^2)$ for the product of bandwidth and latency in any parallel Jacobi algorithm, which we note is attained (up to log factors) by every version of parallel Jacobi's method considered in this paper. This lower bound implies that, when $c > 1$, no parallel Jacobi algorithm can simultaneously hit the best-known bandwidth/latency costs of parallel matrix multiplication. 
    \item Finally, Section~\ref{section: parallel_jacobi_svd} extends our analysis to parallel versions of one-sided Jacobi SVD.
\end{enumerate}

\noindent We close the paper with a summary of our results and a brief discussion of remaining open problems.

\subsection{Notation and Conventions}
Throughout, matrices are represented with bold letters and $\mathbf{A} \in {\mathbb R}^{n \times n}$ is assumed to be symmetric. $|| \cdot ||_2$ is the spectral norm and $\mathbf{A}^{\dagger}$ is the Moore-Penrose pseudoinverse. Complexity bounds are stated in standard big-O notation:\ given $f(n)$ and $g(n)$ two positive functions of $n$, we have $f(n) = O(g(n))$ if $f(n) \leq Cg(n)$ for all $n$ sufficiently large, where $C > 0$ is a constant. Similarly, $f(n) = \Omega(g(n))$ if $C'g(n) \leq f(n)$, again for $n$ sufficiently large and $C' > 0$. Finally, $f(n) = \Theta(g(n))$ if $f(n) = O(g(n))$ and $f(n) = \Omega(g(n))$.

\section{Preliminaries}\label{section: preliminaries}

This section contains the background information necessary for the rest of the paper; we first summarize the blocked version of Jacobi's method to be parallelized before discussing the processor layouts used in the subsequent sections. 

\subsection{Blocking Jacobi's Method}\label{section: block_jacobi}
The simplest way to improve the efficiency of Jacobi's method -- in either the serial or parallel setting -- is to consider blocking the algorithm. A so-called \textit{blocked} version of Jacobi's method operates on submatrices $\M{A}_{IJ}$ defined as follows:
\begin{equation}\label{eqn: A_subblock}
    \Mb{A}{IJ}=\Me{A}{(I{-}1)b{+}1:Ib,(J{-}1)b{+}1:Jb},
\end{equation}
where $b$ is the block size and $1 \leq I,J \leq n/b$ are integers, assuming $b$ divides $n$. In this setting, to annihilate the off-diagonal block $\M{A}_{IJ}$, Jacobi's method diagonalizes the $2b \times 2b$ submatrix 
\begin{equation}
    \hat{\M{A}} = \M{A}([I,J],[I,J]) = \begin{pmatrix} \M{A}_{II} & \M{A}_{IJ} \\ \M{A}_{JI} & \M{A}_{JJ} \end{pmatrix}
\end{equation} and applies the resulting orthogonal eigenvector matrix $\hat{\M{Q}}$ to block columns/rows $I$ and $J$. Pseudocode for this procedure is presented in  Algorithm~\ref{alg: block Jacobi}. 

\begin{algorithm}[h]
\caption{Block Jacobi for the Symmetric Eigenvalue Problem}
\label{alg: block Jacobi} \begin{algorithmic}[1] \Require $\M{A} \in {\mathbb R}^{n \times n}$ is symmetric and partitioned as follows (for block size $b$ that divides $n$):
\[
\Mb{A}{IJ}=\Me{A}{(I{-}1)b{+}1:Ib,(J{-}1)b{+}1:Jb} \; \; \; \; \text{for} \; \; \; \; 1 \leq I,J \leq n/b.
\]
\Ensure On output, $\M{A}$ is an (approximately) diagonal matrix $\M{D}$ containing the eigenvalues of $\M{A}$. If eigenvectors are requested, the additional output $\M{Q}$ is an orthogonal matrix of (approximate) eigenvectors satisfying $\M{A} = \M{Q}\M{D}\M{Q}^{T}$. 
\algrule
\Function{$[\M{Q},\M{A}]=$
Block\_Jacobi}{$\M{A}$} \State $\M{Q}=\M{I}$
\Repeat \For{all off-diagonal blocks $(I,J)$ of $\M{A}$, $I<J$,
in some order,} \State $\M{\hat{A}}=\M{A}([I,J],[I,J])$ \Comment{$\M{\hat{A}}=$ $2b \times 2b$ submatrix of $\M{A}$ in
block rows/columns $I$ and $J$} \If{$\M{\hat{A}}$ is far
enough from diagonal} \State Let $\M{\hat{A}}=\M{\hat{Q}}\M{\hat{D}}\M{\hat{Q}}^{T}$
be the eigen decomposition of $\M{\hat{A}}$ 
\State Multiply block rows $I$ and $J$ of $\M{A}$ by $\M{\hat{Q}}^{T}$
\State Multiply block columns $I$ and $J$ of $\M{A}$ by $\M{\hat{Q}}$
\State Multiply block columns $I$ and
$J$ of $\M{Q}$ by $\M{\hat{Q}}$ \EndIf \EndFor \Until{all off-diagonal
entries of $\M{A}$ are small enough} \EndFunction \end{algorithmic} 
\end{algorithm}
The work performed by block Jacobi can be naturally partitioned into \textit{sweeps}, each of which corresponds to one pass over all off-diagonal blocks of $\M{A}$ (equivalently, the outer loop of Algorithm~\ref{alg: block Jacobi}). The details of each sweep are flexible; see~\cite[Section~3]{arxiv_manuscript} for a discussion of possible block orderings (to use in line 4) and criteria for applying a block rotation (in line 6). Progress toward diagonal form is usually measured by the off-diagonal Frobenius ``norm," which we write as
\begin{equation}\label{eqn: Omega_A}
   \Xi(\M{A})
   \;\coloneqq\;
   \sqrt{\sum_{i\neq j} \bigl|\M{A}(i,j)\bigr|^{2}}.
\end{equation}
Convergence with respect to~\eqref{eqn: Omega_A} means that $\Xi(\M{A}) \rightarrow 0$ as the number of sweeps grows. \\
\indent Intuitively, the runtime of block Jacobi is proportional to the runtime of a single sweep. That is, in the $\alpha$-$\beta$-$\gamma$ model \eqref{eq:total cost model}, the total cost of Algorithm~\ref{alg: block Jacobi} is
\begin{align}
   T
   \;=\;
   N \cdot
   T_{\mathrm{sweep}}
   \;=\;
   N
   \bigl(
      \alpha \cdot S_{\mathrm{sweep}}
      + \beta \cdot W_{\mathrm{sweep}}
      + \gamma \cdot F_{\mathrm{sweep}}
   \bigr),\label{eq:sweep cost}
\end{align}
where $N$ is the number of sweeps performed and
$S_{\mathrm{sweep}}, W_{\mathrm{sweep}},$ and $F_{\mathrm{sweep}}$ denote, respectively, the number of messages, words moved, and floating-point operations incurred by an individual sweep. \\
\indent In practice, $N$ is the number of sweeps required for $\Xi(\M{A})$ to fall below a preset tolerance, assuming that the algorithm converges at all. Drma\v{c}~\cite{Drmac2009} (see also \cite{arxiv_manuscript}) showed that convergence can be guaranteed provided block Jacobi is implemented with the following choices:
\begin{enumerate}
    \item Off-diagonal blocks are handled according to a standard row/column cyclic ordering.
    \item A \textit{pivoting step} is included, in which (1) column-pivoted QR (QRCP) or LU with partial pivoting (LUPP) is applied to the leading rows of $\hat{\M{Q}}$ 
    to generate a permutation matrix $\M{P}$ and (2) the update $\hat{\M{Q}} \leftarrow \hat{\M{Q}} \M{P}$ is applied. In Algorithm~\ref{alg: block Jacobi}, this would occur after line 7.
\end{enumerate}
Importantly, these design choices determine the convergence behavior of block Jacobi, and hence the value of $N$, but do not change the asymptotic order of the per-sweep costs $S_{\mathrm{sweep}}, W_{\mathrm{sweep}},$ and $F_{\mathrm{sweep}}. $ 

\indent Our focus in this paper is complexity and not convergence. For this reason, we omit the pivoting step from Algorithm~\ref{alg: block Jacobi} and do the same in all subsequent versions of Jacobi's method. While this allows us to write streamlined pseudocode, it also means that none of our algorithms come with a guarantee of convergence. Accordingly, the  complexity bounds we present are stated for only a single sweep of each algorithm. Nevertheless, they capture the general cost of the full Jacobi procedure. We note in particular the following:
\begin{itemize}
    \item As explored in \cite{arxiv_manuscript}, failure examples for block Jacobi are rare, meaning a naive implementation (i.e., one that omits pivoting) typically converges in $O(1)$ sweeps without any additional help, in which case our complexity bounds apply as is.
    \item If we are concerned about convergence, the machinery necessary to guarantee it can be added to our algorithms without altering their single-sweep complexities. We discuss this in  more detail in Sections~\ref{section: 2D_parallel} and~\ref{section: 2.5D_parallel}.
\end{itemize}

\subsection{Processor Layouts}\label{section: processor_layouts}
The complexity of a single sweep of parallel Jacobi depends not only on the algebraic operations performed but also on the way that data/work are distributed across the $P$ processors and how those processors are arranged relative to one another. As has been established for a broad class of dense linear algebra algorithms~\cite{BDD11,2.5D_QR_Band_Reduction,Solomonik_thesis,3D_QR} -- and is captured for distributed-memory matrix multiplication in Table~\ref{tab: complexity_bounds} -- the potential variability in complexity, especially bandwidth/latency, can be significant. \\
\indent In this paper, we consider parallelizing (blocked) Jacobi's method on the following processor grids. In both cases, we assume that processors have access to the collective communication operations summarized in Table~\ref{tab:mpi_routines} (see also \cite{thakur05:_optim_of_collec_commun_operat_in_mpich}).
\begin{enumerate}
    \item \textbf{2D Layout:} Processors are arranged in a $\sqrt{P}\times \sqrt{P}$ grid and have fast memory of size $\Theta(n^2/P)$. Input matrices are block distributed over the processors -- i.e., each one owns a contiguous block that is roughly the size of its fast memory. In this setting, collective communication is largely confined to local processor rows/columns, which can reduce communication costs (compared to a single 1D row/column of processors). 
    \item \textbf{2.5D Layout:} Processors are arranged in a $\sqrt{P/c} \times \sqrt{P/c} \times c$ grid for a \textit{replication factor} $c$ satisfying $1 \leq c \leq P^{1/3}$. The third dimension here corresponds to grid \textit{layers}, where each layer (there are $c$ in total) is itself a 2D grid of size $\sqrt{P/c} \times \sqrt{P/c}$. In this setup, input matrices are typically stored on the top (or bottom) layer of the grid and replicated across the others -- e.g., so that each 2D layer has the inputs necessary to perform a local computation, which is eventually merged with computations done by other layers via a global reduction. This can further reduce communication at the expense of increased memory; to store an $n \times n$ matrix on every layer, note that the fast memory available to each processor must increase to size $\Omega(cn^2/P)$. This setting generalizes standard 2D and 3D processor layouts, which correspond to $c = 1$ and $c = P^{1/3}$, respectively.

    
    
\end{enumerate}

\renewcommand{\arraystretch}{1.3}
\renewcommand\multirowsetup{\centering}
\begin{table}[t]
\centering
\begin{tabular}{l p{7cm} l}
\toprule
\textbf{Routine} & \textbf{Description} & \textbf{Cost} \\ 
\midrule
Scatter   & A root scatters $n$ words, each processor receives $n/P$ words
          & $\alpha \cdot \log_2 P + \beta \cdot n$ \\[2pt]
Broadcast & A root broadcasts $n$ words to all processors
          & $\alpha \cdot 2\log_2 P + \beta \cdot 2n$ \\[2pt]
Gather    & Each processor sends $n/P$ words, which are gathered on the root
          & $\alpha \cdot \log_2 P + \beta \cdot n$ \\[2pt]
Reduce    & Reduction on $n$ words from each processor, result returned on the root
          & $\alpha \cdot 2\log_2 P + \beta \cdot 2n + \gamma \cdot n$ \\[2pt]

\bottomrule
\end{tabular}
\caption{Collective communication operations on $P$ processors and their associated costs (in the $\alpha$--$\beta$--$\gamma$ model of \eqref{eq:total cost model}). In each case, we assume $n \geq P$. The \emph{root} is the distinguished processor that initiates the operation (for scatter and broadcast) or receives the final result (for gather and reduce). For more background, see \cite{thakur05:_optim_of_collec_commun_operat_in_mpich}.}
\label{tab:mpi_routines}
\end{table}
\renewcommand{\arraystretch}{1}

 Sections~\ref{section: 2D_parallel} and ~\ref{section: 2.5D_parallel} cover our 2D and 2.5D versions of parallel Jacobi. Our goal throughout is to show that (1) 2D parallel Jacobi can (nearly) attain the complexity bounds of 2D matrix multiplication \cite{Cannon69,vdGW97}
 and (2) that -- as with other dense linear algebra operations (see again the references above) -- moving to a 2.5D processor grid can further reduce communication costs. \\
 \indent While 2D parallel Jacobi algorithms have been considered in the literature before, we note that the 2.5D version presented here is new. This is not, however, the first 2.5D algorithm for the symmetric eigenvalue problem. In \cite{2.5D_QR_Band_Reduction}, Solomonik et al.\ present a 2.5D band reduction algorithm, which can compute the eigenvalues of a symmetric matrix at reduced bandwidth cost. Nevertheless, this algorithm outputs eigenvectors implicitly as products of Householder reflectors, and computing them explicitly may require an additional $O(n^3 \log P/P)$ flops. As we demonstrate in Section~\ref{section: 2.5D_parallel}, our 2.5D version of parallel Jacobi can match the improved bandwidth of this algorithm while outputting a full set of eigenvalues and eigenvectors. 


\section{2D Parallel Jacobi}\label{section: 2D_parallel}

The first algorithm we introduce is 2D parallel Jacobi, where the processor grid and matrix partitioning follow the 2D layout described in Section~\ref{section: processor_layouts}. The setup here is somewhat standard (see  e.g., \cite{Sameh_1971,Berry_Sameh_1989,Shroff_Schreiber_1989}). Our goal in this section is to exhibit (in general) the near optimality of these algorithms. \\
\indent We begin by summarizing the relevant notation. For 2D parallel Jacobi, we consider $P$ processors arranged in a two-dimensional grid of size $\sqrt{P} \times \sqrt{P}$. Each processor is labeled Proc$(I,J)$ for a pair of indices $1 \leq I,J \leq \sqrt{P}$, which denote the row/column of the grid that the processor belongs to. The input matrix $\M A \in \mathbb{R}^{n\times n}$ (likewise the iteratively constructed eigenvector matrix $\M{Q}$) is blockwise partitioned with the same indexing -- i.e., the subblock assigned to Proc$(I,J)$ is $\M{A}_{IJ}$ as defined in \eqref{eqn: A_subblock}, where we take $b = n/\sqrt{P}$. For convenience, we'll assume that $\sqrt{P}$ and $b$ are both integers.

\indent The key insight to implementing block Jacobi in parallel is the following:\ multiple $2b \times 2b$ rotations can be done simultaneously provided the blocks being annihilated are suitably independent, meaning they lie in different (block) rows/columns of $\M{A}$. Hence, following \cite{megiddo1983applying}, we can avoid the serial bottleneck of Algorithm~\ref{alg: block Jacobi} by partitioning off-diagonal blocks of $\M{A}$ into groups that can be handled in parallel. Of course, to make the most of efficiency gains, we need to ensure this can be done in a roughly load-balanced fashion; that is, can enough rotations be done at once to keep all (or most) of the processors busy? For our $\sqrt{P} \times \sqrt{P}$ block matrix we note the following:
\begin{enumerate}
\item First, we formalize the notion of block independence:\ if $\Mb{A}{IJ}$ and $\Mb{A}{KL}$ are two super-diagonal blocks of
$\M{A}$ (i.e., $I<J$ and $K<L$) that lie in different rows and columns (so $\{I,J\}\cap\{K,L\}=\emptyset$),
then their block Jacobi transformations can be computed independently
and applied in parallel. We call such transformations {\em disjoint}.
Blocks like $\Mb{A}{IK}$, which might require updating by both transformations,
require (matrix) multiplication by one transformation on the left and by the other on the right. Hence, they can be updated in any order. 
\item There are a total of $\sqrt{P}(\sqrt{P}-1)/2$ pairs $(I,J)$ with $1 \leq I < J \leq \sqrt{P}$. We can partition these block indices into $\sqrt{P}$ groups, each of size around $\sqrt{P}/2$, such that each group contains pairwise disjoint pairs (equivalently, at most one block in each row or column). For example, when $\sqrt{P}=8$, we can group the superdiagonal blocks into 8 groups of size 3 or 4 as follows. Here, superdiagonal blocks are labeled by their group number:
\begin{equation}
\begin{bmatrix}* & 7 & 6 & 5 & 4 & 3 & 2 & 1\\
 & * & 5 & 4 & 3 & 2 & 1 & 8\\
 &  & * & 3 & 2 & 1 & 8 & 7\\
 &  &  & * & 1 & 8 & 7 & 6\\
 &  &  &  & * & 7 & 6 & 5\\
 &  &  &  &  & * & 5 & 4\\
 &  &  &  &  &  & * & 3\\
 &  &  &  &  &  &  & *
\end{bmatrix}\label{eqn:disjoint}
\end{equation}

\item Since 
each of these groups consists of one or two diagonals, it is easy to see that their transformations are pairwise disjoint. We illustrate this below with group one, where we now label the four members of the group with different numbers and use them to indicate which blocks of $\M{A}$ their corresponding transformations are applied to. For example, the label ``12" means that transformation 1 is applied on the left and transformation 2 on the right (while the label \textbf{3} means transformation 3 is applied on both the left and right). If eigenvectors are requested, then the transformations are also applied on the right of a cumulative rotation matrix $\M{Q}$.
\begin{equation}\label{eqn: group_one_transformation}
\begin{bmatrix}{\bf 4} & 43 & 42 & 41 & 41 & 42 & 43 & {\bf 4}\\
 & {\bf 3} & 32 & 31 & 31 & 32 & {\bf 3} & 34\\
 &  & {\bf 2} & 21 & 21 & {\bf 2} & 23 & 24\\
 &  &  & {\bf 1} & {\bf 1} & 12 & 13 & 14\\
 &  &  &  & {\bf 1} & 12 & 13 & 14\\
 &  &  &  &  & {\bf 2} & 23 & 24\\
 &  &  &  &  &  & {\bf 3} & 34\\
 &  &  &  &  &  &  & {\bf 4}
\end{bmatrix}
\end{equation}
\item Since we assume that each block is assigned to a different processor (in this case, $P = 64$) the work represented by \eqref{eqn: group_one_transformation} is load
balanced. This is approximately true for the other groups in (\ref{eqn:disjoint}), even though they do not necessarily update all block rows/columns from the left and right. For example, group two does not update blocks (4,4), (4,8) and (8,8) (missing updates are indicated by ``\_''): 
\begin{equation}\label{eqn: alt_transformations}
\begin{bmatrix}{\bf 3} & 32 & 31 & 3\_ & 31 & 32 & {\bf 3} & 3\_\\
 & {\bf 2} & 21 & 2\_ & 21 & {\bf 2} & 23 & 2\_\\
 &  & {\bf 1} & 1\_ & {\bf 1} & 12 & 13 & 1\_\\
 &  &  & \_\_ & \_1 & \_2 & \_3 & \_\_\\
 &  &  &  & {\bf 1} & 12 & 13 & 1\_\\
 &  &  &  &  & {\bf 2} & 23 & 2\_\\
 &  &  &  &  &  & {\bf 3} & 3\_\\
 &  &  &  &  &  &  & \_\_
\end{bmatrix}
\end{equation}
These cases, while not perfectly load balanced, are good enough for our purposes.
\item As suggested by the above examples, we can technically store $\M{A}$ using only half of the processors in the grid by keeping only its upper triangle. Nevertheless, we need all $P$ processors to store/update the eigenvector matrix $\M{Q}$ (assuming it is blocked in the same way as $\M{A}$) as it is nonsymmetric in general. 
For this reason, and since it would change our subsequent complexity bounds by only a constant factor, we work with the full matrix $\M{A}$ over the entire processor grid.
\end{enumerate}

Algorithm~\ref{alg:High_Level_Parallel_Jacobi} presents 2D Parallel Jacobi -- a parallel version of Algorithm~\ref{alg: block Jacobi} that exploits these observations. Since similar algorithms have been considered in the literature before, the pseudocode here is presented at a high level and applies only a single sweep of block Jacobi to $\M{A}$; executing the full algorithm will therefore require repeatedly calling this routine and monitoring progress towards convergence. This may require some care if, for example, $\Xi(\M{A})$ is to be computed over the distributed processor grid. We provide a more detailed version of Algorithm~\ref{alg:High_Level_Parallel_Jacobi}, specifically written at the level of an individual processor, in Appendix~\ref{appendix: detailed_code}. 

\begin{algorithm}
\caption{2D Parallel Jacobi for the Symmetric Eigenvalue Problem}
\label{alg:High_Level_Parallel_Jacobi}
\textbf{Input:} (1) $\M{A} \in {\mathbb R}^{n \times n}$ a symmetric matrix. \\
\phantom{Input2 } (2) $\M{Q} \in {\mathbb R}^{n \times n}$ an orthogonal matrix of approximate eigenvectors (optional). \\
\phantom{Input2 } (3) A set $\mathfrak{P}$ of $P$ processors, assumed to be arranged in a $\sqrt{P} \times \sqrt{P}$ grid and labeled Proc$(I,J)$  \\
\phantom{Input2 (1) } for indices $1 \leq I,J \leq \sqrt{P}$. \\
\phantom{Input2 } (4) $L$ a list of $\sqrt{P}$ groups of pairwise disjoint indices $(I,J)$ (e.g.,  generated by Algorithm~\ref{alg:anti_diagonal_list}).

\vspace{1mm}
\textbf{Requires:} Proc$(I,J)$ owns the block $\Mb{A}{IJ}=\Me{A}{(I{-}1)b{+}1:Ib,(J{-}1)b{+}1:Jb}$ of $\M{A}$ for $b = n /\sqrt{P}$, which is assumed to be an integer. If omitted, $\M{Q}$ is taken to be the identity matrix; in either case, it is partitioned and stored in the same way as $\M{A}$. 

\vspace{1mm}
\textbf{Ensure:}
On output, one sweep of (block) Jacobi has been applied to $\M{A}$ and $\M{Q}$.
\algrule
\setstretch{1.1} 
\begin{algorithmic}[1]
\Function{$[\M{Q},\M{A}]=$
2D\_Parallel\_Jacobi}{$\M{A}$, $\M{Q}$, $\mathfrak{P}$, $L$}
\For{\textbf{each} $g\in L$} \Comment{Outer serial loop}
  \For{\textbf{each} $(I,J)\in g$ \textbf{in parallel}}
    \State Proc$(I,I)$: Send $\M{A}_{II}$ to Proc$(I,J)$ 
    \State Proc$(J,J)$: Send $\M{A}_{JJ}$ to Proc$(I,J)$
    \vskip 3pt
    \State $\begin{pmatrix} \Mb{A}{II} & \Mb{A}{IJ} \\
    \Mb{A}{JI} & \Mb{A}{JJ} \end{pmatrix}= \M{V}\M{D}\M{V}^T$ \label{step: diag}
    \Comment{Diagonalize the $2b \times 2b$ subproblem on Proc$(I,J)$}
    \vskip 3pt
    \State $\M{V} = \begin{pmatrix} \M{V}_{11} & \M{V}_{12} \\ \M{V}_{21} & \M{V}_{22} \end{pmatrix}; \; \; \; \M{D} = \begin{pmatrix} \M{D}_{11} & \M{D}_{12} \\ \M{D}_{21} & \M{D}_{22} \end{pmatrix}$ \Comment{Break $\M{V}$ and $\M{D}$ into $b \times b$ subblocks}
    \vskip 3pt
    \State Proc$(I,J)$:\ Send $[\Mb{V}{11};\Mb{V}{21}]$ and $\Mb{D}{11}$ to Proc$(I,I)$, $[\Mb{V}{12};\Mb{V}{22}]$ and $\Mb{D}{22}$ to Proc$(J,J)$, and  \phantom{test test test test test} $[\Mb{V}{11};\Mb{V}{21}]$, $\Mb{V}{22}$, and $\Mb{D}{21}$ to Proc($J,I$)
    \vskip 3pt
    \State $\begin{pmatrix} \M{A}_{II} & \M{A}_{IJ} \\ \M{A}_{JI} & \M{A}_{JJ} \end{pmatrix} \gets \begin{pmatrix} \M{D}_{11} & \M{D}_{12} \\ \M{D}_{21} & \M{D}_{22} \end{pmatrix}$ \Comment{Update $\M{A}([I,J],[I,J])$ on the corresponding processors}
    \vskip 3pt
    \For{$K = 1:\sqrt{P}$ \textbf{in parallel}}
        \If{$K \neq I,J$}
        \State Proc$(I,K)$ and Proc$(J,K)$: Exchange $\M{A}_{IK} \leftrightarrow \M{A}_{JK}$
        \State Proc$(K,I)$ and Proc$(K,J)$: Exchange $\M{A}_{KI} \leftrightarrow \M{A}_{KJ}$
        \EndIf
        \State Proc$(K,I)$ and Proc$(K,J)$: Exchange $\M{Q}_{KI} \leftrightarrow \M{Q}_{KJ}$
    \EndFor

    \State Proc$(I,I)$: Broadcast $[\M V_{11},\M V_{21}]$ to Proc$(I,K)$ and Proc$(K,I)$ for all $K \neq I,J$ \label{broadcast_start_2D}
    \State Proc$(J,J)$: Broadcast $[\M{V}_{22}, \M{V}_{12}]$ to Proc$(J,K)$ and Proc$(K,J)$ for all $K \neq I,J$ \label{broadcast_end}
    \For{$K = 1:\sqrt{P}$ \textbf{in parallel}} \label{step: comp_1} 
    \If{$K \neq I,J$} \Comment{Update block rows/columns of $\M{A}$}
        \vskip 0.5pt
        \State Proc$(I,K)$: $\M{A}_{IK} \gets \M{V}_{11}^{T}\cdot\M{A}_{IK}+\M{V}_{21}^{T}\cdot\M{A}_{JK}$ 
        \vskip 0.5pt
        \State Proc$(J,K)$: $\M{A}_{JK} \gets \M{V}_{12}^{T}\cdot\M{A}_{IK}+\M{V}_{22}^{T}\cdot\M{A}_{JK}$ 
        \State Proc$(K,I)$: $\M{A}_{KI} \gets \M{A}_{KI}\cdot\M{V}_{11}+\M{A}_{KJ}\cdot\M{V}_{21}$ 
        \State Proc$(K,J)$: $\M{A}_{KJ}\gets \M{A}_{KI}\cdot\M{V}_{12}+\M{A}_{KJ}\cdot\M{V}_{22}$ 
    \EndIf
    \State Proc$(K,I)$: $\M{Q}_{KI} \gets  \M{Q}_{KI}\cdot\M{V}_{11}+\M{Q}_{KJ}\cdot\M{V}_{21}$ \Comment{Update block columns of $\M{Q}$}
    \State Proc$(K,J)$: $\M{Q}_{KJ} \gets \M{Q}_{KI}\cdot\M{V}_{12}+\M{Q}_{KJ}\cdot\M{V}_{22}$ 
    \EndFor
  \EndFor
\EndFor
\EndFunction
\end{algorithmic} 
\end{algorithm}

\indent The fourth input to Algorithm~\ref{alg:High_Level_Parallel_Jacobi} is a list $L$ of $\sqrt{P}$ sets, each of which collects the indices of superdiagonal blocks that can be zeroed out in parallel. $L$ can be generated for any number of processors $P$ by calling Algorithm~\ref{alg:anti_diagonal_list}. For the above example, $L$ is as follows: \\
\begin{equation}\label{eqn: parallel_list}
    \aligned
    L=(&\{(4,5),(3,6),(2,7),(1,8)\}, \; \{(3,5),(2,6),(1,7)\}, \; \{(3,4),(2,5),(1,6),(7,8)\}, \{(2,4),(1,5),(6,8)\}, \\
    & \{(2,3),(1,4),(6,7),(5,8)\}, \; \{(1,3),(5,7),(4,8)\}, \{(1,2),(5,6),(4,7),(3,8)\}, \; \{(4,6),(3,7),(2,8)\}).
    \endaligned
\end{equation}
A sweep of block Jacobi is executed by working through this list (i.e., the outer serial loop), at each step computing and applying the corresponding block rotations in parallel. Note that we omit a check on how close each extracted subproblem is to diagonal before computing/applying these rotations. Once again, this is done to simplify the pseudocode and does not change the complexity of each sweep.

\begin{algorithm}[h]
\caption{Index List Generator for Parallel Jacobi}
\label{alg:anti_diagonal_list}
\textbf{Input:} $P$ the number of processors used, assumed to be arranged in a $\sqrt{P} \times \sqrt{P}$ grid for $\sqrt{P} \in {\mathbb N}$.

\vspace{1mm}
\textbf{Output:} A list $L$ of $\sqrt{P}$ groups of indices $(I,J)$ with $1 \leq I < J \leq \sqrt{P}$. Each group $g \in L$ consists of pairwise disjoint indices -- i.e., if $(I,J)$ and $(I',J')$ belong to $g$ then $\left\{ I,J \right\} \cap \left\{ I',J' \right\} = \emptyset$. The corresponding block ordering in parallel Jacobi is cyclic by anti-diagonals.

\algrule
\setstretch{1.1}
\begin{algorithmic}[1]
\Function{$L=$
Index\_List\_Generator}{$P$}
\State $L = [ \; \;]$
\For{$K = 0:\sqrt{P}-1$}
    \State $g = [\; \; ]$
    \For{$I = 1:\sqrt{P}-1$}
        \For{$J = I+1:\sqrt{P}$}
            \If{$I+J = K \bmod \sqrt{P}$}
                \State $g = [g; \;  (I,J)]$
            \EndIf
        \EndFor
    \EndFor
    \State $L = [L;  \;  g]$
\EndFor
\EndFunction
 \end{algorithmic} 
\end{algorithm}

\indent We call the parallel Jacobi ordering generated by Algorithm~\ref{alg:anti_diagonal_list} \textit{cyclic by anti-diagonals}. Notably, Luk and Park \cite{Luk_Park} showed that this ordering is essentially equivalent to doing block Jacobi in a row-cyclic fashion; hence, Algorithm~\ref{alg:High_Level_Parallel_Jacobi} comes with a guarantee of convergence (by way of Drmač \cite{Drmac2009}) provided a pivoting step is added after line~\ref{step: diag}. As mentioned above, pivoting requires performing QRCP or LUPP to a portion of the $2b \times 2b$ eigenvector matrix $\M{V}$. For references on how to do this in a distributed fashion, see \cite{Parallel_CALU_CAQR}.   

\indent Note that the diagonalization in line~\ref{step: diag} is left somewhat ambiguous, and in particular may be either exact or approximate. Since the pseudocode does not assume that the selected off-diagonal blocks have been fully annihilated, instead setting them equal to $\M{D}_{12}$ and $\M{D}_{21}$, Algorithm~\ref{alg:High_Level_Parallel_Jacobi} accommodates both. With this in mind, it should be noted that the convergence proof of Drmač \cite{Drmac2009} does not require exact diagonalization. \\
\indent We now apply our complexity analysis to Algorithm~\ref{alg:High_Level_Parallel_Jacobi}. Despite the more intricate pseudocode, bounds follow fairly easily from the serial blocked case~\cite{arxiv_manuscript}.

\begin{theorem}\label{thm: parallel_comp}
    One sweep of 2D parallel Jacobi (i.e., Algorithm~\ref{alg:High_Level_Parallel_Jacobi}) has complexity
    $$
        \alpha \cdot O\!\left(\sqrt{P}\log P\right) 
    + \beta \cdot O\!\left(\frac{n^2}{\sqrt{P}}\right) 
    + \gamma \cdot O\!\left(\frac{n^3}{P}\right),
    $$
    assuming $O(n^3)$ algorithms are used for matrix multiplication and diagonalization (the latter in line~\ref{step: diag}).
\end{theorem}
\begin{proof}
We consider bandwidth, latency, and arithmetic individually. 
\begin{itemize}

    \item Bandwidth:\ Each communication operation performed by a processor consists of either (1) sending/receiving at most $4b^2=4n^2/P$ words to/from an individual processor, or (2) broadcasting $2b^2$ words to $\sqrt{P}-2$ processors. Since each processor performs $O(1)$ such operations, and multiplying by the $O(\sqrt{P})$ number of steps in one sweep (corresponding to the outer loop of the algorithm), the total bandwidth cost is therefore
    \begin{equation}
        W(n,P)=O(\sqrt{P}\cdot b^2)=O(n^2/\sqrt{P}).
    \end{equation}

    \item Latency:\ Recalling Table~\ref{tab:mpi_routines}, each broadcast to $O(\sqrt{P})$ processors requires $O(\log P)$ messages, while each ``point-to-point'' send/receive requires just $O(1)$ message. Thus, the number of  messages sent per iteration is $O(\log P)$. Since, again, the outer loop of Algorithm~\ref{alg:High_Level_Parallel_Jacobi} executes $O(\sqrt{P})$ times, the total number of messages sent is
    \begin{equation}
        S(n,P)=O(\sqrt{P}\log P).
    \end{equation}
 
    \item Arithmetic:\ Computation is performed in line~\ref{step: diag}, which solves a $2b\times 2b$ symmetric eigenvalue problem, and in the for loop beginning in line~\ref{step: comp_1}, where each participating processor performs up to four $b \times b$ matrix multiplications. Assuming both diagonalization and matrix multiplication incur cubic cost, the per-iteration flop count is $O(b^3)$. Again multiplying by $O(\sqrt{P})$ for the number of iterations yields
    \begin{equation}
        F(n,P)=O(\sqrt{P}\,b^3)=O(n^3/P)
    \end{equation}

\end{itemize}
\indent Summing the above terms gives the desired bound.
\end{proof} \\

We note two key takeaways from this result. First, Algorithm~\ref{alg:High_Level_Parallel_Jacobi} achieves a (perfect) linear speedup in arithmetic cost. Moreover, under the minimum-memory assumption of $\Theta(n^2/P)$ words per processor, the bandwidth cost of Algorithm~\ref{alg:High_Level_Parallel_Jacobi} attains the asymptotic lower bound for matrix multiplication, while the latency cost matches the lower bound up to a factor of $\log P$. In this sense, 2D Parallel Jacobi is nearly optimal. 

\section{2.5D Parallel Jacobi}\label{section: 2.5D_parallel}

The bandwidth and latency costs of 2D parallel matrix multiplication are only optimal in a minimum-memory setting -- i.e., when each processor has fast memory of size $\Theta(n^2/P)$ and only a single copy of the input/output matrices can be stored across the grid. If instead each processor has access to fast memory of size $\Omega(cn^2/P)$ for a replication factor $1 \leq c \leq P^{1/3}$, so that the grid can store $c$ copies of the input matrices, then the 2.5D algorithm of Solomonik and Demmel \cite{SD11,KK+19} 
can be used, which beats the latency and bandwidth costs of the standard 2D algorithm by a factor of $c^{3/2}$ and $\sqrt{c}$, respectively, while maintaining a perfect speedup in arithmetic. As its name suggests, 2.5D matrix multiplication employs a 2.5D processor layout, in which $P$ processors are arranged in a $\sqrt{P/c} \times \sqrt{P/c} \times c$ grid (recall Section~\ref{section: processor_layouts}). Put simply, 2.5D matrix multiplication attains lower communication complexity at the (modest) cost of additional memory. \\
\indent In this section, we take up the question of whether the same can be said of parallel Jacobi. In short, the answer is yes:\ we derive a 2.5D version of the algorithm that can beat both the bandwidth and latency costs of 2D parallel Jacobi -- though, unlike 2.5D matrix multiplication, it cannot beat both simultaneously (see Section~\ref{section: lower_bound}).  \\
\indent Throughout this section, we denote the processors in our 2.5D grid by Proc$\langle p,q,r\rangle$, where
\begin{equation}
  p,q\in\{1,\dots,\sqrt{P/c}\} \quad \text{and} \quad
  r\in\{1,\dots,c\}.
\end{equation}
These labels can be interpreted as follows:\  Proc$\langle p,q,r \rangle$ belongs to row $p$ and column $q$ in the 2D grid on layer $r$. As a convention, $r=1$ corresponds to the \emph{top} layer of the grid.

\subsection{2.5D Matrix Multiplication}
We begin by giving a detailed overview of 2.5D matrix multiplication. In essence, this algorithm is a straightforward generalization of Cannon's 2D matrix multiplication \cite{Cannon69}.
In Cannon's algorithm, input matrices $\M{A}$ and $\M{B}$ are block-distributed over a $\sqrt{P} \times \sqrt{P}$ processor grid and each block of the product $\M{C} = \M{A}\M{B}$ is computed by a single processor; that is, in the notation of the previous section, Proc$(I,J)$ computes $\M{C}_{IJ}$. This is done by breaking the multiplication into $n/\sqrt{P}$ steps, each of which consists of the following:
\begin{enumerate}
    \item First, blocks of $\M{A}$ and $\M{B}$ are cyclically shuffled across the grid so that Proc$(I,J)$ has both $\M{A}_{IK}$ and $\M{B}_{KJ}$ in its fast memory for some $1 \leq K \leq n/\sqrt{P}$. The value of $K$ is different for each step according to a processor-dependent ordering. 
    \item Next, the processors simultaneously perform a local matrix multiplication of size $O(n^2/P)$ and update the block of $\M{C}$ they own. That is, Proc$(I,J)$ executes $\M{C}_{IJ} \leftarrow \M{C}_{IJ} + \M{A}_{IK} \M{B}_{KJ}$. 
\end{enumerate}
The motivation here is simple:\ to compute $\M{C}_{IJ}$, processor Proc$(I,J)$ must see everything in block row $I$ of $\M{A}$ and block column $J$ or $\M{B}$ but -- assuming the same minimum-memory setting as above -- can only fit one pair $\M{A}_{IK}$ and $\M{B}_{KJ}$ in fast memory at a time (at which point the other blocks needed to compute $\M{C}_{IJ}$ can be used by other processors). \\
\indent 2.5D matrix multiplication further parallelizes this algorithm by partitioning the $n/\sqrt{P}$ steps into $c$ groups, each of which is executed on a separate layer of the processor grid. The high-level approach is as follows:
\begin{enumerate}
    \item Initially, $\M{A}$ and $\M{B}$ are block-distributed across the top layer of the grid. The algorithm begins by broadcasting $\M{A}$ and $\M{B}$ (in blocks) across the other layers. Note that the additional memory on each process ensures that $\M{A}$ and $\M{B}$ can be stored on each layer. 
    \item After initially cycling subblocks if necessary, 
    each layer simultaneously performs $1/c$ of the 2D algorithm over its local $\sqrt{P/c} \times \sqrt{P/c}$ grid.
    \item Finally, a reduction across the layers yields the output $\M{C}$ on the top of the grid. 
\end{enumerate}
\noindent Pseudocode for this procedure is presented in Algorithm~\ref{alg:2.5DMM}. 
\\
\indent Complexity bounds for 2.5D matrix multiplication are summarized below. Note that, as mentioned above, both the bandwidth and latency costs of 2.5D matrix multiplication improve on the 2D version.

\begin{theorem}[Solomonik and Demmel \cite{SD11}] \label{thm: 2.5d_MM}
The complexity of 2.5D parallel matrix multiplication (i.e., Algorithm~\ref{alg:2.5DMM}) is 
$$\alpha\cdot O\left(\sqrt{P/c^3} + \log c\right) + \beta\cdot O\left(\frac{n^2}{\sqrt{cP}}\right) + \gamma\cdot O\left(\frac{n^3}{P}\right).$$
for $n$ the size of the (square) input matrices multiplied, $P$ the number of processors used, and $c$ a replication factor satisfying $1 \leq c \leq P^{1/3}$.
\end{theorem}

\begin{algorithm}[t]
\caption{2.5D Matrix Multiplication}
\label{alg:2.5DMM}
\textbf{Input:} $\M{A},\M{B} \in {\mathbb R}^{n \times n}$ and a set $\mathfrak{P}$ of $P$ processors, assumed to be arranged in a $\sqrt{P/c} \times \sqrt{P/c} \times c$ processor grid for replication factor $1 \leq c \leq P^{1/3}$ and labeled Proc$\langle p,q,r \rangle$ for indices $1\leq p,q \leq \sqrt{P/c}$ and $1 \leq r \leq c$.

\vspace{1mm}
\textbf{Requires:} Initially, $\M{A}$ and $\M{B}$ are stored on the top layer of the grid, with Proc$\langle i,j,1 \rangle$ owning both 
\[ \Mb{A}{ij}=\Me{A}{(i{-}1)b{+}1:ib,(j{-}1)b{+}1:jb} \; \; \; \; \text{and} \; \;  \; \; \Mb{B}{ij}=\Me{B}{(i{-}1)b{+}1:ib,(j{-}1)b{+}1:jb} \]
for $b = n/\sqrt{P/c}$. In addition, each processor allocates three $b \times b$ arrays $\M{A}_{\text{local}}$, $\M{B}_{\text{local}}$, and $\M{C}_{\text{local}}$ to be updated throughout the computation.

\vspace{1mm}
\textbf{Output:} $\M{C} = \M{A}\M{B}$, stored on the top layer of the grid.

\algrule
\setstretch{1.1}
\begin{algorithmic}[1]
\Function{$\M{C}=$
2.5D\_MM}{$\M{A}$, $\M{B}$, $\mathfrak{P}$}
\ForAll{$i,j \in [\sqrt{P/c}]$ \textbf{and} $k \in \left[c \right]$ \textbf{in parallel}} \Comment{All actions are executed by Proc$\langle i,j,k \rangle$}
\If{$k = 1$}
\State Broadcast $\M{A}_{ij}$ and $\M{B}_{ij}$ to Proc$\langle i,j,l \rangle$ for $2 \leq l \leq c $ \Comment{Replicate $\M{A}$ and $\M{B}$}
\EndIf
\State $s = \left( j-i+(k-1)\sqrt{P/c^3} \mod \sqrt{P/c} \right) +1$
\State Send $\M{A}_{ij}$ to $\M{A}_{\text{local}}$ on Proc$\langle i,s,k\rangle$
\State $s' = \left( i-j+(k-1)\sqrt{P/c^3} \mod \sqrt{P/c}\right) + 1$
\State Send $\M{B}_{ij}$ to $\M{B}_{\text{local}}$ on  Proc$\langle s',j,k \rangle$
\State $\M{C}_{\text{local}} \gets \M{A}_{\text{local}} \M{B}_{\text{local}}$
\State $s = \left(j\mod \sqrt{P/c}\right)+1$ 
\State $s' = \left( i \mod \sqrt{P/c}\right)+1$
\For{$t = 1:\sqrt{P/c^3}-1$} \Comment{Cyclically update local arrays}
\State Send $\M{A}_{\text{local}}$ to Proc$\langle i,s,k \rangle$ and $\M{B}_{\text{local}}$ to Proc$\langle s',j,k \rangle$
\State $\M{C}_{\text{local}} \gets \M{C}_{\text{local}} + \M{A}_{\text{local}} \M{B}_{\text{local}}$
\EndFor
\State Contribute $\M{C}_{\text{local}}$ to a sum-reduction to Proc$\langle i,j,1 \rangle$.
\EndFor
\EndFunction
 \end{algorithmic} 
\end{algorithm}

\subsection{Extension to Jacobi}
We now consider parallelizing block Jacobi on a 2.5D processor grid. As in 2.5D matrix multiplication, we assume that the input matrix $\M{A}$ and the eigenvector matrix $\M{Q}$ are initially distributed across the top layer of the grid and can be replicated across the others as needed. At a high level, 2.5D parallel Jacobi runs the 2D algorithm on the top layer of the grid while distributing work -- e.g., subproblem diagonalization and block row/column updates -- across the other layers to reduce communication. \\
\indent It is easy to see why this would be beneficial for the block row/column updates. If $B$ is the block size, each such update corresponds to the product of a $2B \times 2B$ orthogonal matrix $\M{V}$ with either a $2B \times n$  or $n \times 2B$ submatrix of $\M{A}$. In 2D parallel Jacobi, these are handled by broadcasting blocks of $\M{V}$ to the relevant processors and doing local matrix multiplications (hence, incurring a bandwidth cost of $O(B^2) = O(n^2/P)$ at each parallel step). If these multiplications are instead handled by the 2.5D algorithm -- e.g., with independent groups of $P'$ processors, each  arranged in a 2.5D grid with $c$ layers -- the  bandwidth cost is only $O(B^2/\sqrt{c P'})$ at each step (applying Theorem~\ref{thm: 2.5d_MM}). \\
\indent There is a catch here:\ if $P'$ processors are to collaborate on each $B \times B$ matrix multiplication via the 2.5D algorithm, and assuming all memory is used, then the block size $B$ must be larger than the fast memory of any individual processor. 
This is in contrast to the 2D algorithm from the previous section, where the block size was exactly the size of each processor's fast memory and each such multiplication was done locally. To avoid confusion, we will use the following notation in this section:
\begin{itemize}
    \item $b = n/\sqrt{P/c}$ is the size of the block of $\M{A}$ (and $\M{Q}$) owned by each processor. This matches how $b$ was defined in Section~\ref{section: 2D_parallel}.
    \item $B$ is the block size used in Jacobi, where $B = \mu b$ for $1 \leq \mu \leq \sqrt{P/c}$ a batching parameter. For simplicity, we assume that $\mu$, $b$, and $B$ are all integers. 
\end{itemize}
These block sizes naturally divide $\M{A}$ into \textit{fine tiles} of size $b \times b$ and  \textit{coarse tiles} of size $B \times B$. On each layer of the processor grid, fine tiles  are owned by individual processors while coarse tiles are owned collectively by $\mu^2$ processors. The aforementioned $B \times B$ matrix multiplications can therefore be done over a $\mu \times \mu \times c$ subset of the 2.5D processor grid. This implies a lower bound on the batching parameter $\mu$; to use the 2.5D algorithm in this context, we need $c \leq \mu$. \\
\indent Since Jacobi will operate on coarse tiles of $\M{A}$, we follow Section~\ref{section: 2D_parallel} and label them with pairs $(I,J)$ for $1 \leq I,J \leq n/B$, where the corresponding submatrix $\M{A}_{IJ}$ is defined as in \eqref{eqn: A_subblock} for block size $B$. Since the division of $\M{A}$ into coarse and fine tiles naturally extends to the processor grid, we use similar notation to refer to groups of processors that operate collectively at the coarse-tile level. In particular, for $1 \leq I,J \leq n/B, 1 \leq k \leq c$, the set
\begin{equation}\label{eqn: processor_tile}
\text{ProcTile}(I,J,k) := \{\, \text{Proc}\langle p,q,k\rangle \;|\; (I{-}1)\mu{+}1\leq p \leq I\mu,\ (J{-}1)\mu{+}1 \leq q \leq J\mu\,\}
\end{equation}
denotes the processors that make up the coarse \textit{processor tile} handling $\M{A}_{IJ}$ on layer $k$.  Naturally, each such tile consists of $\mu \times \mu$ individual processors (which are the processor-analog of the fine tiles from above). For clarity, Figure~\ref{fig: processor_grid_topology} illustrates the 2.5D processor grid and its division into both 2D layers and coarse/fine (processor) tiles on each layer. To make it easier to refer to ``slabs" of processors across multiple coarse tiles or layers, we use ordered lists (e.g., standard Matlab notation) to select multiple coarse tiles. For example, the sets 
\begin{equation}\label{eqn: proc_slab}
    \aligned 
    \text{ProcSlab}([I,J],[I,J],1) &= \bigcup_{\mathcal{P}, \mathcal{Q} \in \{I,J\}} \text{ProcTile}(\mathcal{P}, \mathcal{Q}, 1) \\
    \text{ProcSlab}(:,[I,J],:) &=  
    \bigcup_{\substack{\text{all } {\mathcal P}, r \\ Q \in \left\{ I,J \right\}}} \text{ProcTile}({\mathcal P},{\mathcal Q,r)}
    \endaligned 
\end{equation}
represent, respectively, the $2 \times 2$ block of coarse processor tiles on the top layer that owns the submatrix $\M A([I,J], [I,J])$ and the slab consisting of all processors with coarse column index $I$ or $J$. 

\begin{figure}[t]
    \centering
    \includegraphics[width=\linewidth]{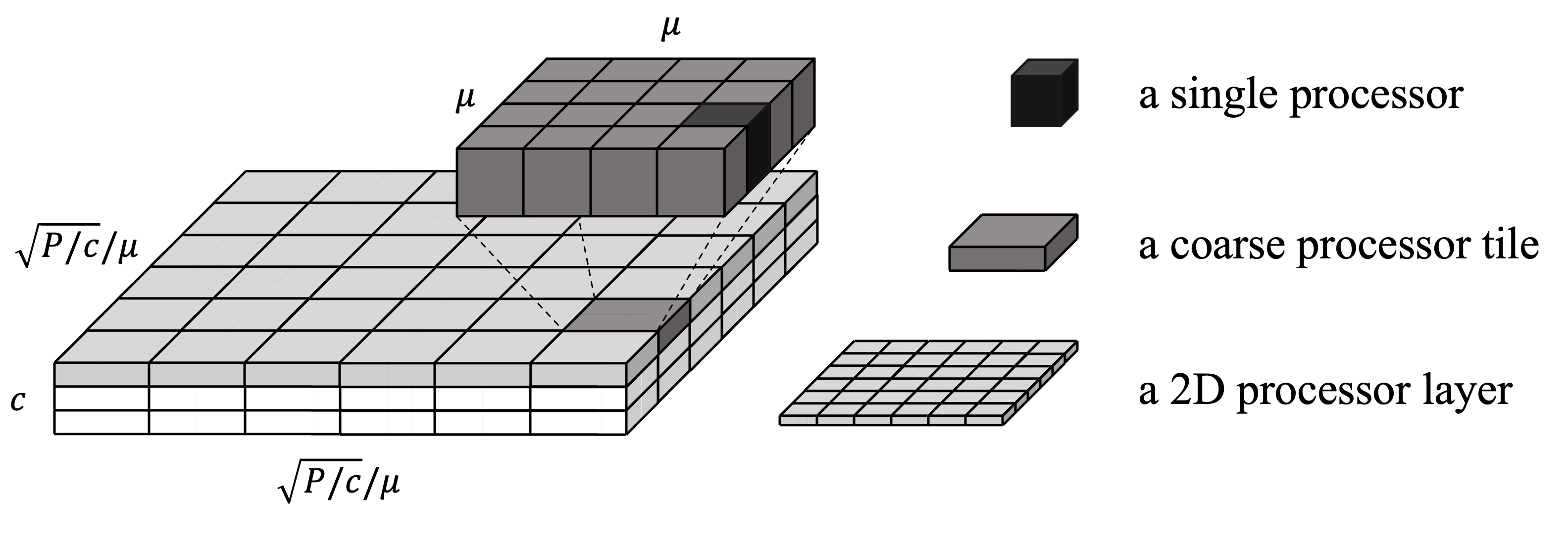}
    \caption{The 2.5D processor grid and its breakdown into 2D processor layers and coarse processor tiles. 
    }
    \label{fig: processor_grid_topology}
\end{figure}

\indent There is one step of Jacobi's method that we have not yet addressed:\ how should we diagonalize each $2B \times 2B$ subproblem $\hat{\M{A}}= \M{A}([I,J],[I,J])$? Since in the 2.5D setting this submatrix is initially block-distributed across the $2\mu \times 2 \mu$ processors in  ProcSlab$([I,J],[I,J],1)$, a natural option is the 2D parallel Jacobi algorithm from Section~\ref{section: 2D_parallel}, which can operate on $b\times b$ blocks of $\hat{\M{A}}$ (equivalently, fine tiles of $\M{A}$) over the 2D subset ProcSlab$([I,J],[I,J],1)$ of the full 2.5D processor grid. Recalling that the bandwidth and arithmetic costs of 2D parallel Jacobi decrease as the number of processors increases (see Theorem~\ref{thm: parallel_comp}), we might hope to use more than just the processors that initially own $\hat{\M{A}}$ here. Fortunately, we have more available; if the same anti-diagonal ordering from Section~\ref{section: 2D_parallel} is used, then the subproblems being diagonalized simultaneously in any step of the algorithm have disjoint (coarse-tile) column indices. Hence, the full processor slab ProcSlab$(:,[I,J],:)$ is free to use to diagonalize $\hat{\M{A}}$. This prompts us to introduce one final parameter -- a utilization factor $\eta \in [2\mu/\sqrt{cP},\,1]$, which determines the fraction of processors in ProcSlab$(:,[I,J],:)$ to be used for each subproblem diagonalization. The lower bound $\eta = 2\mu/\sqrt{cP}$ corresponds to only using the $4\mu^2$ processors in ProcSlab$([I,J],[I,J],1)$

\subsection{Communication Operations Between Coarse Processor Tiles}

\begin{figure}[t]
\centering

\begin{subfigure}{\textwidth}
    \centering
    \includegraphics[width=.4\linewidth]{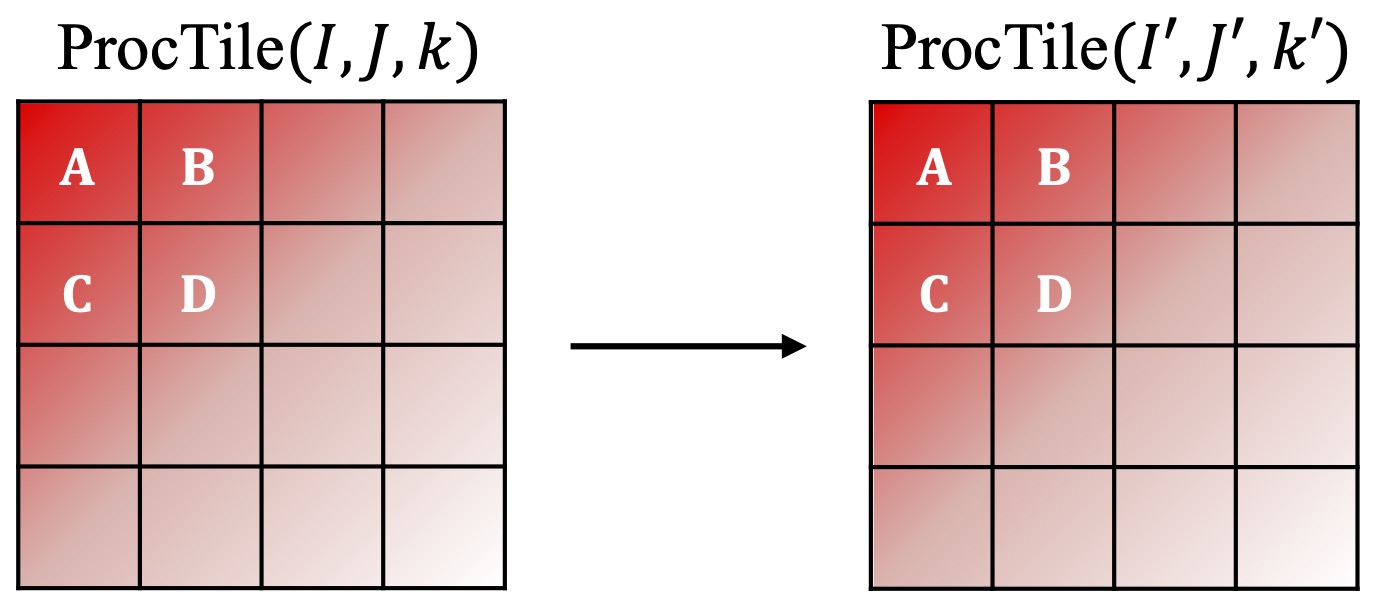}
    \caption{\centering Send}
\end{subfigure}
\begin{subfigure}{.43\textwidth}
    \centering
    \hspace{1cm}
    \includegraphics[width=\linewidth]{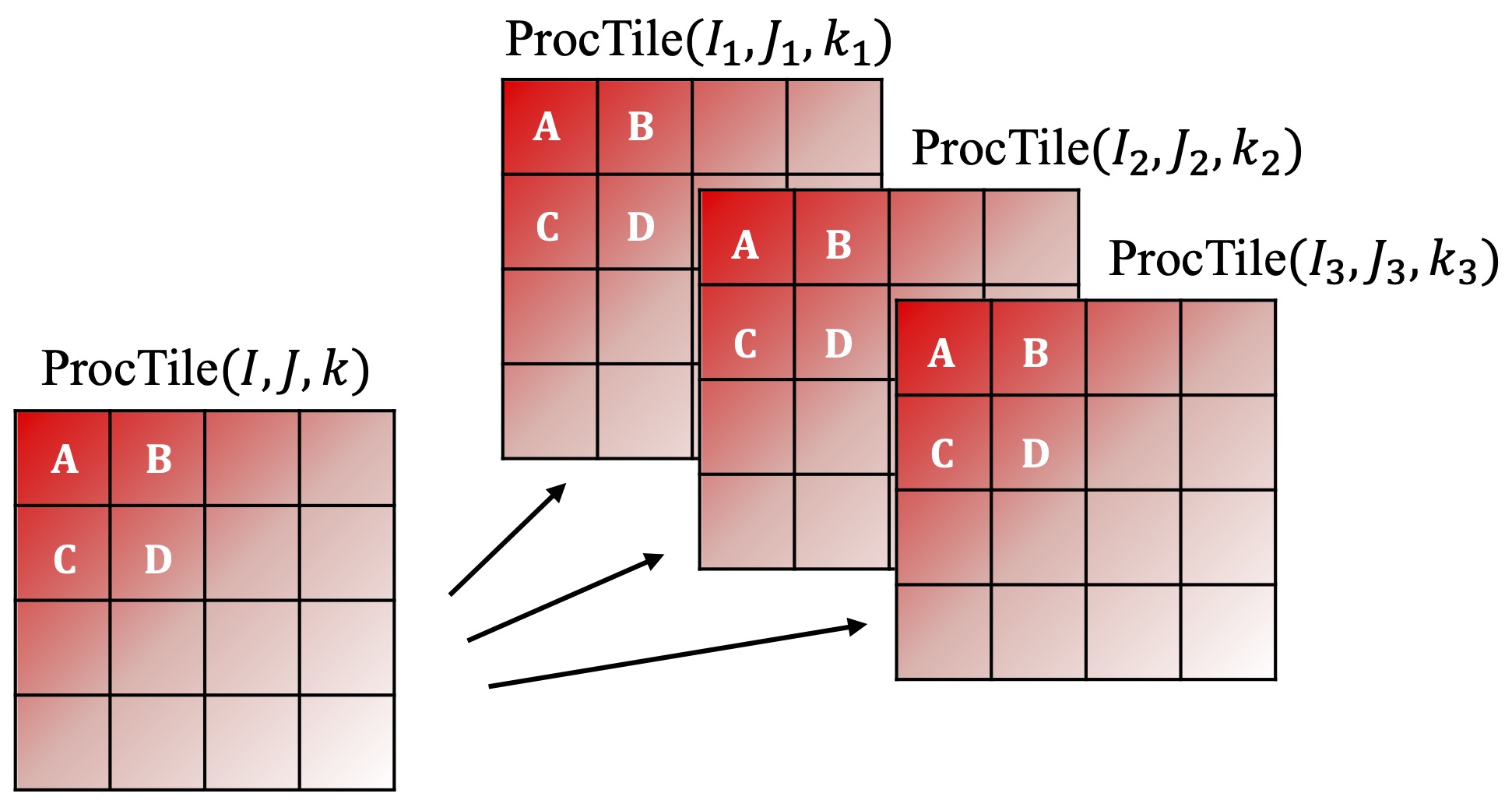}
    \caption{Broadcast}
\end{subfigure}\hfill
\begin{subfigure}{.5\textwidth}
    \centering
    \includegraphics[width=\linewidth]{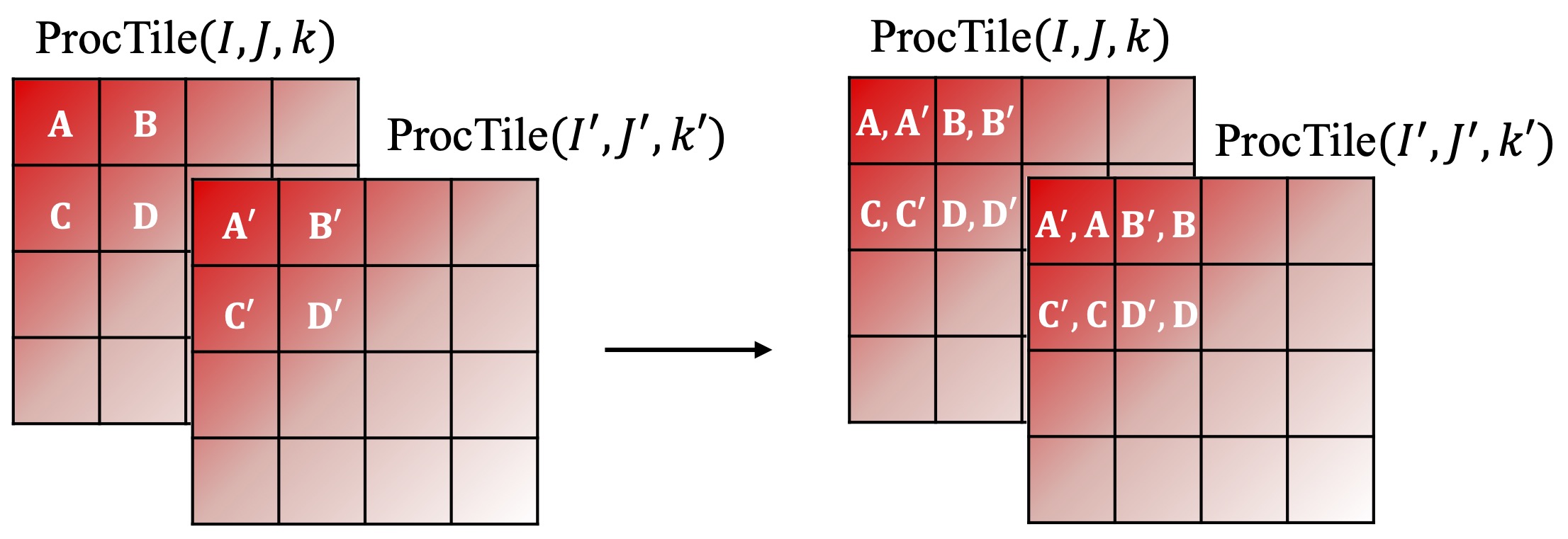}
    \caption{Exchange}
\end{subfigure}

\caption{Communication operations between coarse processor tiles in 2.5D parallel Jacobi. Each grid square represents an individual processor.}
\label{fig:tile_comm}

\end{figure}

In the next section, we present pseudocode for 2.5D parallel Jacobi written at the level of $B \times B$ coarse tiles of $\M{A}$. Accordingly, it includes communication instructions for coarse processor tiles. Each of these operations (e.g., send, broadcast, exchange) is executed in the natural way; that is, communication is performed between individual processors occupying the same location in each tile. This is depicted in Figure~\ref{fig:tile_comm}. \\
\indent There is one communication operation not covered here. As mentioned above, each $2B \times 2B$ subproblem $\hat{\M{A}}$ will be diagonalized via 2D parallel Jacobi. This requires scattering $\hat{\M{A}}$, which is initially owned by ProcSlab$([I,J],[I,J],1)$, to some subset of the processors in ProcSlab$(:,[I,J],:)$, and subsequently gathering the diagonalization back to the top layer of the grid when 2D parallel Jacobi finishes. The picture to have in mind here is Figure~\ref{fig: scatter_and_gather_for_2.5D}, which shows this procedure for two choices of $\eta$. The details of this communication operation are nontrivial; in particular, the matrix $\hat{\M{A}}$ will need to be repartitioned to match the assumptions of Algorithm~\ref{alg:High_Level_Parallel_Jacobi}, where the block size will be dependent on the choice of $\eta$. The processors participating in this computation will similarly need to be re-indexed (again, to match the assumption that the processor grid is 2D and square). For simplicity, and since it won't be necessary to bound the complexity of these operations, we do not specify the fine details of this communication.

\begin{figure}[t]
    \centering
    \includegraphics[width=\linewidth]{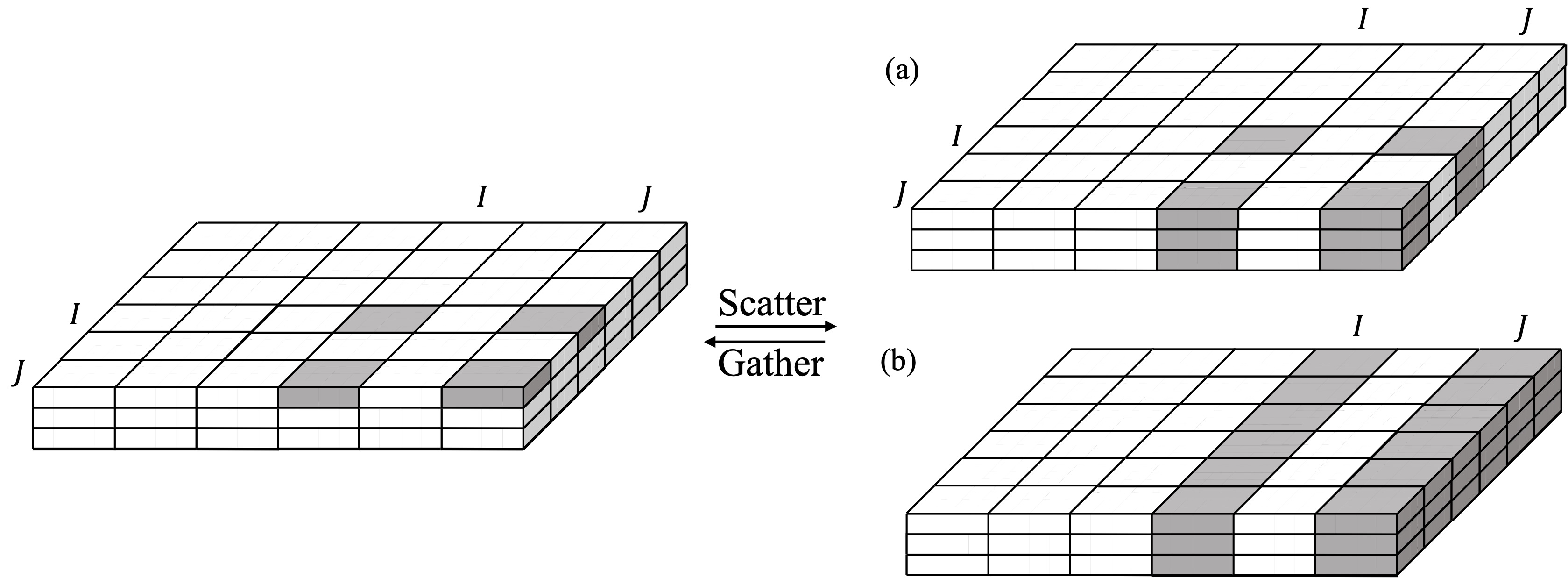}
    \caption{Illustration of the scatter and gather procedures for the coarse subproblem $\M{A}([I,J],[I,J])$. The subproblem, initially distributed over the shaded region on the left panel -- i.e., ProcSlab$([I,J],[I,J],1)$ -- is scattered to the shaded region on the right according to the utilization factor~$\eta$. Two representative cases are shown: (a) $\eta = 2\mu / \sqrt{P/c}$, where only processors within the coarse tile across all $c$ layers are used, that is, ProcSlab$([I,J],[I,J],:)$, and (b) $\eta = 1$, where all processors in the coarse columns $I$ and $J$ across all layers are utilized, i.e., ProcSlab$(:,[I,J],:)$.}
    \label{fig: scatter_and_gather_for_2.5D}
\end{figure}

\subsection{Latency and Bandwidth Optimality}

\begin{algorithm}
\caption{2.5D Parallel Jacobi for the Symmetric Eigenvalue Problem}
\label{alg:2.5D_Parallel_Jacobi}
\textbf{Input:} (1) $\M{A} \in {\mathbb R}^{n \times n}$ a symmetric matrix. \\
\phantom{Input2 } (2) $\M{Q} \in {\mathbb R}^{n \times n}$ an orthogonal matrix of approximate eigenvectors (optional). \\
\phantom{Input2 } (3) A set $\mathfrak{P}$ of $P$ processors, assumed to be arranged in a $\sqrt{P/c} \times \sqrt{P/c} \times c$ grid for replication \\ 
\phantom {Input2 (1) } factor $1 \leq c \leq P^{1/3}$ and labeled Proc$\langle p,q,r\rangle$ for $1 \leq p,q\leq \sqrt{P/c}$ and $1 \leq r \leq c$. \\
\phantom{Input2 } (4) $\mu$ a batching parameter satisfying $c \leq \mu \leq \sqrt{P/c}$. \\
\phantom{Input2 } (5) $\eta \in [2\mu/\sqrt{cP},1]$ a utilization factor. \\
\phantom{Input2 } (6) Lists $L$ and $\hat{L}$ of, respectively, $\sqrt{P/c}/\mu$ and $\sqrt{2 \eta \mu \sqrt{cP}}$ groups of pairwise disjoint indices $(I,J)$  \\
\phantom{Input2 (6) } (e.g.,  generated by two calls to Algorithm~\ref{alg:anti_diagonal_list}).

\vspace{2mm}
\textbf{Requires:} Initially, $\M{A}$ is stored on the top layer of the grid, with the coarse tile 
\[ \Mb{A}{IJ}=\Me{A}{(I{-}1)B{+}1:IB,(J{-}1)B{+}1:JB} \; \; \; \text{for} \; \; \; B = \mu n/\sqrt{P/c}, \]
block-distributed over the coarse processor tile ProcTile$(I,J,1)$ for $ 1\leq I,J \leq \sqrt{P/c}/\mu$, 
where each individual processor owns a $b \times b$ fine tile of $\M{A}$ for $b = n/\sqrt{P/c}$. If omitted, $\M{Q}$ is taken to be the identity matrix; in either case, it is partitioned and stored in the same way as $\M{A}$. 

\vspace{2mm}
\textbf{Ensure:}
On output, one sweep of (block) Jacobi has been applied to $\M{A}$ and $\M{Q}$.
\algrule
\setstretch{1.1}
\begin{algorithmic}[1]
\Function{$[\M{Q},\M{A}]=$
2.5D\_Parallel\_Jacobi}{$\M{A}$, $\M{Q}$, $\mathfrak{P}$, $\mu$, $\eta$, $L$, $\hat{L}$}
\For{\textbf{each group} $g\in L$}  \Comment{Outer serial loop}
  \For{\textbf{each}  $(I,J)\in g$ \textbf{in parallel}}
    \vspace{1mm}
    \State $\hat{\M{A}} = \begin{pmatrix} \M{A}_{II} & \M{A}_{IJ} \\ \M{A}_{JI} & \M{A}_{JJ} \end{pmatrix} $ \Comment{Coarse subproblem to be diagonalized}
    \vspace{1mm}
    \State  ProcSlab$([I,J],[I,J],1)$: Scatter $\hat{\M{A}}$ to $\hat{\mathfrak{P}}$, an $\eta$ fraction of processors in ProcSlab$(:,[I,J],:)$
    \State Initialize $\hat{\M{V}}$ as the $2B \times 2B$ identity matrix over $\hat{\mathfrak{P}}$
    \Repeat
        \State $[\hat{\M{V}}, \hat{\M{A}}]$ = \Call{2D\_Parallel\_Jacobi}{$\hat{\M{A}}$, $\hat{\M{V}}$, $\hat{\mathfrak{P}}$, $\hat{L}$} \label{line: 2D_call}
    \Until{converged or a maximum number of sweeps is reached}
    \State Gather $\hat{\M{V}}, \hat{\M{A}}$ from $\hat{\mathfrak{P}}$ back to ProcSlab$([I,J],[I,J],1)$
    \vspace{1mm}
    \State $\hat{\M{V}} = \begin{pmatrix} \hat{\M{V}}_{11} & \hat{\M{V}}_{12} \\ \hat{\M{V}}_{21} & \hat{\M{V}}_{22} \end{pmatrix}; \; \; \; \hat{\M{A}} = \begin{pmatrix} \hat{\M{A}}_{11} & \hat{\M{A}}_{12} \\ \hat{\M{A}}_{21} & \hat{\M{A}}_{22} \end{pmatrix}$ \Comment{Break $\hat{\M{V}}$ and $\hat{\M{A}}$ into $B \times B$ (coarse) blocks}
    \vspace{1mm}
    \State $\begin{pmatrix} \M{A}_{II} & \M{A}_{IJ} \\ \M{A}_{JI} & \M{A}_{JJ} \end{pmatrix} \gets \begin{pmatrix} \hat{\M{A}}_{11} & \hat{\M{A}}_{12} \\ \hat{\M{A}}_{21} & \hat{\M{A}}_{22} \end{pmatrix}$ \Comment{Update $\M{A}([I,J],[I,J])$ on the  corresponding processors}
    \vspace{1mm}
    \State ProcTile$(J,I,1)$: Send $\hat{\M V}_{21}$ to ProcTile$(I,I,1)$  \label{data_exchange_start} 
    \State ProcTile$(I,J,1)$: Send $\hat{\M V}_{12}$ to ProcTile$(J,J,1)$ 
    \For{$K = 1:\sqrt{P/c}/\mu$ \textbf{in parallel}}
        \If{$K \neq I,J$}
        \State ProcTile$(I,K,1)$ and ProcTile$(J,K,1$): Exchange $\M{A}_{IK} \leftrightarrow \M{A}_{JK}$ 
        \State ProcTile$(K,I,1)$ and ProcTile$(K,J,1$): Exchange $\M{A}_{KI} \leftrightarrow \M{A}_{KJ}$ 
        \EndIf
        \State ProcTile$(K,I,1)$ and ProcTile$(K,J,1)$: Exchange $\M{Q}_{KI} \leftrightarrow \M{Q}_{KJ}$
    \EndFor
    \Comment{Continued on next page $\rightarrow $} \algstore{bkbreak} \end{algorithmic} 
\end{algorithm}

\begin{algorithm}
\begin{algorithmic} \algrestore{bkbreak}

    \State \Comment{Continued from previous page}
    \vspace{1mm}
    
     \State ProcTile$(I,I,1)$: Broadcast $[\hat{\M V}_{11}, \hat{\M V}_{21}]$ to ProcTile$(I,K,1)$ and ProcTile$(K,I,1)$ for all $K \neq I$ \label{broadcast_start}
    \State ProcTile$(J,J,1)$: Broadcast $[\hat{\M V}_{12}, \hat{\M V}_{22}]$ to ProcTile$(J,K,1)$ and ProcTile$(K,J,1)$ for all $K \neq J$ 
    \vspace{1mm}

    \For{$K = 1:n/B$ \textbf{in parallel}} \Comment{Update (coarse) block rows/columns of $\M{A}$ and $\M{Q}$.} \label{line: start_update}
        \If{$K \neq I,J$} 

        \State ProcTile$(I,K,1)$: Swap off-diagonal blocks $\hat{\M{V}}_{11} \gets \hat{\M{V}}_{11}^T$ and $\hat{\M{V}}_{21} \gets \hat{\M{V}}_{21}^T$
        \State \hspace{2.5cm} $\M{M}_1 =$ \Call{2.5D\_MM}{$\hat{\M{V}}_{11}$, $\M{A}_{IK}$, ProcSlab$(I,K,:)$}
        \
        \State \hspace{2.5cm} $\M{M}_2 = $ \Call{2.5D\_MM}{$\hat{\M{V}}_{21}$, $\M{A}_{JK}$, ProcSlab$(I,K,:)$}
        \State \hspace{2.5cm} $\M{A}_{IK} \gets  \M{M}_1 + \M{M}_2$ 
        \vspace{1mm}

        \State ProcTile$(J,K,1)$: Swap off-diagonal blocks $\hat{\M{V}}_{12} \gets \hat{\M{V}}_{12}^T$ and $\hat{\M{V}}_{22} \gets \hat{\M{V}}_{22}^T$
        \State \hspace{2.5cm} $\M{M}_1 =$ \Call{2.5D\_MM}{$\hat{\M{V}}_{12}$, $\M{A}_{IK}$, ProcSlab$(J,K,:)$}
        \
        \State \hspace{2.5cm} $\M{M}_2 = $ \Call{2.5D\_MM}{$\hat{\M{V}}_{22}$, $\M{A}_{JK}$, ProcSlab$(J,K,:)$}
        \State \hspace{2.5cm} $\M{A}_{JK} \gets  \M{M}_1 + \M{M}_2$ 
        \vspace{2mm}

        \State ProcTile$(K,I,1)$: $\M{M}_1 =$ \Call{2.5D\_MM}{$\M{A}_{KI}$, $\hat{\M{V}}_{11}$, ProcSlab$(K,I,:)$}
        \State \hspace{2.5cm} $\M{M}_2 = $ \Call{2.5D\_MM}{$\M{A}_{KJ}$, $\hat{\M{V}}_{21}$, ProcSlab$(K,I,:)$}
        \State \hspace{2.5cm} $\M{A}_{KI} \gets  \M{M}_1 + \M{M}_2$ 
        \vspace{2mm}

         \State ProcTile$(K,J,1)$: $\M{M}_1 =$ \Call{2.5D\_MM}{$\M{A}_{KI}$, $\hat{\M{V}}_{12}$, ProcSlab$(K,J,:)$}
        \State \hspace{2.5cm} $\M{M}_2 = $ \Call{2.5D\_MM}{$\M{A}_{KJ}$, $\hat{\M{V}}_{22}$, ProcSlab$(K,J,:)$}
        \State \hspace{2.5cm} $\M{A}_{KJ} \gets  \M{M}_1 + \M{M}_2$ 
        \EndIf
        \vspace{1mm}

        \State ProcTile$(K,I,1)$: $\M{M}_1 =$ \Call{2.5D\_MM}{$\M{Q}_{KI}$, $\hat{\M{V}}_{11}$, ProcSlab$(K,I,:)$} 
        \State \hspace{2.5cm} $\M{M}_2 = $ \Call{2.5D\_MM}{$\M{Q}_{KJ}$, $\hat{\M{V}}_{21}$, ProcSlab$(K,I,:)$}
        \State \hspace{2.5cm} $\M{Q}_{KI} \gets  \M{M}_1 + \M{M}_2$ 
        \vspace{2mm}

        \State ProcTile$(K,J,1)$: $\M{M}_1 =$ \Call{2.5D\_MM}{$\M{Q}_{KI}$, $\hat{\M{V}}_{12}$, ProcSlab$(K,J,:)$}
        \State \hspace{2.5cm} $\M{M}_2 = $ \Call{2.5D\_MM}{$\M{Q}_{KJ}$, $\hat{\M{V}}_{22}$, ProcSlab$(K,J,:)$}
        \State \hspace{2.5cm} $\M{Q}_{KJ} \gets  \M{M}_1 + \M{M}_2$ 

    \EndFor
  \EndFor
\EndFor
\EndFunction
\end{algorithmic}
\end{algorithm}

Algorithm~\ref{alg:2.5D_Parallel_Jacobi} presents pseudocode for our 2.5D version of parallel Jacobi. Like the 2D algorithm, the pseudocode here describes only a single sweep of the algorithm, which again proceeds according to the anti-diagonal block ordering discussed in Section~\ref{section: 2D_parallel} (e.g., generated by Algorithm~\ref{alg:anti_diagonal_list}). Accordingly, the convergence discussion from the previous section carries over here, where we again omit a pivoting step -- and a check on each subproblem -- for simplicity. Note that, in this case, we could use the 2.5D version of LUPP due to Solomonik and Demmel \cite{SD11} to guarantee convergence while leveraging the extra memory available in the 2.5D setting.  \\
\indent As outlined above, Algorithm~\ref{alg:2.5D_Parallel_Jacobi} can be interpreted as running 2D parallel Jacobi over the top layer of the grid, only calling on the remaining layers for subproblem diagonalization (line~\ref{line: 2D_call}) and the subsequent updates to $\M{A}$ and $\M{Q}$ (beginning in line~\ref{line: start_update}). The latter employs 2.5D matrix multiplication with explicit calls to Algorithm~\ref{alg:2.5DMM}. Each such call is executed by the processors on the top layer of the grid that both own the matrices being multiplied and will store their product on output.

\begin{remark} {\normalfont 
   In the pseudocode of Algorithm~\ref{alg:2.5D_Parallel_Jacobi}, the (coarse) block row/column updates to $\M{A}$ and $\M{Q}$ are broken into 2.5D matrix multiplications of size $B \times B$. In practice, it may be more efficient to, for example, allow all processors in ProcSlab$([I,J],:,:)$ to collaborate on the update to coarse rows $I$ and $J$ of $\M{A}$, which could be done as one large (rectangular) 2.5D matrix multiplication. Moreover, this would allow us to relax the condition $c \leq \mu$, which is currently necessary to ensure that the processor grid used for the aforementioned $B \times B$ multiplications satisfies the requirements of Algorithm~\ref{alg:2.5DMM}. We choose not to do this in the pseudocode presented here for the following reasons:\ (1) it does not change the asymptotic complexity of the algorithm and (2) it both makes for easier comparison with Algorithm~\ref{alg:High_Level_Parallel_Jacobi} and allows us to avoid introducing additional notation.}
\end{remark}

We now analyze the complexity of Algorithm~\ref{alg:2.5D_Parallel_Jacobi} for arbitrary choices of the batching parameter $\mu \in [c,\sqrt{P/c}]$ and the utilization factor $\eta \in [2\mu/\sqrt{cP},\,1]$.

\begin{theorem}\label{thm: 2.5D_parallel_comp} 
    One sweep of 2.5D parallel Jacobi (i.e., Algorithm~\ref{alg:2.5D_Parallel_Jacobi}) has complexity
    $$
        \alpha \cdot O \left( X \cdot\log \left( \eta \mu \sqrt{cP} \right) + \frac{\sqrt{P/c}}{\mu} \log \left( \frac{\sqrt{P/c}}{\mu} \right)\right) + \beta \cdot O \left(\frac{n^2}{X} + \frac{n^2}{\sqrt{cP}}\right) + \gamma \cdot O\left( \frac{n^3}{X^2} + \frac{n^3}{P} \right)
    $$
    for $X = \sqrt[4]{\frac{\eta^2 P^3}{\mu^2c}}$, assuming that $O(1)$ sweeps of 2D parallel Jacobi are applied to each subproblem. 
\end{theorem}
\begin{proof}
      We summarize below the work associated with each index pair $(I,J) \in g$. For reference, recall that communication operations between coarse processor tiles follow Figure~\ref{fig:tile_comm}.
      \begin{enumerate}
          \item Line 5 (Scatter step): The $4 \mu^2$ processors in ProcSlab$([I,J],[I,J],1)$ scatter the $2B \times 2B$ subproblem $\hat{\M{A}}$ to $2\eta\mu\sqrt{cP}$ processors $\mathfrak{P}$ in ProcSlab$(:,[I,J],:)$. Since processors on both the sending and receiving end of this operation own contiguous blocks of $\hat{\M{A}}$, each individual processor in ProcSlab$([I,J],[I,J],1)$ will scatter its $b \times b$ block of $\hat{\M{A}}$ to $O\left(\frac{2 \eta \mu \sqrt{cP}}{4 \mu^2} \right) = O\left( \frac{\eta \sqrt{cP}}{\mu} \right)$ processors in $\mathfrak{P}$. Recalling Table~\ref{tab:mpi_routines}, this has communication complexity
          \begin{equation} \label{eqn: scatter_cost}
              \alpha \cdot O\left( \log \left( \frac{\eta \sqrt{cP}}{\mu} \right)\right) + \beta \cdot O\left( b^2 \right).
          \end{equation}
          \item Line 8 (Inner diagonalization): $2\eta\mu\sqrt{cP}$ processors apply one sweep of 2D parallel Jacobi to the $2B \times 2B$ matrix $\hat{\M{A}}$. By Theorem~\ref{thm: parallel_comp}, the complexity of this operation is
          \begin{equation}\label{eqn: cost_of_call_to_2D}
          \alpha \cdot O \left( \sqrt{\eta \mu\sqrt{cP}} \log \left( \eta \mu \sqrt{cP} \right) \right) + \beta \cdot O \left( \frac{B^2}{\sqrt{\eta \mu \sqrt{cP}}} \right) + \gamma \cdot O\left( \frac{B^3}{\eta \mu \sqrt{cP}} \right).
          \end{equation}
          Moreover, this bounds the cost of the loop in lines 7-9 under the assumption that only $O(1)$ calls to 2D parallel Jacobi are made.
          \item Line 10 (Gather step): Each processor in $\mathfrak{P}$ sends its blocks of $\hat{\M{A}}$ and $\hat{\M{V}}$ to $O(1)$ processors in ProcSlab$([I,J],[I,J],1)$. Making the same argument as in item one, the cost of this operation is given by \eqref{eqn: scatter_cost}.
          \item Lines 13-14 (Rotation-block send): Each processor in ProcTile$(J,I,1)$ and ProcTile$(I,J,1)$ sends one $b \times b$ block of $\hat{\M{V}}$ to a unique processor in either ProcTile $(I,I,1)$ or ProcTile$(J,J,1)$. This has complexity $\alpha \cdot O(1) + \beta \cdot O(b^2)$.
          \item Lines 15-21 (Block exchange): Each processor in coarse processor row/column $I$ and $J$ exchanges $b \times b$ blocks of $\M{A}$ and $\M{Q}$ with up to two processors. Again the cost is $\alpha \cdot O(1) + \beta \cdot O(b^2)$.
          \item Lines 23-24 (Rotation broadcast): Each processor in ProcTile$(I,I,1)$ and ProcTile$(J,J,1)$ broadcasts two $b \times b$ blocks of $\hat{\M{V}}$ to $O \left( \frac{\sqrt{P/c}}{\mu} \right)$ processors (that is, one in each processor tile in the same coarse row/column). Accordingly, the cost of this operation is 
          \begin{equation}
              \alpha \cdot  O\left(\log \left( \frac{\sqrt{P/c}}{\mu} \right) \right) + \beta \cdot O(b^2).
          \end{equation}
          \item Lines 25-48 (Right/left updates): Each processor in coarse processor row/column $I$ and $J$ participates in at most four 2.5D matrix multiplications, each of which computes the product of two $B \times B$ matrices over a $\mu \times \mu \times c$ processor grid. By Theorem~\ref{thm: 2.5d_MM}, these multiplications have complexity
          \begin{equation}
              \alpha \cdot O \left( \frac{\mu}{c} + \log c \right) + \beta \cdot O \left( \frac{B^2}{\mu c} \right)  + \gamma \cdot O \left( \frac{B^3}{\mu^2 c} \right)
          \end{equation}
          This dominates the cost associated with both (1) taking the transpose of blocks of $\hat{\M{V}}$ when necessary and (2) summing  the resulting $B \times B$ matrices $\M{M}_1$ and $\M{M}_2$.
      \end{enumerate}
      Altogether, we obtain the following bounds on the latency, bandwidth, and arithmetic costs of handling one pair $(I,J) \in g$.
      \begin{equation}\label{eqn: separate_cost}
        \aligned 
          \text{Latency:}& \; \; \alpha \cdot O\left( \log \left( \frac{\eta \sqrt{cP}}{\mu} \right) + \sqrt{\eta \mu \sqrt{cP}} \log \left( \eta \mu \sqrt{cP} \right) + \log \left( \frac{\sqrt{P/c}}{\mu} \right) + \frac{\mu}{c} + \log c \right)  \\
          \text{Bandwidth:}& \; \; \beta \cdot O \left( b^2 + \frac{B^2}{\sqrt{\eta \mu \sqrt{cP}}} + \frac{B^2}{\mu c} \right) \\
          \text{Arithmetic:}& \; \; \gamma \cdot O \left( \frac{B^3}{\eta \mu \sqrt{cP}} + \frac{B^3}{\mu^2 c} \right)
          \endaligned 
      \end{equation}
    To simplify, recall $\eta \geq 2 \mu/\sqrt{cP}$, which implies that $\sqrt{\eta \mu \sqrt{cP}} \log \left( \eta \mu \sqrt{cP} \right) $ dominates all but (possibly) the third term in the latency bound.  Additionally, we have $\frac{B^2}{\mu c} = \frac{\mu^2 b^2}{\mu c} = \frac{\mu b^2}{c} \geq b^2$ since $\mu \geq c$. Hence, \eqref{eqn: separate_cost} can be written cumulatively as 

    \begin{equation}
      \alpha \cdot O\left( \sqrt{\eta \mu \sqrt{cP}} \log \left( \eta \mu \sqrt{cP} \right) + \log \left( \frac{\sqrt{P/c}}{\mu} \right) \right) + \beta \cdot O \left( \frac{B^2}{\sqrt{\eta \mu \sqrt{cP}}} + \frac{B^2}{\mu c} \right) +  \gamma \cdot O \left( \frac{B^3}{\eta \mu \sqrt{cP}} + \frac{B^3}{\mu^2 c} \right)
    \end{equation}

    Since each processor participates in some part of the computation for at most two pairs $(I,J)$ in $g$, this bounds the complexity of one parallel step of the algorithm (that is, corresponding to each group of indices in the list $L$). There are $|L| = \sqrt{P/c}/\mu$ such steps in total, so we conclude that the complexity of one sweep of 2.5D parallel Jacobi is
    \begin{equation}
    \aligned
    \alpha \cdot O \left( \sqrt{ \frac{\eta P^{3/2}}{\mu \sqrt{c}}} \log \left( \eta \mu \sqrt{cP} \right) + \frac{\sqrt{P/c}}{\mu} \log \left( \frac{\sqrt{P/c}}{\mu} \right)\right) &+ \beta \cdot O \left(\frac{n^2}{\sqrt{\eta P^{3/2}/\mu \sqrt{c}}} + \frac{n^2}{\sqrt{cP}}\right) \\
    & \hspace{-4cm} + \gamma \cdot O\left( \frac{n^3}{\eta P^{3/2}/\mu \sqrt{c}} + \frac{n^3}{P} \right).
    \endaligned 
    \end{equation}
    Here, we further simplify by recalling $B = \mu n/\sqrt{P/c}$. The result now follows by taking $X = \sqrt[4]{\frac{\eta^2 P^{3}}{\mu^2 c}}$.
\end{proof} \\

Clearly, the complexity of 2.5D parallel Jacobi depends dramatically on the choice of batching parameter $\mu \in [c,\sqrt{P/c}]$ and utilization factor $\eta \in [ 2\mu/\sqrt{cP},1].$ Indeed, both the bandwidth and latency costs of Algorithm~\ref{alg:2.5D_Parallel_Jacobi} may exceed or beat that of 2D parallel Jacobi for particular choices of $\mu$ and $\eta$. To illustrate this -- and to make Theorem~\ref{thm: 2.5D_parallel_comp} more interpretable -- we exhibit below a handful of extreme cases. The takeaway is that optimal values of latency and bandwidth cannot be achieved simultaneously in Algorithm~\ref{alg:2.5D_Parallel_Jacobi}.

\subsubsection{2.5D Latency Optimal Jacobi}
We start by noting that the dominant term in the latency bound from Theorem~\ref{thm: 2.5D_parallel_comp} increases with $\eta$. Thus, latency is minimized by taking $\eta$ as small as allowed. If we take the minimum value $\eta = 2\mu/\sqrt{cP}$, which corresponds to using only the $4\mu^2$ processors in ProcSlab$([I,J],[I,J],1)$ in the calls to 2D parallel Jacobi, the sweep-complexity of Algorithm~\ref{alg:2.5D_Parallel_Jacobi} becomes
\begin{equation}
     \alpha \cdot O \left( \sqrt{\frac{P}{c}} \log(\mu^2) + \frac{\sqrt{P/c}}{\mu} \log \left( \frac{\sqrt{P/c}}{\mu} \right)\right) + \beta \cdot O \left( \frac{n^2}{\sqrt{P/c}} \right) + \gamma \cdot O\left(\frac{n^3}{P/c} \right).
 \end{equation}
This is not particularly satisfying; while the latency cost decreased (relative to Theorem~\ref{thm: parallel_comp}) both the bandwidth \textit{and} arithmetic costs went up. In fact, this is essentially the complexity of running the 2D algorithm over only the top layer of the grid. \\
\indent A more reasonable choice would be $\eta = \mu\sqrt{c/P}$, which is the smallest value of $\eta$ for which  the arithmetic cost of Algorithm~\ref{alg:2.5D_Parallel_Jacobi} is $O(n^3/P)$. In this setting, the total complexity is 
\begin{equation}\label{eqn: latency_optimal}
     \alpha \cdot O \left( \sqrt{P} \log(\mu^2c) + \frac{\sqrt{P/c}}{\mu} \log \left( \frac{\sqrt{P/c}}{\mu} \right)\right) + \beta \cdot O \left( \frac{n^2}{\sqrt{P}} \right) + \gamma \cdot O\left(\frac{n^3}{P} \right).
 \end{equation}
Since $\mu^2c\le P$, the logarithmic factor is no larger than that of the 2D algorithm up to constants. Hence, this choice preserves the optimal arithmetic cost while modestly improving latency relative to the 2D bound (and leaving bandwidth unchanged).
Note that $\eta = \mu\sqrt{c/P}$ corresponds to using only processors that lie directly below the processor slab that owns $\hat{\M{A}}$ in line 8. 

\subsubsection{2.5D Bandwidth Optimal Jacobi}
In contrast to latency, the bandwidth cost in Theorem~\ref{thm: 2.5D_parallel_comp} decreases as $\eta$ increases. Consequently, the best obtainable bound is always at least $\Omega(n^2/\sqrt{cP})$. If we take $\eta = 1$, equivalently if we use all available processors in the subproblem diagonalization (i.e., plot (b) of Figure~\ref{fig: scatter_and_gather_for_2.5D}), the complexity of Algorithm~\ref{alg:2.5D_Parallel_Jacobi} becomes
\begin{equation}
    \alpha \cdot O \left(\sqrt{\frac{P^{3/2}}{\mu \sqrt{c}}}\log \left( \mu \sqrt{cP} \right) + \frac{\sqrt{P/c}}{\mu} \log \left( \frac{\sqrt{P/c}}{\mu} \right) \right) + \beta \cdot O \left( \frac{n^2}{\sqrt{P^{3/2}/\mu \sqrt{c}}} + \frac{n^2}{\sqrt{cP}} \right)+ \gamma \cdot O\left( \frac{n^3}{P} \right)
\end{equation}
Subsequently setting $\mu = \Theta(\sqrt{P/c^3})$ yields optimal bandwidth for 2.5D parallel Jacobi (at the cost of an increase in latency). We summarize this in the following corollary.
\begin{corollary}\label{cor: bandwidth_optimal_jacobi}
    Suppose we run 2.5D parallel Jacobi (Algorithm~\ref{alg:2.5D_Parallel_Jacobi}) with the following:
    \begin{enumerate}
        \item $\mu = \Theta(\sqrt{P/c^3})$.
        \item $\eta = 1$, meaning all available processors are utilized in each call to 2D parallel Jacobi.
    \end{enumerate}
    Then the complexity of each sweep is
    \[ \alpha \cdot O \left(\sqrt{cP} \log \left( \frac{P}{c} \right) \right) + \beta \cdot O \left( \frac{n^2}{\sqrt{cP}}\right) + \gamma \cdot O\left(\frac{n^3}{P} \right), \]
    again assuming that only $O(1)$ sweeps of the 2D algorithm are applied to each subproblem.
\end{corollary}

Note that the optimal value $\mu = \Theta(\sqrt{P/c^3})$ implies $B = \Theta(n/c)$ and $P \geq c^5$ since $\mu \geq c$. The latter indicates that the bandwidth-optimal algorithm operates on a true 2.5D rather than 3D processor grid.
This parameter choice attains the 2.5D matrix-multiplication bandwidth bound $O(n^2/\sqrt{cP})$ and is therefore bandwidth-optimal under the stated lower-bound model. \\
\indent Of course, in practice a user may not want to use either of the ``extremal" versions of 2.5D parallel Jacobi represented by Corollary~\ref{cor: bandwidth_optimal_jacobi} and \eqref{eqn: latency_optimal}. With this in mind, Figure~\ref{fig:heatmap} shows how the latency/bandwidth costs of Algorithm~\ref{alg:2.5D_Parallel_Jacobi} vary with $\mu$ and $\eta$, specifically in comparison to 2D parallel Jacobi. Note that the regions on these plots where the bandwidth/latency costs of Algorithm~\ref{alg:2.5D_Parallel_Jacobi} improve on those of Algorithm~\ref{alg:High_Level_Parallel_Jacobi} are disjoint. For a more detailed discussion of optimal parameter selection in 2.5D parallel Jacobi, see Appendix~\ref{appendix: optimization}.

\begin{figure}[t]
    \centering
    \includegraphics[width=\linewidth]{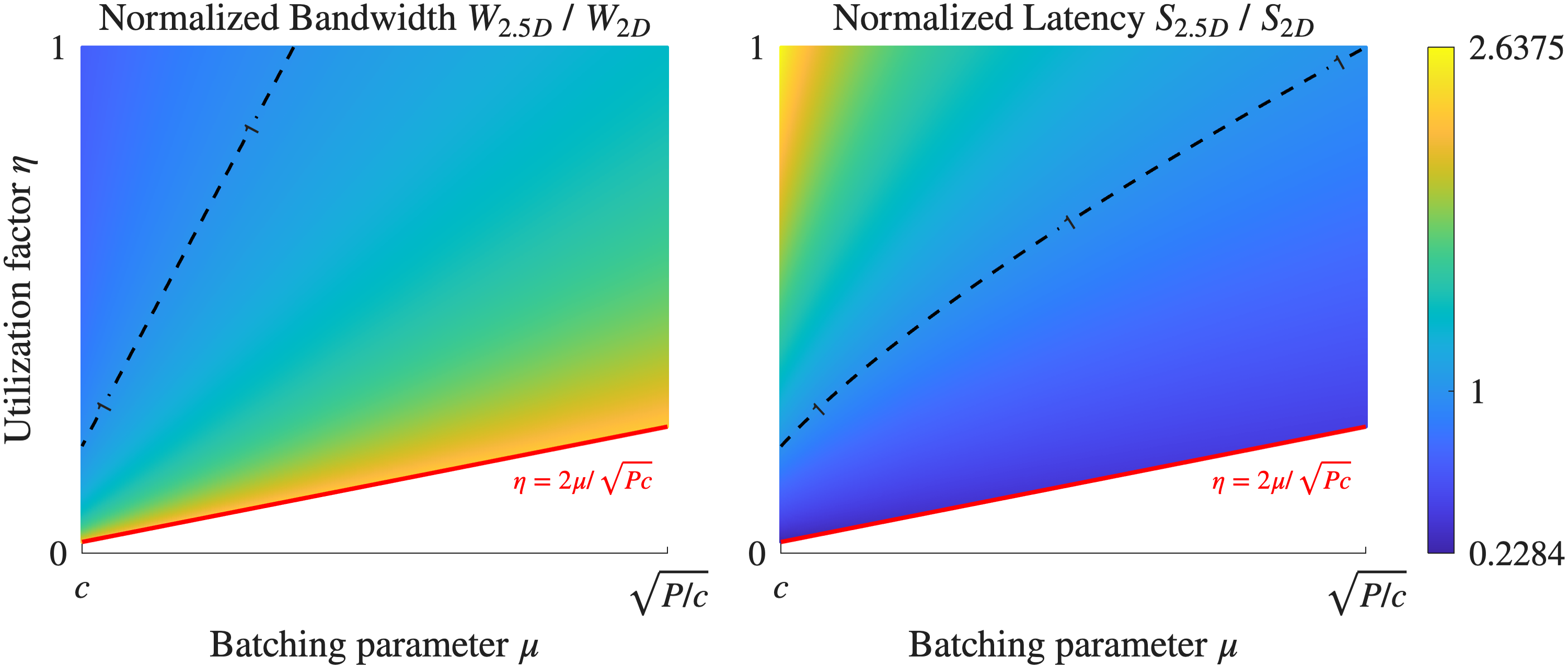}
    \caption{Comparison of bandwidth/latency bounds for 2D and 2.5D parallel Jacobi as a function of the batching parameter $\mu$ and utilization factor $\eta$. In both cases, we record the magnitude of the ratio of the terms in Theorems~\ref{thm: parallel_comp} and~\ref{thm: 2.5D_parallel_comp} (without taking into account hidden constants), which we note are independent of the problem size $n$. In both plots, $c = 8$ and $P = 2^{16}$. We also plot the contour where the normalized bandwidth/latency is 1 (i.e., the black, dashed curve) to mark the point at which 2.5D Jacobi becomes more efficient.}
    \label{fig:heatmap}
\end{figure}

\section{Lower Bounds for Parallel Jacobi}\label{section: lower_bound}
\begin{figure}[t]
    \centering
    \begin{subfigure}{.4\linewidth}
        \centering
        \includegraphics[width=\linewidth]{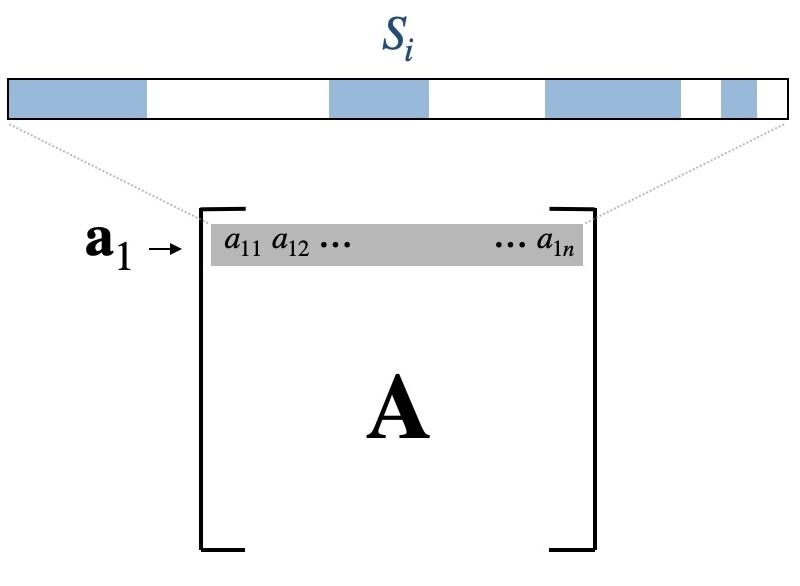}
        \caption{}
    \end{subfigure}\hspace{2cm}%
    \begin{subfigure}{.35\linewidth}
        \centering
        \includegraphics[width=\linewidth]{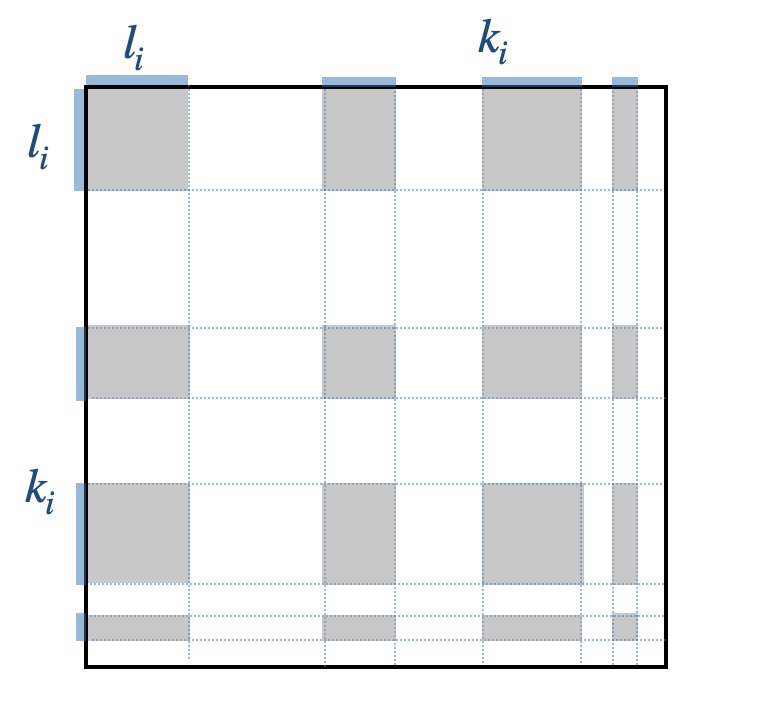}
        \caption{}
    \end{subfigure}
\caption{Step $i$ in the critical path used to prove Theorem~\ref{thm: parallel_lower_bound}: (a) presents the subset $S_i$ of ${\bf a}_1$ zeroed out while (b) highlights the corresponding $(l_i+k_i) \times (l_i+k_i)$ subproblem of ${\bf A}$ to be diagonalized.}
    \label{fig: 2.5D_lower_bound}
\end{figure}
In each set of complexity bounds derived above (as well as in Section~\ref{section: 2D_parallel}), the product of the latency and bandwidth costs was, up to log factors, $O(n^2)$. For 2.5D parallel Jacobi, this stemmed from the fact that any decrease in bandwidth/latency (relative to Theorem~\ref{thm: parallel_comp}) accompanied a proportional increase in the other, suggesting an unavoidable latency/bandwidth tradeoff. In this section, we codify this relationship by proving the lower bound $W(n,P) \cdot S(n,P) = \Omega(n^2)$ for any parallel Jacobi algorithm that satisfies the following:
\begin{enumerate}
    \item Eventually -- e.g., after possibly executing some amount of recursion, following \cite[Section 4]{arxiv_manuscript} -- the algorithm diagonalizes subproblems of $\M{A}$ on individual processors, using any $O(n^3)$ routine to do so.
    \item After each of these ``direct" diagonalizations, the corresponding processor must communicate with at least one other processor before row/column updates can be applied. 
\end{enumerate}

These are mild assumptions, which we note are satisfied by both the 2D and 2.5D version of parallel Jacobi presented in this paper. For these algorithms, we also exhibit a similar tension between latency and arithmetic, captured by the lower bound $F(n,P)\cdot S(n,P)^2 = \Omega(n^3)$. In doing so, we add parallel Jacobi to the list of numerical linear algebra computations for which these tradeoffs have been established \cite{SD11,solomonik2017trade}. 

\begin{theorem}\label{thm: parallel_lower_bound}
    Let $F(n,P)$, $W(n,P)$, and $S(n,P)$ be the arithmetic, bandwidth, and latency costs, respectively, of one sweep of any parallel Jacobi algorithm that satisfies the two assumptions listed above (when applied to an $n \times n$ real symmetric matrix $\M{A}$ using $P$ processors). We have the following:
    $$ W(n,P) \cdot S(n,P) = \Omega(n^2) \; \; \; \text{and} \; \; \; F(n,P) \cdot S(n,P)^2 = \Omega(n^3).$$
\end{theorem}
\begin{proof}
    Consider the first row ${\bf a}_1 = \begin{pmatrix} a_{11} & a_{12} & \cdots & a_{1n} \end{pmatrix}$ of {\bf A}. In one sweep, parallel Jacobi must visit each of $a_{11}, \ldots, a_{1n}$  (possibly in groups/with entries from other rows) and diagonalize the corresponding subproblem of ${\bf A}$, which we assume takes place within the fast memory of a single processor. In doing so, each algorithm partitions $a_{11}, \ldots, a_{1n}$ into $d$ (possibly non-disjoint) 
    sets $S_1, \ldots S_d$, where $S_i$ contains the entries of ${\bf a}_1$ annihilated after the $i$-th single-processor diagonalization involving the first row of $\M{A}$. \\
    \indent Let $l_i$ be the largest index such that $a_{11}, a_{12}, \ldots, a_{1l_i}$ belongs to $S_i$ and set $k_i = |S_i|-l_i$. Clearly, $l_i \geq 1$ and  $k_i \geq 1$ for all $1 \leq i \leq d $, and $n \leq \sum_{i=1}^d l_i+k_i$.
   Moreover, the sets $S_1, \ldots S_d$ define a $d$ step path within the algorithm, where at step $i$ a subproblem that contains $S_i$ is diagonalized (see Figure~\ref{fig: 2.5D_lower_bound}). Since only one collection of entries in the first row can be zeroed out at a time, this path is serial. Moreover, each step has the following costs
    \begin{itemize}
        \item[(a)] Arithmetic: $\Omega((l_i+k_i)^3)$ 
        \item[(b)] Bandwidth: $\Omega((l_i+k_i)^2)$, assuming the diagonalization must be communicated to at least one other processor. 
        \item[(c)] Latency: $\Omega(1)$.
    \end{itemize}
    We therefore have \ $F(n,P) = \sum_{i=1}^d \Omega((l_i+k_i)^3)$,  $W(n,P) = \sum_{i=1}^d \Omega((l_i+k_i)^2)$, and $S(n,P) = \Omega(d).$
    Now applying Cauchy-Schwarz yields 
    \begin{equation}
        W(n,P) \cdot S(n,P) = \Omega \left( d \sum_{i=1}^d (l_i+k_i)^2 \right) \geq \Omega\left( \left[ \sum_{i=1}^d l_i+k_i \right]^2 \right) \geq \Omega(n^2).
    \end{equation}
    The analogous result for latency and arithmetic follows similarly from Hölder's inequality, which implies $F(n,P)^{1/3} \cdot S(n,P)^{2/3} \geq \Omega(n)$. 
\end{proof} \\

Theorem~\ref{thm: parallel_lower_bound} is flexible in that it allows variable blocking strategies. If we assume that the size of each subproblem  diagonalized directly is fixed and equal to $2b \times 2b$, the complexity bounds simplify to 
\begin{equation}
    F(n,P) = \Omega(db^3), \; \; \; \;  W(n,P) = \Omega(db^2), \; \; \; \; \text{and} \; \; \; \;S(n,P) = \Omega(d). 
\end{equation}
By taking $d = \sqrt{P}$ and $b = n/\sqrt{P}$, these recover the complexity bounds of 2D parallel Jacobi (up to logarithmic factors): 
\begin{equation}
    F(n,P)=\Omega(n^3/P),\qquad
    W(n,P)=\Omega(n^2/\sqrt{P}),\qquad
    S(n,P)=\Omega(\sqrt{P}).
\end{equation}
Similarly, setting $d = \sqrt{cP}$ and $b = n/\sqrt{cP}$ yields 
\begin{equation} \label{eqn: second_lower_bounds}
    F(n,P)=\Omega(n^3/(cP)),\qquad
    W(n,P)=\Omega(n^2/\sqrt{cP}),\qquad
    S(n,P)=\Omega(\sqrt{cP}),
\end{equation}
which matches the bandwidth and latency of bandwidth-optimal 2.5D parallel Jacobi.\footnote{It may not be obvious given the complicated pseudocode, but $b = n/\sqrt{cP}$ is the block size used by 2D parallel Jacobi in line 8 of Algorithm~\ref{alg:2.5D_Parallel_Jacobi} when $\eta = 1$ and $\mu = \sqrt{P/c^3}$, after re-distributing $\hat{\M{A}}$ and $\hat{\M{V}}$.} Of course, the arithmetic lower bound of \eqref{eqn: second_lower_bounds} can never be attained. 

\section{2.5D Parallel Jacobi SVD}\label{section: parallel_jacobi_svd}

Like all algorithms for the symmetric eigenvalue problem, Jacobi's method can be applied to an arbitrary matrix $\M{G} \in {\mathbb R}^{m \times n}$ to obtain a singular value decomposition (SVD). The underlying connection is straightforward:\ diagonalizing the Gram matrix $\M{G}^T\M{G}$, e.g., via Jacobi, yields both the right singular vectors of $\M{G}$ and their corresponding singular values. Accordingly, we can use the results of the previous section to obtain a 2.5D parallel SVD algorithm, which we detail here. Throughout, we assume $m \geq n$ for simplicity. \\
\indent The specific version of Jacobi SVD that we consider parallelizing is presented here (in serial) as Algorithm~\ref{alg:Jacobi_SVD}. This routine runs block Jacobi (i.e., Algorithm~\ref{alg: block Jacobi}) on $\M{G}^T\M{G}$ implicitly,\footnote{In general, we prefer not to form $\M{G}^T\M{G}$ explicitly since $\kappa_2(\M{G}^T\M{G}) = \kappa_2(\M{G})^2$, where $\kappa_2(\cdot)$ is the spectral norm condition number.} forming each submatrix on the fly from (block) columns of $\M{G}$ before diagonalizing them and using the resulting eigenvector matrix to perform a column update. Each of these updates forces the corresponding columns to be mutually orthogonal; as a result, $\M{G}$ gradually converges to $\M{U} {\bf \Sigma}$, where $\M{U}$ is a matrix of left singular vectors and ${\bf \Sigma}$ is a diagonal matrix of corresponding singular values, with the associated right singular vectors accumulated in $\M{V}$. Since in this approach orthogonal transformations are applied only to the columns of $\M{G}$ -- as opposed to the columns \textit{and} rows as in standard Jacobi -- Algorithm~\ref{alg:Jacobi_SVD} is usually referred to as ``one-sided" Jacobi SVD. For a discussion of other versions, see \cite[Section 5]{arxiv_manuscript}. 

\begin{algorithm}[h]
\caption{One-Sided Jacobi SVD}
\label{alg:Jacobi_SVD} \begin{algorithmic}[1]
\Require $\mathbf{G} \in {\mathbb R}^{m \times n}$ with $m \geq n$ and columns partitioned as follows (for $b$ that divides $n$):
\[ \mathbf{G}(:,I) = \mathbf{G}(:, (I-1)b+1:Ib) \; \; \; \text{for} \; \; \; 1 \leq I \leq n/b. \]
 
\Ensure On output, $\mathbf{G} = \mathbf{U} {\bm \Sigma}\mathbf{V}^T$ is an approximate (reduced) singular value decomposition of $\mathbf{G}$, with $\mathbf{U} \in {\mathbb R}^{m \times n}$ and ${\bm \Sigma},\mathbf{V} \in {\mathbb R}^{n \times n}$. ${\bm \Sigma}$ is diagonal and contains approximate singular values of $\mathbf{G}$.
\algrule
\setstretch{1.1}
\Function{$[\mathbf{U},{\bm \Sigma},\mathbf{V}]=$
Jacobi\_SVD}{$\mathbf{G}$} \State $\mathbf{V}=\mathbf{I}_n$
\Repeat \For{all $1 \leq I < J \leq n/b $ in some order,} \State $\mathbf{\hat{A}}= \mathbf{G}(:, [I,J])^T \mathbf{G}(:,[I,J])$ \Comment{$\mathbf{\hat{A}}$ is a $2b \times 2b$ submatrix of $\mathbf{G}^T\mathbf{G}$ formed on the fly} \If{$\mathbf{\hat{A}}$ is far
enough from diagonal} \State $\mathbf{\hat{A}}=\mathbf{\hat{V}}\mathbf{\hat{D}}\mathbf{\hat{V}}^{T}$
is an eigendecomposition of $\mathbf{\hat{A}}$ 
\State Multiply block columns $I$ and $J$ of $\mathbf{G}$ by $\mathbf{\hat{V}}$
\State Multiply block columns $I$ and
$J$ of $\mathbf{V}$ by $\mathbf{\hat{V}}$ \EndIf \EndFor \Until{the columns of $\mathbf{G}$ are nearly orthogonal} 
\State ${\bm \Sigma} = \text{diag}(||\mathbf{G}(:,1)||_2, \ldots, ||\mathbf{G}(:,n)||_2)$ 
\State $\mathbf{U} = \mathbf{G}{\bm \Sigma}^{\dagger}$ \EndFunction \end{algorithmic} 
\end{algorithm} 

\indent Algorithm~\ref{alg:Jacobi_SVD} can be parallelized in essentially the same way as standard Jacobi. Note in particular that the submatrices of $\M{G}^T\M{G}$ corresponding to two pairs of block-column indices $(I,J)$ and $(I',J')$ in line 5 will satisfy the block-independence criteria from Section~\ref{section: 2D_parallel} if $\left\{ I,J \right\} \cap \left\{ I', J' \right\} = \emptyset$. Hence, we can once again break each sweep into several serial steps (according to the usual anti-diagonal ordering) each of which handles several (block) column pairs of $\M{G}$ in parallel. 2D parallel versions of one-sided Jacobi SVD rooted in  this observation have already been explored in the literature \cite{abdelfattah2026efficientbatchsolversingular}. Their advantages mirror those of the 2D algorithm from Section~\ref{section: 2D_parallel}; they are fairly easy to implement and spend most of their time on matrix multiplication (and hence are amenable to running on a GPU).  \\
\indent We focus on extending this work to the 2.5D setting, specifically with the goal of achieving the improved bandwidth/latency bounds from Section~\ref{section: 2.5D_parallel} for the SVD. To do this, we once again consider block-distributing the input matrix $\M{G}$ (as well as the matrix $\M{V}$ of approximate right singular vectors) over a 2.5D processor grid, this time of dimension $\sqrt{\frac{mP}{nc}} \times \sqrt{\frac{nP}{mc}} \times c$. In this case, each layer of the 2.5D grid is (potentially) nonsquare with the same aspect ratio as $\M{G}$. We assume here that $1\leq c \leq \left(\frac{n}{m} \cdot P\right)^{1/3}$ so that the third dimension of the grid remains the smallest, which also implies $c \leq P^{1/3}$ since $m \geq n$. \\
\indent By constructing the grid in this way, we can store $\M{G}$ on each layer in \textit{square} subblocks, with each individual processor owning a block of size $b \times b$ for $b = \sqrt{\frac{mnc}{P}}$, which we assume is an integer. This is primarily done for convenience -- e.g., for compatibility with subroutines like Algorithm~\ref{alg:2.5DMM}, which assumes that matrices are broken into square blocks. An additional upshot is that it allows us to re-use the notation/terminology from the previous sections. In particular, we have the following:
\begin{enumerate}
    \item  As in Section~\ref{section: 2.5D_parallel}, a \textit{fine tile} of $\M{G}$ or $\M{V}$ is a $b \times b$ submatrix owned by an individual processor. 
    \item Like 2.5D parallel Jacobi, our 2.5D Jacobi SVD algorithm operates primarily at the level of larger \textit{coarse} tiles of $\M{G}$ and $\M{V}$, which are aggregated from neighboring fine tiles and have size $B \times B$ for $B = \mu b$, where $\mu$ is a batching parameter satisfying $1 \leq \mu \leq \sqrt{\frac{nP}{mc}}$.
    \item  Coarse tiles are labeled by pairs of indices $(I,J)$ for $1 \leq I \leq m/B$ and $1 \leq J \leq n/B$, with $\M{G}_{IJ}$ and $\M{V}_{IJ}$ following \eqref{eqn: A_subblock} but with block size $B$. We additionally use $\M{G}_J = \M{G}(:,(J-1)B+1:JB)$ to refer to the $J$-th \textit{coarse column} of $\M{G}$.
    \item Processor labels extend naturally from Section~\ref{section: 2.5D_parallel}. That is, individual processors are labeled Proc$\langle p,q,r \rangle$ for $1 \leq p \leq \sqrt{\frac{mP}{nc}}$, $1 \leq q \leq \sqrt{\frac{nP}{mc}}$, and $1 \leq r \leq c$, while larger processor tiles are labeled ProcTile$(I,J,k)$ as in \eqref{eqn: processor_tile}, this time with $1 \leq I \leq m/B$, $1 \leq J \leq n/B$ and $1 \leq k \leq c$. Following \eqref{eqn: proc_slab}, we again use the label ``ProcSlab" to refer to larger collections of processors (e.g., several processor tiles across multiple layers).
\end{enumerate}

\indent The main difference between standard Jacobi and Jacobi SVD is that, in the SVD case, each subproblem to be diagonalized must be computed from columns of $\M{G}$ instead of simply read from the input matrix. In our 2.5D Jacobi SVD algorithm, each of these submatrices is obtained by selecting two (coarse) columns $\M{G}_I$ and $\M{G}_J$ and explicitly computing 
\begin{equation}\label{eqn: SVD_subproblem}
    \hat{\M{A}} = \M{G}(:,[I,J])^T\M{G}(:,[I,J]) = \begin{pmatrix} 
        \M{G}_I^T\M{G}_I & \M{G}_I^T\M{G}_J \\
        \M{G}_J^T \M{G}_I & \M{G}_J^T\M{G}_J 
    \end{pmatrix} \in {\mathbb R}^{2B \times 2B}
\end{equation}
over the processor grid. Taking $\M{G}_I^T \M{G}_I$ as an example, we can do this efficiently by breaking the product into a sum 
 \begin{equation}
     \M{G}_I^T\M{G}_I = \sum_{K=1}^{m/B} \M{G}_{KI}^T\M{G}_{KI}.
 \end{equation}
 where each $\M{G}_{KI}^T\M{G}_{KI}$ is a $B\times B$ multiplication that can be handled by a unique set of processors, recalling that $\M{G}_{KI}$ is initially owned by ProcSlab$(K,I,1)$, via Algorithm~\ref{alg:2.5DMM}. The same can be done for $\M{G}_I^T\M{G}_J$ and $\M{G}_J^T\M{G}_J$, after possibly swapping tiles of $\M{G}$ if necessary. $\M{G}_J^T\M{G}_I$, meanwhile, can be obtained by simply taking the transpose of $\M{G}_I^T\M{G}_J$. \\ \indent At the conclusion of these 2.5D matrix multiplications, each intermediate product is stored in a processor tile belonging to (coarse) column $I$ or $J$ of the processor grid's top layer (following the output assumptions of Algorithm~\ref{alg:2.5DMM}). We can then perform a reduction along these columns to obtain $\hat{\M{A}}$ on ProcSlab$([I,J],[I,J],1)$. Once this is accomplished, our parallel Jacobi SVD algorithm can proceed exactly as Algorithm~\ref{alg:2.5D_Parallel_Jacobi} -- e.g., using multiple processors to diagonalize $\hat{\M{A}}$ via 2D parallel Jacobi and subsequently using 2.5D matrix multiplication for the column updates. Note that the latter requires strictly \textit{less} work in the SVD setting since we only need to multiply $\M{G}$ and $\M{V}$ on the right. In addition, we once again introduce a utilization factor $ \eta \in [2\mu/\sqrt{\frac{mcP}{n}},1]$ to determine the fraction of the $2 \mu \sqrt{\frac{mcP}{n}}$ available processors in ProcSlab$(:,[I,J],:)$ to use for each call to 2D parallel Jacobi. As in Algorithm~\ref{alg:2.5D_Parallel_Jacobi}, the lower bound $\eta = 2 \mu / \sqrt{\frac{mcP}{n}}$ corresponds to using only the $4\mu^2$ processors that initially own $\hat{\M{A}}$. 
 
\begin{algorithm}
\caption{2.5D Parallel Jacobi SVD}
\label{alg:2.5D_Jacobi_SVD}
\textbf{Input:} (1) $\M{G} \in {\mathbb R}^{m \times n}$ with $m \geq n$. \\
\phantom{Input2 } (2) $\M{V} \in {\mathbb R}^{n \times n}$ an orthogonal matrix of approximate right singular vectors (optional). \\
\phantom{Input2 } (3) A set $\mathfrak{P}$ of $P$ processors, arranged in a $\sqrt{\frac{mP}{nc}} \times \sqrt{\frac{nP}{mc}} \times c$ grid for replication factor  \\ 
\phantom {Input2 (1) } $1 \leq c \leq \left( \frac{n}{m} \cdot P \right)^{1/3}$ and labeled Proc$\langle p,q,r\rangle$ for $1 \leq p \leq \sqrt{\frac{mP}{nc}}$, $1 \leq q\leq \sqrt{\frac{nP}{mc}}$ and $1 \leq r \leq c$. \\
\phantom{Input2 } (4) $\mu$ a batching parameter satisfying $ c\leq \mu \leq \sqrt{\frac{nP}{mc}}. $\\
\phantom{Input2 } (5) $\eta \in [2\mu/\sqrt{\frac{mcP}{n}}, 1]$ a utilization factor. \\
\phantom{Input2 } (6) Lists $L$ and $\hat{L}$ of, respectively, $\sqrt{\frac{nP}{mc}}/\mu$ and $\sqrt{2\eta \mu \sqrt{\frac{mcP}{n}}}$ groups of pairwise disjoint indices $(I,J)$  \\
\phantom{Input2 (4) } (e.g.,  generated by two calls to Algorithm~\ref{alg:anti_diagonal_list}).

\vspace{2mm}
\textbf{Requires:} Initially, $\M{G}$ is stored on the top layer of the grid, with the coarse tile
\[ \Mb{G}{IJ}=\Me{G}{(I{-}1)B{+}1:IB,(J{-}1)B{+}1:JB} \; \; \; \text{for} \; \; \; B = \mu \sqrt{\frac{mnc}{P}} \]
block-distributed over the coarse processor tile ProcTile$(I,J,1)$ for $1 \leq I \leq m/B$ and $1 \leq J \leq n/B$ so that each individual processor owns a $b \times b$ fine tile of $\M{G}$ for $b =  \sqrt{mnc/P}$. We denote the $J$-th coarse column block of $\M{G}$ as $\M{G}_J = \M{G}(:, (J-1)B+1:JB)$. If omitted, $\M{V}$ is taken to be the identity matrix; since it is $n \times n$, $\M{V}$ is stored only on the square subgrid corresponding to coarse indices $1 \leq K,J \leq n/B$, with $\M{V}_{KJ}$ owned by ProcTile$(K,J,1)$. 

\vspace{2mm}
\textbf{Ensure:}
On output, one sweep of one-sided Jacobi SVD has been applied to $\M{G}$ and $\M{V}$. 
\algrule
\setstretch{1.2}
\begin{algorithmic}[1]
\Function{$[\M{V}, \M{G}]=$2.5D\_Jacobi\_SVD}{$\M{G}$, $\M{V}$, $\mathfrak{P}$, $\mu$, $\eta$, $L$, $\hat{L}$} 
\For{each group $g\in L$}  
\For{each pair $(I,J)\in g$ \textbf{in parallel}} 
\vspace{1mm}
\State $\hat{\M{A}} = \begin{pmatrix} \M{G}_I^T\M{G}_I & \M{G}_I^T \M{G}_J \\ \M{G}_J^T\M{G}_I & \M{G}_J^T \M{G}_J \end{pmatrix}$ \Comment{Subproblem to be formed on ProcSlab$([I,J],[I,J],1)$.}
\vspace{1mm}
\For{$K = 1:m/B$ \textbf{in parallel}} \label{line: start_forming_subproblem}
    \State ProcTile$(K,I,1)$ and ProcTile$(K,J,1)$: Exchange $\M{G}_{KI} \leftrightarrow \M{G}_{KJ}$ and, if $K < n/B$, $\M{V}_{KI} \leftrightarrow \M{V}_{KJ}$    
    \State ProcTile$(K,I,1)$: Swap off-diagonal blocks to compute $\M{G}_{KI}^T$
    \State \hspace{2.5cm} $\M{G}_{KI}^T\M{G}_{KI}= $ \Call{2.5D\_MM}{$\M{G}_{KI}^T$, $\M{G}_{KI}$, ProcSlab$(K,I,:)$}
    \State \hspace{2.5cm} $\M{G}_{KI}^T\M{G}_{KJ} = $ \Call{2.5D\_MM}{$\M{G}_{KI}^T$, $\M{G}_{KJ}$, ProcSlab$(K,I,:)$}
    \State ProcTile$(K,J,1)$: Swap off-diagonal blocks to compute $\M{G}_{KJ}^T$
    \State \hspace{2.5cm} $\M{G}_{KJ}^T\M{G}_{KJ} = $ \Call{2.5D\_MM}{$\M{G}_{KJ}^T$, $\M{G}_{KJ}$, ProcSlab$(K,J,:)$}
    
\EndFor
\State ProcTile$(I,I,1)$: Reduce sum $\M{G}_{KI}^T\M{G}_{KI}$ along ProcSlab$(:,I,1)$ to obtain $\M{G}_I^T\M{G}_I$
\State \hspace{2.4cm} Reduce sum $\M{G}_{KI}^T\M{G}_{KJ}$ along ProcSlab$(:,I,1)$ to obtain $\M{G}_I^T\M{G}_J$
\State \hspace{2.4cm} Send $\M{G}_I^T\M{G}_J$ and $\M{G}_J^T\M{G}_I$ to ProcTile$(I,J,1)$ and ProcTile$(J,I,1)$
\State ProcTile$(J,J,1)$: Reduce sum $\M{G}_{KJ}^T\M{G}_{KJ}$ along ProcSlab$(:,J,1)$ to obtain $\M{G}_J^T\M{G}_J$ \label{line: end_forming_subproblem}
\State ProcSlab$([I,J],[I,J],1)$: Scatter $\hat{\M{A}}$ to $\hat{\mathfrak{P}}$, an $\eta$ fraction of processors in ProcSlab$(:,[I,J],:)$
\State Initialize $\hat{\M{V}}$ as the $2B \times 2B$ identity matrix over $\hat{\mathfrak{P}}$
\Repeat
\State $[\hat{\M{V}},\hat{\M{A}}]$ = \Call{2D\_Parallel\_Jacobi}{$\hat{\M{A}}$, $\hat{\M{V}}$, $\hat{\mathfrak{P}}$, $\hat{L}$}
\Until{converged or a maximum number of sweeps is reached}

\State \Comment{Continued on next page $\rightarrow $} \algstore{bkbreak} \end{algorithmic} 
\end{algorithm}

\begin{algorithm}
\begin{algorithmic} \algrestore{bkbreak}

\State \Comment{Continued from previous page}

\State Gather $\hat{\M{V}}$ from $\hat{\mathfrak{P}}$ back to ProcSlab$([I,J],[I,J],1)$
\State $\hat{\M{V}} = \begin{pmatrix} \hat{\M{V}}_{11} & \hat{\M{V}}_{12} \\ \hat{\M{V}}_{21} & \hat{\M{V}}_{22} \end{pmatrix} $ \Comment{Break $\hat{\M{V}}$ into $B \times B$ blocks}
\State ProcTile$(J,I,1)$: Send $\hat{\M{V}}_{21}$ to ProcTile$(I,I,1)$
\State ProcTile$(I,J,1)$: Send $\hat{\M{V}}_{12}$ to ProcTile$(J,J,1)$
\State ProcTile$(I,I,1)$: Broadcast $[\hat{\M{V}}_{11}, \hat{\M{V}}_{21}]$ to ProcTile$(K,I,1)$ for all $K \neq I$.
\State ProcTile$(J,J,1)$: Broadcast $[\hat{\M{V}}_{12}, \hat{\M{V}}_{22}]$ to ProcTile$(K,J,1)$ for all $K \neq J$.
\For{$K = 1:m/B$ \textbf{in parallel}} \Comment{Update $\M{G}$ and $\M{V}$}
    \State ProcTile$(K,I,1)$: $\M{M}_1 = $ \Call{2.5D\_MM}{$\M{G}_{KI}$,$\hat{\M{V}}_{11}$, ProcSlab$(K,I,:)$}
    \State \hspace{2.5cm} $\M{M}_2 = $ \Call{2.5D\_MM}{$\M{G}_{KJ}$,$\hat{\M{V}}_{21}$, ProcSlab$(K,I,:)$} 
    \State \hspace{2.5cm} $\M{G}_{KI} \gets \M{M}_1 + \M{M}_2$
    \State \hspace{2.5cm} \textbf{if} $K \leq n/B$ \textbf{then}
    \State \hspace{3.05cm} $\M{M}_1 = $ \Call{2.5D\_MM}{$\M{V}_{KI}$,$\hat{\M{V}}_{11}$, ProcSlab$(K,I,:)$}
    \State \hspace{3.05cm} $\M{M}_2 = $ \Call{2.5D\_MM}{$\M{V}_{KJ}$,$\hat{\M{V}}_{21}$, ProcSlab$(K,I,:)$} 
    \State \hspace{3.05cm} $\M{V}_{KI} \gets \M{M}_1 + \M{M}_2$
    \State \hspace{2.5cm} \textbf{end if}
    \vspace{2mm}
            
    \State ProcTile$(K,J,1)$: $\M{M}_1 = $ \Call{2.5D\_MM}{$\M{G}_{KI}$,$\hat{\M{V}}_{21}$, ProcSlab$(K,J,:)$}
    \State \hspace{2.5cm} $\M{M}_2 = $ \Call{2.5D\_MM}{$\M{G}_{KJ}$,$\hat{\M{V}}_{22}$, ProcSlab$(K,J,:)$} 
    \State \hspace{2.5cm} $\M{G}_{KJ} \gets \M{M}_1 + \M{M}_2$
    \State \hspace{2.5cm} \textbf{if} $K \leq n/B$ \textbf{then}
    \State \hspace{3.05cm} $\M{M}_1 = $ \Call{2.5D\_MM}{$\M{V}_{KI}$,$\hat{\M{V}}_{21}$, ProcSlab$(K,J,:)$}
    \State \hspace{3.05cm} $\M{M}_2 = $ \Call{2.5D\_MM}{$\M{V}_{KJ}$,$\hat{\M{V}}_{22}$, ProcSlab$(K,J,:)$} 
    \State \hspace{3.05cm} $\M{V}_{KJ} \gets \M{M}_1 + \M{M}_2$
    \State \hspace{2.5cm} \textbf{end if}
\EndFor
\EndFor
\EndFor
\EndFunction

\end{algorithmic} 
\end{algorithm}

\indent In this context, each sweep of our Jacobi SVD algorithm loops through pairs of columns $(I,J)$ with  $1 \leq I < J \leq n/B = \sqrt{\frac{nP}{mc}}/\mu$. As noted above, we can use the same anti-diagonal ordering from the previous sections -- e.g., generated by Algorithm~\ref{alg:anti_diagonal_list} with input $\sqrt{\frac{nP}{mc}}/\mu$ -- to break each sweep into several serial steps. Algorithm~\ref{alg:2.5D_Jacobi_SVD} presents our 2.5D version of one-sided Jacobi SVD, which parallelizes Algorithm~\ref{alg:Jacobi_SVD} in this way and calls 2D parallel Jacobi and 2.5D matrix multiplication as discussed above. Following Algorithms~\ref{alg:High_Level_Parallel_Jacobi} and~\ref{alg:2.5D_Parallel_Jacobi}, it applies a single sweep to the input matrix. Accordingly, we omit the post-processing steps of Algorithm~\ref{alg:Jacobi_SVD} (i.e., lines 13 and 14). Convergence can be guaranteed over multiple sweeps if a pivoting step is included, though we once again leave this out for brevity. Our main contribution here is a general complexity bound for this algorithm that covers any choice of batching parameter $\mu$ and utilization factor $\eta$. 

\begin{theorem}\label{thm: 2.5D_SVD_comp} 
    One sweep of 2.5D parallel Jacobi SVD (i.e., Algorithm~\ref{alg:2.5D_Jacobi_SVD}) has complexity
    $$
    \aligned 
        \alpha \cdot O \left( X \cdot \sqrt[4]{\frac{n}{m}}\cdot \log \left( \eta \mu \sqrt{ \frac{mcP}{n}} \right) + \frac{\sqrt{mP/nc}}{\mu} \log \left( \frac{\sqrt{mP/nc}}{\mu} \right) \right) &+ \beta \cdot O \left( \frac{\sqrt[4]{mn^7}}{X} + \sqrt{\frac{mn^3}{cP}} \right) \\
        & \hspace{-5cm} + \gamma \cdot O\left( \frac{\sqrt{mn^5}}{X^2}+ \frac{mn^2}{P}\right) 
    \endaligned
    $$
    for $X = \sqrt[4]{\frac{\eta^2  P^3}{\mu^2c}}$, assuming that $O(1)$ sweeps of 2D parallel Jacobi are applied to each subproblem.
\end{theorem}
\begin{proof}
    Following the proof of Theorem~\ref{thm: 2.5D_parallel_comp}, we start by cataloging the work associated with a single index pair $(I,J) \in g$.
    \begin{enumerate}

        \item Lines 5--16: Processors in coarse columns $I$ and $J$ collaborate to form
        $\hat{\M{A}}=\M{G}(:,[I,J])^T\M{G}(:,[I,J])$
        on ProcSlab$([I,J],[I,J],1)$. This requires computing the $B\times B$ products
        $\M{G}_{KI}^T\M{G}_{KI}$, $\M{G}_{KI}^T\M{G}_{KJ}$, and $\M{G}_{KJ}^T\M{G}_{KJ}$ via 2.5D matrix multiplication over a $\mu\times\mu\times c$ subgrid. By Theorem~\ref{thm: 2.5d_MM}, the cost (per-processor) of this operation is
        \begin{equation}\label{eqn: 2.5D_MM_cost}
            \alpha\,O\!\left(\frac{\mu}{c}+\log c\right)
            +
            \beta\,O\!\left(\frac{B^2}{\mu c}\right)
            +
            \gamma\,O\!\left(\frac{B^3}{\mu^2c}\right).
        \end{equation}
        Since the reduction of the partial products along the row index has per-processor complexity
        \begin{equation}\label{eqn: 2.5D_MM_cost2}
        \alpha\,O(\log(m/B))+\beta\,O(b^2)+\gamma\,O(b^2),
        \end{equation}
        the total cost of this step is given by \eqref{eqn: 2.5D_MM_cost}
        
        \item Lines 17-24: $\hat{\M{A}}$ is diagonalized by $2 \eta \mu \sqrt{\frac{mcP}{n}}$ processors via 2D parallel Jacobi. Theorem~\ref{thm: parallel_comp} implies that the corresponding call to Algorithm~\ref{alg:High_Level_Parallel_Jacobi} has complexity
        \begin{equation}
            \alpha \cdot O \left( \sqrt{ \eta \mu \sqrt{ \frac{mcP}{n}}} \log \left( \eta \mu \sqrt{ \frac{mcP}{n}} \right) \right) + \beta \cdot O \left( \frac{B^2}{\sqrt{ \eta \mu \sqrt{\frac{mcP}{n}}}} \right) + \gamma \cdot O \left( \frac{B^3}{ \eta\mu  \sqrt{\frac{mcP}{n}}} \right)
        \end{equation}
        if only $O(1)$ sweeps are executed. Following the same argument made in the proof of Theorem~\ref{thm: 2.5D_parallel_comp}, the communication associated with the initial scatter and subsequent gather operations has bandwidth/latency cost
        \begin{equation}
            \alpha \cdot O \left( \log \left( \frac{\eta \sqrt{mcP/n}}{\mu} \right)\right) + \beta \cdot O \left( b^2 \right).
        \end{equation}

        \item Lines 26 and 27: Each processor in ProcTile$(J,I,1)$ and ProcTile$(I,J,1)$ sends its (fine) block of $\hat{\M{V}}$ to one processor in either ProcTile$(I,I,1)$ or ProcTile$(J,J,1)$. The cost of this operations is $\alpha \cdot O(1) + \beta \cdot O(b^2)$.

        \item Lines 28 and 29: Each processor in ProcTile$(I,I,1)$ and ProcTile$(J,J,1)$ broadcasts two $b \times b$ blocks of $\hat{\M{V}}$ to $O \left( \sqrt{\frac{mP}{nc}}/\mu \right)$ processors (one in each coarse processor tile in the same column). Following Table~\ref{tab:mpi_routines}, each of these broadcasts has bandwidth/latency cost
        \begin{equation}
            \alpha \cdot O \left( \log \left(\frac{\sqrt{mP/nc}}{\mu} \right)\right) + \beta \cdot O \left(b^2 \right)
        \end{equation}

        \item Lines 30-43: Each processor in (coarse) processor columns $I$ and $J$ participates in at most four 2.5D matrix multiplications, each of size $B \times B$ done with $2\mu^2c$ processors. The cost of this operation is given by \eqref{eqn: 2.5D_MM_cost} and dominates the associated block-wise matrix addition. 
    \end{enumerate}

    In total, we have the following for each $(I,J) \in g$:
    \begin{equation}\label{eqn: separate_cost_SVD}
        \aligned 
          \text{Latency:}& \; \; \alpha \cdot O\left( \log \left( \frac{\eta \sqrt{mcP/n}}{\mu} \right) + \sqrt{\eta \mu \sqrt{\frac{mcP}{n}}} \log \left( \eta \mu \sqrt{\frac{mcP}{n}} \right) + \log \left( \frac{\sqrt{mP/nc}}{\mu} \right) + \frac{\mu}{c} + \log c \right).  \\
          \text{Bandwidth:}& \; \; \beta \cdot O \left( b^2 + \frac{B^2}{\sqrt{\eta \mu \sqrt{mcP/n}}} + \frac{B^2}{\mu c} \right). \\
          \text{Arithmetic:}& \; \; \gamma \cdot O \left( \frac{B^3}{\eta \mu \sqrt{mcP/n}} + \frac{B^3}{\mu^2 c} \right).
          \endaligned 
      \end{equation}
      As in the proof of Theorem~\ref{thm: 2.5D_parallel_comp}, the lower bounds $\eta \geq 2\mu/\sqrt{\frac{mcP}{n}}$ and $\frac{B^2}{\mu c} \geq b^2$ allow us to simplify each of these complexities to obtain a cumulative cost of
      \begin{equation}\label{eqn: one_step_SVD}
      \aligned   
      \alpha \cdot O\left( \sqrt{\eta \mu \sqrt{\frac{mcP}{n}}} \log \left( \eta \mu \sqrt{\frac{mcP}{n}} \right)  + \log \left( \frac{\sqrt{mP/nc}}{\mu} \right) \right) & + \beta \cdot O \left( \frac{B^2}{\sqrt{\eta \mu \sqrt{mcP/n}}} + \frac{B^2}{\mu c} \right) \\
      & \hspace{-4cm} +  \gamma \cdot O \left( \frac{B^3}{\eta \mu \sqrt{mcP/n}} + \frac{B^3}{\mu^2 c} \right).
      \endaligned 
    \end{equation}
    In one-sided Jacobi SVD -- which applies only (block) column updates to the input matrix as opposed to row \textit{and} column updates -- a single processor participates in the computation for at most one $(I,J) \in g$. Hence, \eqref{eqn: one_step_SVD} bounds the complexity of one parallel step of Algorithm~\ref{alg:2.5D_Jacobi_SVD}. The result now follows from the fact that there are $|L| = \sqrt{\frac{nP}{mc}}/\mu$ of these steps in total (and recalling $B = \mu \sqrt{\frac{mnc}{P}}$).
\end{proof} \\

\subsection{Latency and Bandwidth Optimality}
Given the similarity of the complexity bounds for 2.5D parallel Jacobi and 2.5D parallel Jacobi SVD and noting in particular that  $X$ is the same in Theorem~\ref{thm: 2.5D_parallel_comp} and Theorem~\ref{thm: 2.5D_SVD_comp},  much of the discussion of parameter optimization from the previous section carries over here. That is, for fixed $\mu$, taking $\eta$ at its lower feasible endpoint$\eta = 2 \mu / \sqrt{\frac{mcP}{n}}$, which again corresponds to using only the $4 \mu^2$ processors in ProcSlab$([I,J],[I,J],1)$ to diagonalize $\hat{\M{A}}$, yields a ``latency optimal" implementation, whose sweep complexity is 
\begin{equation}
     \aligned   
      \alpha \cdot O\left( \sqrt{\frac{nP}{mc}} \log(2\mu^2)  + \frac{\sqrt{mP/nc}}{\mu} \log \left( \frac{\sqrt{mP/nc}}{\mu} \right) \right) & + \beta \cdot O \left( \sqrt{\frac{mn^3}{P/c}} \right) + \gamma \cdot O \left( \frac{mn^2}{P/c} \right).
      \endaligned 
\end{equation}
As in the last section, minimizing latency comes at the cost of both bandwidth \textit{and} arithmetic. Setting $\eta = \Theta(\mu \sqrt{c/P})$, with the implicit constant chosen so that $\eta$ is feasible, preserves the optimal arithmetic cost for every $m \geq n$. The arithmetic term in Theorem~\ref{thm: 2.5D_SVD_comp} itself only requires $\eta = \Omega\!\left(\mu \sqrt{nc/(mP)}\right)$, in addition to the standing feasibility bound; hence, the displayed choice has the smallest latency among arithmetic-optimal versions when $m = \Theta(n)$, while remaining a valid choice for taller matrices. In this case, the complexity of one sweep of Algorithm~\ref{alg:2.5D_Jacobi_SVD} is 
\begin{equation}\label{eqn: reasonable_latency_optimal}
     \aligned   
      \alpha \cdot O\left( \sqrt{P} \cdot \sqrt[4]{\frac{n}{m}} \cdot \log \left( \mu^2 c \sqrt{\frac{m}{n}} \right) + \frac{\sqrt{mP/nc}}{\mu} \log \left( \frac{\sqrt{mP/nc}}{\mu} \right) \right) & + \beta \cdot O \left( \sqrt[4]{ \frac{mn^7}{P^2}}   + \sqrt{\frac{mn^3}{cP}} \right) \\
      & \hspace{-4cm} + \gamma \cdot O \left( \frac{mn^2}{P} \right).
      \endaligned 
\end{equation}
\indent If $\M{G}$ is sufficiently tall, in particular if $n = O(m/c^2)$, then this version of 2.5D parallel Jacobi SVD also achieves optimal bandwidth cost $\beta \cdot O \left( \sqrt{\frac{mn^3}{cP}} \right)$. In general, we improve the bandwidth cost of Algorithm~\ref{alg:2.5D_Jacobi_SVD} by making $X$ as large as possible. If, for example, $\eta = 1$, then $X = \sqrt[4]{\frac{P^3}{\mu^2 c}}$ and the complexity of each sweep of 2.5D parallel Jacobi SVD is
\begin{equation}
     \aligned   
      \alpha \cdot O\left( \sqrt[4]{\frac{nP^3}{m \mu^2 c}} \cdot \log \left( \mu \sqrt{\frac{mcP}{n}} \right) + \frac{\sqrt{mP/nc}}{\mu} \log \left( \frac{\sqrt{mP/nc}}{\mu} \right) \right) & + \beta \cdot O \left( \sqrt[4]{ \frac{mn^7 \mu^2 c}{P^3}}   + \sqrt{\frac{mn^3}{cP}} \right) \\
      & \hspace{-4cm} + \gamma \cdot O \left( \sqrt{\frac{mn^5 \mu^2 c}{P^3}} + \frac{mn^2}{P} \right).
      \endaligned 
\end{equation}
Bandwidth-optimality is now guaranteed, for example, by choosing $\mu = \Theta\left(\sqrt{\frac{mP}{nc^3}} \right)$, which we note satisfies the assumptions of Algorithm~\ref{alg:2.5D_Jacobi_SVD} if $c^5 \lesssim \frac{mP}{n}$ and $c \gtrsim \frac{m}{n}$. This implies the following analog of Corollary~\ref{cor: bandwidth_optimal_jacobi}.
\begin{corollary}\label{cor: bandwidth_optimal_SVD}
    Suppose we run 2.5D parallel Jacobi SVD (Algorithm~\ref{alg:2.5D_Jacobi_SVD}) with the following parameter choices (which we assume satisfy the necessary assumptions):
    \begin{enumerate}
        \item $\mu = \Theta\left(\sqrt{\frac{mP}{nc^3}}\right)$.
        \item $\eta = 1$, meaning all available processors are utilized in each call to 2D parallel Jacobi.
    \end{enumerate}
    Then the complexity of each sweep is
    \[ \alpha \cdot O \left(\sqrt{\frac{ncP}{m}} \log \left( \frac{mP}{nc}\right) \right) + \beta \cdot O \left( \sqrt{\frac{mn^3}{cP}}\right) + \gamma \cdot O\left(\frac{mn^2}{P} \right), \]
    again assuming that only $O(1)$ sweeps of the 2D algorithm are applied to each subproblem.
\end{corollary}
These results echo the bandwidth/latency tradeoff discussed in Section~\ref{section: lower_bound}. 
While we do not include a proof, we conjecture that similar lower bounds apply to one-sided Jacobi SVD.  Indeed, note that the first terms of the latency and bandwidth costs in Theorem~\ref{thm: 2.5D_SVD_comp} always multiply to $n^2$ (up to log factors) regardless of the choice of $\mu$ and $\eta$.

\subsection{Preprocessing with QR}
In practice, the efficiency of one-sided Jacobi SVD can be improved by first computing a reduced QR factorization $\M{G} = \M{Q}\M{R}$ (with or without pivoting) and subsequently running the algorithm on $\M{R}$ (or its transpose). An SVD of $\M{G}$ can then be obtained by multiplying the matrix of left singular vectors of $\M{R}$ by $\M{Q}$; that is, if $\M{R} = \M{U} {\bf \Sigma} \M{V}^T$ then $\M{G} = (\M{Q}\M{U}) {\bf \Sigma} \M{V}^T$. The benefits of doing this are straightforward. If $m \gg n$, then $\M{R}$ is much smaller than $\M{G}$, meaning each sweep of Jacobi will be much cheaper. Moreover, we often see faster convergence from Jacobi if pivoting is used in the QR decomposition (see the discussion in \cite{DrmacVeselic08a}). \\
\indent In the 2.5D setting, this preprocessing step, without pivoting, can be handled via the 2.5D parallel QR algorithm of Solomonik et al. \cite{2.5D_QR_Band_Reduction} (see also \cite[Chapter 6]{Solomonik_thesis}), which similarly runs over a 2.5D processor grid with replication factor $1 \leq c \leq \left( \frac{n}{m} \cdot P \right)^{1/3}$. For brevity we do not state the details of this algorithm, instead simply noting that its complexity is
\begin{equation}\label{eqn: 2.5D_QR_cost}
    \alpha \cdot O \left( \sqrt{\frac{ncP}{m}}\log^2(P) \right) + \beta \cdot O \left( \sqrt{\frac{mn^3}{cP}} + \frac{mn}{P} \right) + \gamma \cdot O \left( \frac{mn^2}{P} \right).
\end{equation}
This follows from \cite[Theorem 3.6]{2.5D_QR_Band_Reduction}, where, in their notation, $c = \left( \frac{nP}{m} \right)^{2 \delta - 1}$ for $\delta \in [1/2, 2/3]$. \\
\indent This expression essentially matches Corollary~\ref{cor: bandwidth_optimal_SVD}. Accordingly, preprocessing with QR cannot improve the complexity bounds presented in this section. Nevertheless, recalling that our results for one-sided Jacobi SVD technically only cover a single sweep, an implementation that includes QR is likely to be much more efficient in practice.

\section{Conclusion}\label{section: conclusion} 
In this paper, we analyzed several parallel versions of Jacobi's method for the symmetric eigenvalue problem. Using both 2D (Section~\ref{section: 2D_parallel}) and 2.5D (Section~\ref{section: 2.5D_parallel}) processor grids, we derived parallel Jacobi algorithms with optimal arithmetic cost and with bandwidth or latency matching the corresponding distributed-memory matrix-multiplication bounds. In each case, we observed a bandwidth/latency tradeoff that was codified theoretically in Section~\ref{section: lower_bound}. Finally, we presented analogous results for one-sided Jacobi SVD in Section~\ref{section: parallel_jacobi_svd}. In total, this work adds Jacobi's method to the growing collection of (dense) parallel linear algebra routines for which the benefits of additional memory -- specifically for reducing communication -- have been explored rigorously. In combination with \cite{arxiv_manuscript}, this effort reasserts the relevance of Jacobi's method for modern scientific computing. \\
\indent Several open questions/directions remain. The most immediate is implementation, as the theoretical efficiency gains of 2.5D parallel Jacobi proved here have not yet been demonstrated in practice. There are a number of related details to be explored there  -- e.g., the potential application of variable precision arithmetic as well as the use of automated tuning algorithms \cite{cho2023harnessing,cho2025surrogate,luo2021non,luo2024hybrid} to set the batching/utilization parameters $\mu$ and $\eta$. In future work, we plan to (1) extend the complexity analysis presented here to accommodate fast (parallel) matrix multiplication (e.g., \cite{LBDS12}) and (2) revisit the improved accuracy bounds of \cite{DemmelVeselic92} with the goal of obtaining analogous results for the parallel setting. Finally, we hope that this paper inspires additional work on 2.5D algorithms for eigenvalue problems, noting in particular that spectral divide and conquer \cite{banks2020pseudospectral,Demmel2024,definite_dnc,Nakatsukasa_Higham_2013} -- which is already known to be communication optimal in the 2D parallel setting \cite{BDD11} --  has not yet been adapted to the 2.5D setting.

\section*{Acknowledgements}
This work was supported by NSF grant MSPRF 2402027 and partly supported by the U.S. Department of Energy under contract DE-AC02-05CH11231 and by NSF grants DMS-2412403, DMS-2606336, and DMS-2616828.

\bibliographystyle{abbrv}
\bibliography{book-ref}

\begin{center}
\textbf{\Large APPENDICES}
\end{center}
\appendix
\section{Detailed Pseudocode for 2D Parallel Jacobi}\label{appendix: detailed_code}
This appendix contains a more detailed version of 2D parallel Jacobi (i.e., Algorithm~\ref{alg:High_Level_Parallel_Jacobi}) written at the processor level. This is presented as Algorithm~\ref{alg:Parallel_Jacobi} and reuses the notation of Section~\ref{section: 2D_parallel}. \\
\indent The pseudocode of Algorithm~\ref{alg:Parallel_Jacobi} should be read from the perspective of an individual processor. Accordingly, it includes instructions for determining what processors in the 2D grid any one needs to communicate with. This requires fairly intricate case work; for each group of disjoint block  indices in the list $L$, there are a total of seven distinct (though not mutually exclusive) cases that Proc$(I,J)$ can be assigned, which we list below. Note that, at each step, any processor applying a rotation to the block of $\M{A}$ it owns on both the left \textit{and} right will be assigned two cases -- one of 4/5 and one of 6/7, with the combination depending on its location in the grid.
 
 \renewcommand{\arraystretch}{1.2}
\begin{table}[h]
    \centering
    \begin{tabular}{c|p{.9\linewidth}}
         Case & State of Proc$(I,J)$ \\
         \hline
         1 & Contains a block to be zeroed out -- will (approximately) diagonalize a $2b \times 2b$ matrix if $I < J$.\\
         2 & Lies on the diagonal of the processor grid and to the left of a superdiagonal block being zeroed out (will broadcast pieces of the corresponding rotation matrix to block row/ column $I$).\\
         3 & Lies on the diagonal of the grid and below a superdiagonal block being zeroed out (will broadcast pieces of the corresponding rotation matrix to block row/column $J$).\\
         4 & Will apply a rotation on the right, exchange blocks with a rightward processor in the same row. \\
         5 & Will apply a rotation on the right, exchange blocks with a leftward processor in the same row.\\
         6 & Will apply a rotation on the left, exchange blocks with a lower processor in the same column. \\
         7 & Will apply a rotation on the left, exchange blocks with an upper processor in the same column.\\
    \end{tabular}
\end{table}
 \renewcommand{\arraystretch}{1}

\indent To keep track of this information, each processor updates a book-keeping array ``info," which records both the case number(s) and the indices of processors to be communicated with. For group two of the motivating example from Section~\ref{section: 2D_parallel}, again handled by an $8\times 8$ processor grid, the casework labels apply as follows: 
\begin{equation}\label{eqn: casework_example}
    \renewcommand\fbox{\fcolorbox{blue}{white}}
    \begin{bmatrix} 2 & 4/6 & 4/6 & 6 & 5/6 & 5/6 & \fbox{1} & 6\\
     4/6 & 2 & 4/6 & 6 & 5/6 & \fbox{1} & 5/6 & 6 \\
     4/6 & 4/6 & 2 & 6 & \fbox{1} & 5/6 & 5/6 & 6 \\
     4 & 4 & 4 & - & 5 & 5 & 5 & - \\
     4/7 & 4/7 & 1 & 7 & 3 & 5/7 & 5/7 & 7 \\
     4/7 & 1 & 4/7 & 7 & 5/7 & 3 & 5/7 & 7 \\
     1 & 4/7 & 4/7 & 7 & 5/7 & 5/7& 3 & 7 \\
     4 & 4 & 4 & - & 5 & 5 & 5& - \\
    \end{bmatrix}
\end{equation}
Here, superdiagonal processors whose indices belong to group two are highlighted in blue, while those without a label do not participate in the computation (compare with \eqref{eqn: alt_transformations}).

\begin{algorithm}
\caption{2D Parallel Jacobi for the Symmetric Eigenproblem (Processor Level)}
\label{alg:Parallel_Jacobi}


\textbf{Input:} (1) $\M{A} \in {\mathbb R}^{n \times n}$ a symmetric matrix. \\
\phantom{Input2 } (2) $\M{Q} \in {\mathbb R}^{n \times n}$ an orthogonal matrix of approximate eigenvectors (optional). \\
\phantom{Input2 } (3) $L$ a list of $\sqrt{P}$ groups of pairwise disjoint indices $(I,J)$ (e.g.,  generated by Algorithm~\ref{alg:anti_diagonal_list}).

\vspace{1mm}
\textbf{Requires:} Proc$(I,J)$ owns the block $\Mb{A}{IJ}=\Me{A}{(I{-}1)b{+}1:Ib,(J{-}1)b{+}1:Jb}$ of $\M{A}$ for $b = n /\sqrt{P}$, which is assumed to be an integer. If omitted, $\M{Q}$ is taken to be the identity matrix; in either case, it is partitioned and stored in the same way as $\M{A}$. 

\vspace{1mm}
\textbf{Ensure:}
On output, one sweep of (block) Jacobi has been applied to $\M{A}$ and $\M{Q}$.
\algrule
\setstretch{1.1}
\begin{algorithmic}[1]
\Function{$[\M{Q},\M{A}]= \; $2D\_Parallel\_Jacobi}{$\M{A}$,$\M{Q}$,$P$}
\State $(I,J)=\Call{MyProcID}$ \Comment{Get processor index}
 \For{$g=1$ to $\sqrt{P}$} \Comment{For each group in $L$, decide how processor $(I,J)$ participates}
 \State info = [\;] \Comment{Casework data}
\If{$(I,J)\in L(g)$
or $(J,I)\in L(g)$}  \State info = [1, $I$, $I$] \Comment{Case 1: $\Mb{A}{IJ}$ and $\Mb{A}{JI}$
to be zeroed out}\ElsIf{$I=J$ and $(I,J')\in L(g)$
for some $J'$}  \State info = [2, 0, $J'$] \Comment{Case 2: Left diagonal processor of block being zeroed out} \ElsIf{$I=J$
and $(I',J)\in L(g)$ for some $I'$}
\State info = [3, $I'$, 0] \Comment{Case 3: Lower
diagonal processor of block being zeroed 
out} \Else
\If{$J=\min(I',J')$ for some $(I',J')\in L(g)$}
\State info = [info; 4, $I'$, $J'$] \Comment{Case 4: Other processor in left column doing an update} 
\EndIf \If{$J=\max(I',J')$
for some $(I',J')\in L(g)$} \State
info = [info; 5, $I'$, $J'$] \Comment{Case 5: Other processor
in right column doing an update} \EndIf
\If{$I=\min(I',J')$ for some $(I',J')\in L(g)$}
 \State info = [info; 6, $I'$, $J'$] \Comment{Case 6: Other processor in upper row doing an update} \EndIf
 \If{$I=\max(I',J')$
for some $(I',J')\in L(g)$} \State
info = [info; 7, $I'$, $J'$] \Comment{Case 7: Other processor
in lower row doing an update} \EndIf \EndIf
\For{k = 1 to length(info)} 
\State [Case, $I'$, $J'$] = info$(k,\;:)$
\If{Case = 2} Send $\Mb{A}{II}$ to
Proc($I,J'$) \ElsIf{Case = 3} Send $\Mb{A}{JJ}$ to Proc($I',J$)
\ElsIf{Case = 1 and $I<J$} \State Receive $\Mb{A}{II}$ from
Proc($I,I$) and $\Mb{A}{JJ}$ from Proc($J,J$)
\vskip 3pt
\State $\begin{pmatrix} \Mb{A}{II} & \Mb{A}{IJ} \\
\Mb{A}{IJ}^{T} & \Mb{A}{JJ} \end{pmatrix}= \M{V}\M{D}\M{V}^T$ 
\Comment{(Approximately) solve $2b \times 2b$ symmetric eigenproblem}
\vskip 3pt
\State $\M{V} = \begin{pmatrix} \M{V}_{11} & \M{V}_{12} \\ \M{V}_{21} & \M{V}_{22} \end{pmatrix}; \; \; \; \M{D} = \begin{pmatrix} \M{D}_{11} & \M{D}_{12} \\ \M{D}_{21} & \M{D}_{22} \end{pmatrix}$ \Comment{Break $\M{V}$ and $\M{D}$ into $b \times b$ subblocks}
\vskip 3pt
\State Send $[\Mb{V}{11};\Mb{V}{21}]$
and $\Mb{D}{11}$ to Proc($I,I$) \State Send $[\Mb{V}{12};\Mb{V}{22}]$
and $\Mb{D}{22}$ to Proc($J,J$) \State Send $[\Mb{V}{11};\Mb{V}{21}]$,
$\Mb{V}{22}$, and $\Mb{D}{21}$ to Proc($J,I$) \State $\Mb{A}{IJ}=\Mb{D}{12}$
\EndIf \Comment{Continued on next page $\rightarrow $} \algstore{bkbreak} \end{algorithmic} 
\end{algorithm}

\begin{algorithm}
\begin{algorithmic} \algrestore{bkbreak} \If{Case = 2} \State
Receive $[\Mb{V}{11};\Mb{V}{21}]$ and $\Mb{D}{11}$ from Proc($I,J'$)
\State $\Mb{A}{II}=\Mb{D}{11}$  \State Broadcast
$[\Mb{V}{11};\Mb{V}{21}]$ to Proc($I,K$) for $K\neq I, J'$ \Comment{Send to upper processor row}  \State Broadcast $[\Mb{V}{11};\Mb{V}{21}]$ to Proc($K,I$)
for $K\neq I, J'$ \Comment{Send to left processor column}\ElsIf{Case = 3} \State Receive $[\Mb{V}{12};\Mb{V}{22}]$
and $\Mb{D}{22}$ from Proc($I',J$) \State $\Mb{A}{JJ}=\Mb{D}{22}$
 \State Broadcast $[\Mb{V}{12};\Mb{V}{22}]$ to
Proc($J,K$) for $K\neq J,I'$ \Comment{Send to lower processor row} \State Broadcast $[\Mb{V}{12};\Mb{V}{22}]$
to Proc($K,J$) for $K\neq J,I'$ \Comment{Send
to right processor column}\ElsIf{Case = 1 and $J<I$} \State
Receive $[\Mb{V}{11};\Mb{V}{21}]$, $\Mb{V}{22}$, and $\Mb{D}{21}$
from Proc($J,I$) \State $\Mb{A}{IJ}=\Mb{D}{21}$ \EndIf \If{Case = 4} \Comment{Exchange block columns of $\M{A}$ and update local
block of $\M{A}$} \State Receive $[\Mb{V}{11},\Mb{V}{21}]$ from
Proc($J,J$) \State Send $\Mb{A}{IJ}$ to Proc($I,J'$) \State Receive
$\Mb{A}{IJ'}$ from Proc($I,J'$) \State $\Mb{A}{IJ}=\Mb{A}{IJ}\cdot\Mb{V}{11}+\Mb{A}{IJ'}\cdot\Mb{V}{21}$
\ElsIf{Case = 5}  \Comment{Exchange block columns of
$\M{A}$ and update local block of $\M{A}$} \State Receive $[\Mb{V}{12},\Mb{V}{22}]$
from Proc($J,J$) \State Receive $\Mb{A}{II'}$ from Proc($I,I'$)
\State Send $\Mb{A}{IJ}$ to Proc($I,I'$) \State $\Mb{A}{IJ}=\Mb{A}{II'}\cdot\Mb{V}{12}+\Mb{A}{IJ}\cdot\Mb{V}{22}$
\ElsIf{Case = 6} \Comment{Exchange block rows of $\M{A}$ and
update local block of $\M{A}$} \State Receive $[\Mb{V}{11},\Mb{V}{21}]$
from Proc($I,I$) \State Send $\Mb{A}{IJ}$ to Proc($J',J$) \State
Receive $\Mb{A}{J'J}$ from Proc($J',J$) \State $\Mb{A}{IJ}=\Mb{V}{11}^{T}\cdot\Mb{A}{IJ}+\Mb{V}{21}^{T}\cdot\Mb{A}{J'J}$
\ElsIf{Case=7} \Comment{Exchange block rows of $\M{A}$ and
update local block of $\M{A}$} \State Receive $[\Mb{V}{12},\Mb{V}{22}]$
from Proc($I,I$) \State Receive $\Mb{A}{I'J}$ from Proc($I',J$)
\State Send $\Mb{A}{IJ}$ to Proc($I',J$) \State $\Mb{A}{IJ}=\Mb{V}{12}^{T}\cdot\Mb{A}{I'J}+\Mb{V}{22}^{T}\cdot\Mb{A}{IJ}$
\EndIf 
\If{Eigenvectors are requested}
\Comment{Update $\M{Q}$} \If{Case = 4 or Case = 2 or (Case = 1
and $I>J$)} \State Send $\Mb{Q}{IJ}$ to Proc($I,J'$)
\State Receive $\Mb{Q}{IJ'}$ from Proc($I,J'$) \Comment{Exchange columns of $\M{Q}$,
update local block of $\M{Q}$}\State $\Mb{Q}{IJ}=\Mb{Q}{IJ}\cdot\Mb{V}{11}+\Mb{Q}{IJ'}\cdot\Mb{V}{21}$
\ElsIf{Case = 5 or Case = 3 or (Case = 1 and $I<J$)} \State Receive $\Mb{Q}{II'}$ from Proc($I,I'$) \State
Send $\Mb{Q}{IJ}$ to Proc($I,I'$) \Comment{Exchange columns of $\M{Q}$, update local block
of $\M{Q}$} \State $\Mb{Q}{IJ}=\Mb{Q}{II'}\cdot\Mb{V}{12}+\Mb{Q}{IJ}\cdot\Mb{V}{22}$
\EndIf \EndIf \EndFor \EndFor \EndFunction \end{algorithmic} 
\end{algorithm}

\section{Optimal Parameter Selection in 2.5D Parallel Jacobi}\label{appendix: optimization}

Assume
\[
  2\le c\le P^{1/3}.
\]
The feasible parameter set from Theorem~\ref{thm: 2.5D_parallel_comp}  is
\begin{equation}
\mathcal D:=
\left\{(\mu,\eta)\in\R^2:
 c\le\mu\le\sqrt{\frac{P}{c}},\qquad
 \frac{2\mu}{\sqrt{cP}}\le\eta\le 1
\right\}.
\label{eq:D}
\end{equation}
The assumption \(c\ge2\) ensures that the interval for \(\eta\) is nonempty
throughout the stated range of \(\mu\).

For \((\mu,\eta)\in\mathcal D\), define
\begin{equation}
X(\mu,\eta):=
\left(\frac{\eta^2P^3}{\mu^2c}\right)^{1/4}
=\sqrt{\frac{\eta P^{3/2}}{\mu\sqrt c}},
\label{eq:X}
\end{equation}
and let the leading terms from the latency and bandwidth parts of Theorem~\ref{thm: 2.5D_parallel_comp}  be
\begin{align}
S(\mu,\eta)
&:=X(\mu,\eta)\log\!\left(\eta\mu\sqrt{cP}\right)
 +\frac{\sqrt{P/c}}{\mu}
  \log\!\left(\frac{\sqrt{P/c}}{\mu}\right),
\label{eq:S}\\
W(\mu,\eta)
&:=\frac{n^2}{X(\mu,\eta)}+\frac{n^2}{\sqrt{cP}}.
\label{eq:W}
\end{align}
The minimax communication criterion for a fixed hardware pair
\((\alpha,\beta)\) is
\begin{equation}
\mathcal M_{\alpha,\beta}(\mu,\eta)
:=\max\left\{\alpha S(\mu,\eta),\ \beta W(\mu,\eta)\right\}.
\label{eq:minimax-objective}
\end{equation}
This is a \emph{bottleneck} criterion.  It is distinct from the additive runtime
model \(\alpha S+\beta W+\gamma F\) in equation~(1) of the paper.

Theorem~\ref{thm: 2.5D_parallel_comp}  uses big-\(O\) notation, so it does not identify the constant factors
needed for an exact machine-level optimum.  The result below is exact for the
leading-term model \eqref{eq:S}--\eqref{eq:W}.  If calibration gives constants
\(C_S,C_W>0\), simply replace \(\alpha\) and \(\beta\) by
\(C_S\alpha\) and \(C_W\beta\), respectively.

\begin{theorem} 
\label{thm:minimax}
Let \(\alpha,\beta>0\).  The optimization problem
\begin{equation}
\min_{(\mu,\eta)\in\mathcal D}
\mathcal M_{\alpha,\beta}(\mu,\eta)
\label{eq:main-minimax}
\end{equation}
has a unique solution.  It can be obtained from a one-dimensional monotone
root problem, as follows.

Set
\begin{equation}
q:=\sqrt{\frac{P}{c}},\qquad
x_-:=\sqrt{\frac{2P}{c}},\qquad
x_+:=\left(\frac{P}{c}\right)^{3/4}.
\label{eq:q-x}
\end{equation}
Under the standing assumption $2\le c\le P^{1/3}$, the interval
\([x_-,x_+]\) is nonempty, implying $P\ge 8$ and hence $4c\le P.$ For \(x\in[x_-,x_+]\), define
\begin{equation}
m(x):=\min\left\{q,\frac{P^{3/2}}{\sqrt c\,x^2}\right\},
\qquad
I(x):=[c,m(x)],
\label{eq:m-I}
\end{equation}
For an interval $I=[a,b]$, we write $\operatorname{proj}_{I}(z):=\min\{b,\max\{a,z\}\}$ for the Euclidean projection of $z$ onto $I$,
and
\begin{equation}
s(x,\mu):=
 x\log\!\left(\frac{\mu^2cx^2}{P}\right)
 +\frac{q}{\mu}\log\!\left(\frac{q}{\mu}\right),
\qquad
w(x):=\frac{n^2}{x}+\frac{n^2}{\sqrt{cP}}.
\label{eq:s-w}
\end{equation}
Let $W_0$ denote the principal Lambert $W$-function. For $z>0$,
$W_0(z)$ is the unique nonnegative solution of
$W_0(z)e^{W_0(z)}=z$ \cite{goerg2011lambert}, and set
\begin{equation}
\mu_S(x):=
\proj_{I(x)}
\left(
 \frac{q}{2x}\Wfun(2ex)
\right),
\qquad
\overline S(x):=s\bigl(x,\mu_S(x)\bigr).
\label{eq:muS}
\end{equation}
Then \(\overline S\) is continuous and strictly increasing on
\([x_-,x_+]\), while \(w\) is continuous and strictly decreasing on that
interval.  Consequently, the unique minimizer \(x_{\alpha,\beta}\) is
\begin{equation}
x_{\alpha,\beta}=
\begin{cases}
x_-,
& \alpha\overline S(x_-)\ge\beta w(x_-),\\[1mm]
x_+,
& \alpha\overline S(x_+)\le\beta w(x_+),\\[1mm]
\text{the unique root in \((x_-,x_+)\) of}
\quad \alpha\overline S(x)=\beta w(x),
& \text{otherwise}.
\end{cases}
\label{eq:x-ab}
\end{equation}
The corresponding optimal theorem parameters are
\begin{equation}
\mu_{\alpha,\beta}:=\mu_S(x_{\alpha,\beta}),
\qquad
\eta_{\alpha,\beta}:=
\frac{\mu_{\alpha,\beta}\sqrt c\,x_{\alpha,\beta}^2}{P^{3/2}}.
\label{eq:mu-eta-ab}
\end{equation}
In particular, the optimizer depends on \((\alpha,\beta)\) only through the
ratio \(\beta/\alpha\).
\end{theorem}

\begin{proof}
Use \(x=X(\mu,\eta)\) as a new variable.  Equation~\eqref{eq:X} gives
\begin{equation}
\eta=\frac{\mu\sqrt c\,x^2}{P^{3/2}}.
\label{eq:eta-x}
\end{equation}
The lower constraint on \(\eta\) in \eqref{eq:D} is equivalent to
\(x\ge x_-\).  The upper constraint \(\eta\le1\) is equivalent to
\(\mu\le P^{3/2}/(\sqrt c\,x^2)\).  Combining this with
\(c\le\mu\le q\) shows that the feasible pairs are equivalently
\begin{equation}
 x\in[x_-,x_+],\qquad \mu\in I(x),
\label{eq:feasible-xmu}
\end{equation}
where \(x_+\) is obtained by setting \((\mu,\eta)=(c,1)\).  Under this
change of variables,
\[
\eta\mu\sqrt{cP}=\frac{\mu^2cx^2}{P},
\]
so the functions in \eqref{eq:S}--\eqref{eq:W} become \(s(x,\mu)\) and
\(w(x)\) in \eqref{eq:s-w}.

For fixed \(x\), the bandwidth term \(w(x)\) is independent of \(\mu\).
Thus, it is sufficient to minimize \(s(x,\mu)\) over \(I(x)\).  Direct
calculation gives
\begin{equation}
\frac{\partial s}{\partial\mu}(x,\mu)
=
\frac{2x\mu-q\bigl(1+\log(q/\mu)\bigr)}{\mu^2}.
\label{eq:s-mu-derivative}
\end{equation}
The numerator in \eqref{eq:s-mu-derivative} is strictly increasing in
\(\mu\), because its derivative is \(2x+q/\mu>0\).  Its unique zero is
obtained from
\[
2x\mu=q\left(1+\log\!\left(\frac{q}{\mu}\right)\right).
\]
Writing \(r=2x\mu/q\) yields \(re^r=2ex\); hence
\(\mu=q\Wfun(2ex)/(2x)\).  Projecting this unique unconstrained minimizer
onto \(I(x)\) proves \eqref{eq:muS}.

For every feasible \(\mu\),
\begin{equation}
\frac{\partial s}{\partial x}(x,\mu)
=
\log\!\left(\frac{\mu^2cx^2}{P}\right)+2
=\log\!\left(\eta\mu\sqrt{cP}\right)+2>0,
\label{eq:s-x-derivative}
\end{equation}
because \(\eta\mu\sqrt{cP}\ge2\mu^2\ge2\).  Also, \(I(x_2)\subseteq
I(x_1)\) whenever \(x_1<x_2\).  Therefore, for \(x_1<x_2\),
\[
\overline S(x_2)
=\min_{\mu\in I(x_2)}s(x_2,\mu)
>\min_{\mu\in I(x_2)}s(x_1,\mu)
\ge\min_{\mu\in I(x_1)}s(x_1,\mu)
=\overline S(x_1).
\]
Thus \(\overline S\) is strictly increasing; its continuity follows from
\eqref{eq:muS}.  On the other hand,
\(w'(x)=-n^2/x^2<0\).

The original objective is now
\[
\min_{x\in[x_-,x_+]}
\max\{\alpha\overline S(x),\beta w(x)\}.
\]
The first argument of the maximum is strictly increasing and the second is
strictly decreasing.  Therefore the minimum is attained at the left endpoint,
the right endpoint, or their unique crossing, precisely as stated in
\eqref{eq:x-ab}.  Equations~\eqref{eq:muS} and \eqref{eq:eta-x} then yield
\eqref{eq:mu-eta-ab}.  Multiplying both \(\alpha\) and \(\beta\) by the same
positive constant leaves the crossing equation unchanged, so only
\(\beta/\alpha\) matters.
\end{proof}

\begin{corollary} 
\label{cor:recovery} The following statements follow

\begin{enumerate}

\item \emph{Latency only.}  Fix any admissible \(\mu\) and set \(\beta=0\).
Then
\begin{equation}
\argmin_{\eta:\,(\mu,\eta)\in\mathcal D}S(\mu,\eta)
=\left\{\frac{2\mu}{\sqrt{cP}}\right\},
\qquad
X=\sqrt{\frac{2P}{c}}.
\label{eq:lat-only}
\end{equation}
This is exactly the parameter selection used to obtain equation~(23).
If, in addition, one requires the leading arithmetic term to satisfy
\[
\frac{n^3}{X^2}+\frac{n^3}{P}\le\frac{2n^3}{P},
\]
then \(X^2\ge P\), or equivalently
\begin{equation}
\eta\ge\mu\sqrt{\frac{c}{P}}.
\label{eq:arith-constraint}
\end{equation}
Since latency increases with \(\eta\), equality in \eqref{eq:arith-constraint}
is optimal.  Hence
\begin{equation}
\eta=\mu\sqrt{\frac{c}{P}},
\qquad X=\sqrt P,
\label{eq:arith-preserving}
\end{equation}
which gives equation~(24).

\smallskip
\noindent\emph{SVD specialization.}
For the SVD bound in Theorem~\ref{thm: 2.5D_SVD_comp}, the leading
arithmetic term has optimal order when
\[
\frac{\sqrt{mn^5}}{X^2}=O\!\left(\frac{mn^2}{P}\right),
\]
which is equivalent to
\[
X^2=\Omega\!\left(P\sqrt{\frac{n}{m}}\right),
\qquad
\eta=\Omega\!\left(\mu\sqrt{\frac{nc}{mP}}\right).
\]
The uniform square-case choice
\(\eta=\Theta\!\left(\mu\sqrt{c/P}\right)\) gives
\(X=\Theta(\sqrt P)\), so it preserves optimal arithmetic for every
\(m\ge n\); it is also order-minimal under this arithmetic constraint when
\(m=\Theta(n)\).  This is the choice used in
\eqref{eqn: reasonable_latency_optimal}.

\item \emph{Bandwidth only.}  Fix any admissible \(\mu\) and set \(\alpha=0\).
Then
\begin{equation}
\argmin_{\eta:\,(\mu,\eta)\in\mathcal D}W(\mu,\eta)=\{1\}.
\label{eq:bw-only-fixed-mu}
\end{equation}
 With
\(\eta=1\), the condition
\begin{equation}
W(\mu,1)\le\frac{2n^2}{\sqrt{cP}}
\label{eq:bw-order}
\end{equation}
is equivalent to
\begin{equation}
\mu\le\sqrt{\frac{P}{c^3}}.
\label{eq:mu-bound}
\end{equation}
Choosing equality in \eqref{eq:mu-bound} gives
\begin{equation}
\mu=\sqrt{\frac{P}{c^3}},
\qquad \eta=1,
\qquad X=\sqrt{cP},
\label{eq:paper-bw-choice}
\end{equation}
and therefore
\[
S=\sqrt{cP}\log\!\left(\frac{P}{c}\right)+c\log c,
\qquad
W=\frac{2n^2}{\sqrt{cP}},
\qquad
\frac{n^3}{X^2}+\frac{n^3}{P}
=\left(1+\frac1c\right)\frac{n^3}{P}.
\]
These are the leading terms in Corollary~1.
\end{enumerate}

If both \(\mu\) and \(\eta\) are allowed to vary in a \emph{literal}
bandwidth-only minimization, the minimizer is instead \((\mu,\eta)=(c,1)\).
Thus Corollary~1 should be understood as the paper's selected
\(\eta=1\), bandwidth-order-optimal representative, rather than as the
unique minimizer of the raw bandwidth leading term over all of \(\mathcal D\).
\end{corollary}

\begin{proof} 

For latency only, at fixed \(\mu\), \(X\) is proportional to \(\sqrt\eta\),
and differentiation gives
\[
\frac{\partial S}{\partial\eta}
=\frac{X}{2\eta}
\left[\log\!\left(\eta\mu\sqrt{cP}\right)+2\right]>0.
\]
Therefore \(\eta\) is minimized at its lower feasible endpoint, proving
\eqref{eq:lat-only}.  The arithmetic constraint is equivalent to
\(X^2\ge P\), which is equivalent to \eqref{eq:arith-constraint} by
\eqref{eq:X}; the same monotonicity gives \eqref{eq:arith-preserving}.
For the SVD specialization,
\[
\frac{\sqrt{mn^5}}{X^2}=O\!\left(\frac{mn^2}{P}\right)
\quad\Longleftrightarrow\quad
X^2=\Omega\!\left(P\sqrt{\frac{n}{m}}\right).
\]
Substitution into \eqref{eq:X} yields
\(\eta=\Omega\!\left(\mu\sqrt{nc/(mP)}\right)\).
For fixed \(\mu\), the SVD latency term is likewise strictly increasing in
\(\eta\) on the feasible set.  The choice
\(\eta=\Theta\!\left(\mu\sqrt{c/P}\right)\) instead gives
\(X=\Theta(\sqrt P)\), and the two utilization orders agree when
\(m=\Theta(n)\).

For bandwidth only, \(W\) decreases with \(X\), and at fixed \(\mu\), \(X\)
increases with \(\eta\).  This proves \eqref{eq:bw-only-fixed-mu}.  When
\(\eta=1\), condition \eqref{eq:bw-order} is equivalent to
\(X\ge\sqrt{cP}\).  Using \(X^2=P^{3/2}/(\mu\sqrt c)\) gives
\eqref{eq:mu-bound}; substituting equality yields \eqref{eq:paper-bw-choice}
and the displayed formulas.  Finally, over all of \(\mathcal D\), \(X\) is
maximized by minimizing \(\mu\) and maximizing \(\eta\), namely by
\((\mu,\eta)=(c,1)\), proving the last statement.
\end{proof}

\begin{figure*}[t]
    \centering
    \includegraphics[width=0.95\textwidth]{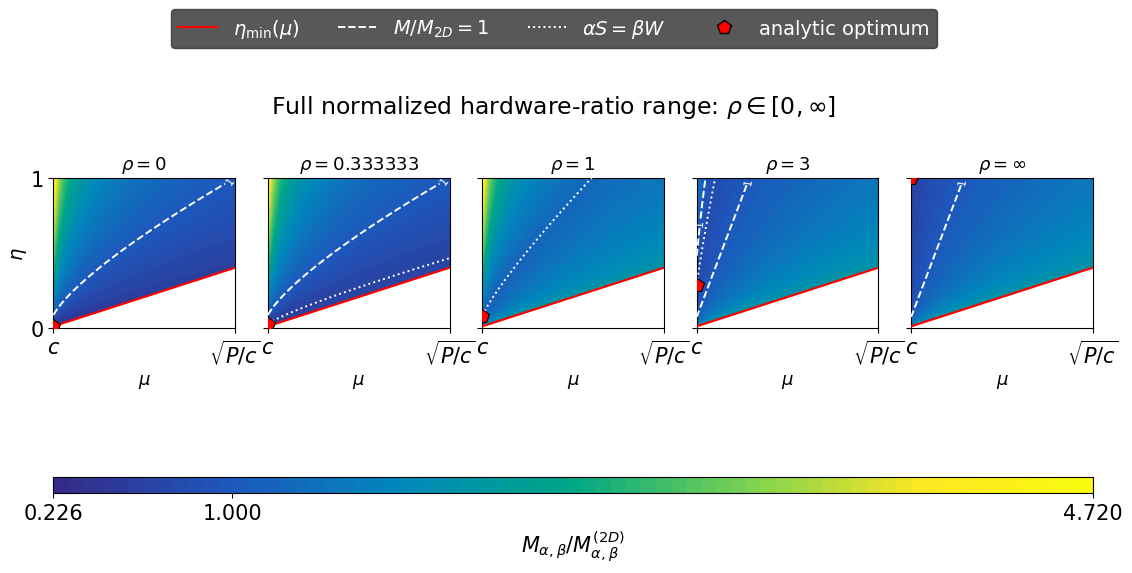}
    \includegraphics[width=0.95\textwidth]{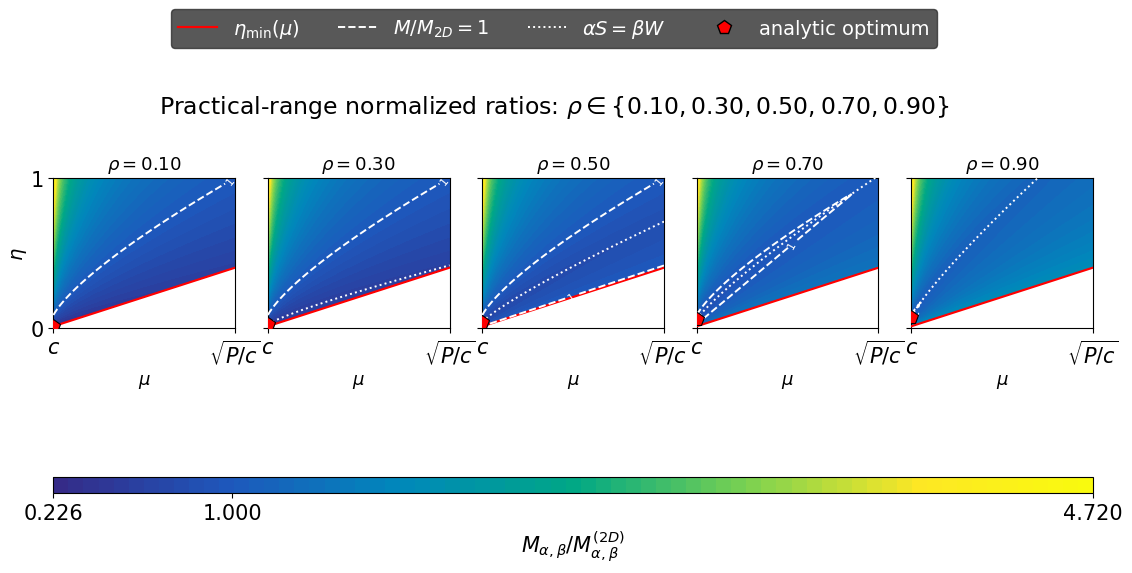}
    \caption{
    Optimization in Theorem \ref{thm:minimax}   of the latency and bandwidth bounds for 2.5D
    parallel Jacobi as functions of the batching parameter $\mu$ and the
    utilization factor $\eta$. 
    In all panels, $c=2$, $P=2^{18}$, $n=10*P$, $\alpha=1$, and
    $\beta=\rho\beta_0$, where
    $\beta_0:=S_{2D}/W_{2D}$. 
    Hidden constants in the complexity bounds are set to 5, and a
    $100\times100$ grid is used only to render the objective surface.
    The red line is the feasibility boundary
    $\eta_{\min}(\mu)=2\mu/\sqrt{cP}$, the white dashed contour marks    $M_{\alpha,\beta}/M_{\alpha,\beta}^{(2D)}=1$, and the white dotted
    contour marks the balance condition $\alpha S=\beta W$.}
    \label{fig:heatmap_1}
\end{figure*}

Figure \ref{fig:heatmap_1} shows that the optimal parameters vary continuously with the hardware ratio, interpolating between latency-dominated choices and the bandwidth-only limit. For small bandwidth weight, the optimizer can lie slightly inside the feasible region; as bandwidth becomes more important, the reduced optimum reaches the constraint $\mu=c$ and then follows that boundary while $\eta$ increases. This agrees with broader latency-bandwidth tradeoff: improving one communication cost necessarily worsens the other, while the selected bandwidth-optimal configuration is an order-optimal representative rather than the literal minimizer of the raw bandwidth term. In practice, these continuous optima should guide parameter selection rather than prescribe it: hidden constants, calibration, and rounding to admissible processor-grid and divisibility choices can shift the implemented optimum.

In particular,
$q=\sqrt{P/c}$, $x=X(\mu,\eta)$, and, for fixed $x$, the latency term is
minimized over the feasible interval $I(x)$ by
\[
 \mu_S(x)
 =
 \operatorname{proj}_{I(x)}
 \left(\mu_0(x)\right),
 \qquad
 \mu_0(x):=\frac{q}{2x}W_0(2ex).
\]
As shown in the proof of Theorem~\ref{thm:minimax}, this formula follows by setting the $\mu$-derivative of the reduced latency term to zero:
\[
 2x\mu=q\left(1+\log\frac{q}{\mu}\right).
\]
The left-hand side increases with $\mu$, while the right-hand side decreases; hence the stationary point is unique.  Moreover, $\mu_0(x)$ is strictly decreasing in $x$.  Thus the lower constraint becomes active once $\mu_0(x)$ reaches $c$, and it remains active for all larger $x$.

Let $x_c$ denote this switching value.  Substituting $\mu=c$ into the stationarity equation gives
\begin{equation}
 x_c
 =
 \frac{q}{2c}\left(1+\log\frac{q}{c}\right).
 \label{eq:mu-c-switch}
\end{equation}
Whenever $x_c$ lies in the reduced feasible interval of
Theorem~\ref{thm:minimax},
\begin{equation}
 \mu_S(x)=c
 \quad\Longleftrightarrow\quad
 x\geq x_c.
 \label{eq:mu-c-equivalence}
\end{equation}
Consequently, the occurrence of $\mu_\star=c$ is not a numerical artifact: it is the exact point at which the unconstrained latency-minimizing batching parameter leaves the feasible set.

To translate \eqref{eq:mu-c-equivalence} into a condition on hardware
parameters, define
\[
 S_{\mathrm{red}}(x):=s\bigl(x,\mu_S(x)\bigr),
 \qquad
 w(x):=\frac{n^2}{x}+\frac{n^2}{\sqrt{cP}}.
\]
Theorem~\ref{thm:minimax} shows that $S_{\mathrm{red}}$ is strictly
increasing and $w$ is strictly decreasing.  Therefore  optimizer is determined by the unique crossing of $\alpha S_{\mathrm{red}}(x)$ and $\beta w(x)$, except when the crossing is clipped to an endpoint.
 
Cache effects, message  aggregation, local eigensolver efficiency, and any dependence of the number of Jacobi sweeps on the block size are not represented in the asymptotic  model.  Accordingly, the threshold should be used to identify the expected regime and a small set of nearby admissible candidates, which should then be benchmarked on the target machine.

\end{document}